\documentclass[12pt,reqno]{amsart}
\usepackage{lmodern}
\usepackage{soul}

\usepackage{multirow}
\usepackage{makecell}
\usepackage{array}

\usepackage[utf8]{inputenc}
\usepackage[left=3.2cm,right=3.2cm,top=3cm,bottom=3cm]{geometry}
\usepackage{graphicx}
\usepackage[normalem]{ulem}
\usepackage{amsmath, amsfonts, amssymb, amsthm}
\usepackage{mathtools}
\usepackage[mathscr]{eucal}
\usepackage[colorlinks=true,citecolor=blue,linkcolor=blue]{hyperref}
\usepackage{color}
\usepackage{tikz}
\usepackage{tikz-cd}
\usepackage{xparse}
\usepackage{ifthen}
\usepackage{xifthen}
\usepackage{cleveref}
\usepackage{enumitem}
\usepackage{thmtools}
\usepackage{mathrsfs}
\usepackage{pdflscape}

\theoremstyle{plain}
\newtheorem{proposition}{Proposition}[section]
\newtheorem{theorem}[proposition]{Theorem}
\newtheorem{lemma}[proposition]{Lemma}
\newtheorem{corollary}[proposition]{Corollary}
\theoremstyle{remark}

\theoremstyle{definition}
\newtheorem{example}[proposition]{Example}
\newtheorem{definition}[proposition]{Definition}

\declaretheoremstyle[qed=$\triangle$]{myremark}
\declaretheorem[style=myremark, numberlike=proposition]{remark}
\numberwithin{equation}{section}
\newcommand{\ipic}[3]{\raisebox{#1\height}{\scalebox{#3}{\includegraphics{./#2}}}}
\allowdisplaybreaks

\newcommand{\inv}{^{-1}}
\newcommand{\tensor}{\otimes}
\NewDocumentCommand{\id}{O{}}{
	\ifthenelse{\equal{#1}{}}{\operatorname{id}}{\operatorname{id}_{#1}}
}
\newcommand{\placeholder}{-}
\NewDocumentCommand{\field}{ O{} }{
	\ifthenelse{ \equal{#1}{} }{\ensuremath{k}}{\ensuremath{\mathbb{#1}}}
}
\newcommand{\oversetEq}[1][ ]{
	\stackrel{\mathmakebox[\widthof{=}]{#1}}{=}
}

\newcommand{\cat}[1][C]{{\mathcal{#1}}}
\newcommand{\pivotalStruct}{\delta}
\newcommand{\tensUnit}{\mathbf{1}}

\newcommand{\To}{\Rightarrow}

\DeclareMathOperator{\End}{End}
\newcommand{\dualsymbol}{\vee}
\newcommand{\dualsymbolR}{\vee}
\newcommand{\dualL}[1]{{#1^{\dualsymbol}}}
\newcommand{\ddualL}[1]{{#1^{\dualsymbol\dualsymbol}}}
\NewDocumentCommand{\dualR}{ m O{-0.1} }{
	{\prescript{\dualsymbolR}{}{\hspace*{#2em}#1}}
}

\DeclareMathOperator{\ev}{ev}
\DeclareMathOperator{\coev}{coev}
\DeclareMathOperator{\evL}{\ev}
\DeclareMathOperator{\coevL}{\coev}
\DeclareMathOperator{\evR}{\widetilde{\ev}}
\DeclareMathOperator{\coevR}{\widetilde{\coev}}

\newcommand{\Vect}{\mathsf{Vect}}
\NewDocumentCommand{\hmodM}{O{H}O{}O{}O{}}
{ {\prescript{#4}{#1}{\mathcal{M}}_{#2}^{#3}} }

\newcommand{\hmod}[1][H]{#1\text{-mod}}
\newcommand{\sweedler}[1]{_{(#1)}}
\newcommand{\op}{^{\textup{op}}}
\newcommand{\cop}{^{\textup{cop}}}

\newcommand{\elQbold}[1]{\boldsymbol{#1}}
\newcommand{\Dt}{\elQbold{f}}
\newcommand{\modulus}{\boldsymbol{\gamma}}
\newcommand{\counit}{\varepsilon}
\newcommand{\coint}{\boldsymbol{\lambda}}
\newcommand{\oneQ}{\elQbold{1}}
\newcommand{\pivotQ}{\elQbold{g}}
\newcommand{\alphaQ}{\elQbold{\alpha}}
\newcommand{\betaQ}{\elQbold{\beta}}
\newcommand{\pR}{p^R}
\newcommand{\qR}{q^R}
\newcommand{\pL}{p^L}
\newcommand{\qL}{q^L}

\newcommand\coassQ{\Phi}
\newcommand\invCoassQ{\Psi}

\newcommand{\braiding}{c}
\newcommand{\coend}{\mathcal{L}}
\newcommand{\cocoend}{\mathcal{E}}
\newcommand{\intLyu}{{\Lambda}}

\newcommand{\rMatrix}{\textup{\textsf{R}}}

\newcommand{\monodromy}{\textup{\textsf{M}}}
\newcommand{\slTwoZ}{SL(2,\mathbb{Z})}
\newcommand{\elRibbon}{\boldsymbol{v}}

\newcommand{\modS}{\mathcal{S}}
\newcommand{\modT}{\mathcal{T}}
\newcommand{\hopfPair}{\omega_{\coend}}
\newcommand{\ribTwist}{\vartheta}
\newcommand{\dinatLyu}{j}

\newcommand{\symRightCoint}{\widehat\coint{}^r}
\newcommand{\symLeftCoint}{\widehat\coint{}^l}
\newcommand{\symRightCointA}{\widehat\coint{}^{A,r}}
\newcommand{\symLeftCointA}{\widehat\coint{}^{A,l}}
\newcommand{\eQ}{\elQbold{e}}

\newcommand{\sympFerm}{\mathsf{Q}}

\newcommand{\tr}{\operatorname{tr}}
\newcommand{\trCat}[1][{\cat}]{\tr^{#1}}
\newcommand{\modTr}{\textup{\textsf{t}}}

\newcommand{\RT}[1][\cat]{\mathcal{Z}^{RT}_{#1}}
\newcommand{\dimredRT}[1][\cat]{\RT[\cat]_2}
\newcommand{\grothring}[1][\cat]{\operatorname{Gr}(#1)}

\newcommand{\DD}{\mathscr{D}}
\newcommand{\fusioncoeff}[2]{N_{#1}^{#2}}

\newcommand{\vvec}[1]{\underline{#1}}

\newcommand{\Irr}[1][\cat]{\operatorname{Irr}(#1)}
\newcommand{\deltaEven}[1]{\delta_{#1,\textup{even}}}
\newcommand{\deltaOdd}[1]{\delta_{#1,\textup{odd}}}
\newcommand{\qandq}{\quad\text{and}\quad}
\newcommand{\ProjIdeal}[1][{\cat[C]}]{\operatorname{Proj}({#1})}

\newcommand{\genK}{\mathsf{K}}
\newcommand{\genF}{\mathsf{f}}
\newcommand{\ribbon}{\boldsymbol{v}}
\newcommand{\drinfeldElement}{\boldsymbol{u}}

\newcommand{\grMultiplicity}[2]{[#1:#2]}

\newcommand{\Dinat}{\operatorname{Dinat}}

\usetikzlibrary{%
	arrows,
	arrows.meta,
	calc,
	graphs,
	decorations.pathmorphing,
	decorations.markings,
	external
}

\crefname{equation}{equation}{equations}
\crefname{enumi}{point}{points}
\Crefname{enumi}{Point}{Points}

\newcommand{\injHullMorphism}[1]{\nu_{#1}}

\newcommand{\generatedIdeal}[1]{(#1)}
\newcommand{\projCoverMorphism}[1]{p_{#1}}
\newcommand{\projCover}[1]{P_{#1}}
\newcommand{\tensIdeal}[1][I]{\cat[#1]}
\newcommand{\RTfunctor}[1][C]{F_{\cat[#1]}}
\newcommand{\renRTinvariant}[1][{\modTr}]{ F^\prime_{\cat,\, #1} }
\newcommand{\CategoryColoredRibGraph}{\operatorname{Rib}_{\cat}}

\newcommand{\qtextq}[1]{\quad\text{#1}\quad}

\newcommand{\pullbackTrace}[1][\modTr]{ \restrictionFunctor^\ast (#1) }
\newcommand{\projCoverTensorUnitSF}[1][A]{P_0^+(#1)}
\newcommand{\relativeProjectives}[2]{#1\operatorname{-Proj} (\hmodM[{#2}]) }
\newcommand{\restrictionFunctor}{\operatorname{Res}}
\newcommand{\restrictionFunctorExplicit}[2]{\restrictionFunctor^{#1}_{#2}}

\newcommand{\grothringclass}[1]{[#1]}
\newcommand{\dimCat}[1][{\cat}]{\dim_{\cat}}
\newcommand{\framedUnknot}[2]{O(#1, #2)}

\newcommand{\torusknot}[2]{T(#1, #2)}
\newcommand{\hopflink}[2]{S(#1; #2)}

\newcommand{\anomaly}{\delta}
\newcommand{\stabCoeff}{\Delta_\pm}
\newcommand{\stabCoeffP}{\Delta_+}
\newcommand{\stabCoeffM}{\Delta_-}
\newcommand{\modularityParameter}{\zeta}

\newcommand{\multL}{m_\coend}
\newcommand{\comultL}{\Delta_\coend}
\newcommand{\unitL}{\eta_{\coend}}
\newcommand{\counitL}{\counit_{\coend}}
\newcommand{\cointLyu}{ {\intLyu^{\textup{co}}} }

\newcommand{\LRTfunctor}{F_{\intLyu}}
\newcommand{\renLRTinvariant}[1][{\modTr}]{F'_{\intLyu,\, #1}}
\newcommand{\invLyu}{\operatorname{Lyu}_{\intLyu}}
\newcommand{\invDGGPRtrace}[1]{L'_{\intLyu,\, #1}}
\newcommand{\invDGGPR}{\invDGGPRtrace{\modTr}}

\newcommand{\CategoryBichromeGraph}{\operatorname{BiRib}_{\cat}}

\newcommand{\rightact}{\, \resizebox{2.8mm}{!}{$\bigcirc\hspace*{-5mm}<$}~}
\newcommand{\phiQ}[1]{\elQbold{\phi}_{#1}}

\newcommand{\TOP}{{\resizebox{7mm}{!}{$\operatorname{TOP}$}}}

\newcommand{\parMat}{\mu}

\makeatletter
\newcommand*{\updatelabelname}[1]{%
	\xdef\@currentlabelname{#1}%
}
\makeatother

\usepackage{caption}
\usepackage{subcaption}
\usepackage{afterpage}
\setenumerate{font=\upshape}

\title[Non-semisimple link and manifold invariants]{Non-semisimple link and manifold invariants\\
for symplectic fermions}
\date{}

\begin{document}
\tikzstyle{my line} = [smooth, line width=0.25mm]

  \setcounter{tocdepth}{2}

  \maketitle
  \thispagestyle{empty}

\begin{center} 
    Johannes Berger\,$^{a}$, Azat M.\ Gainutdinov\,$^{b}$~~and~~Ingo
    Runkel\,$^c$~~\footnotemark
    \\[1.5em]
	{\sl\small $^a$ D\'epartement de Math\'ematique, Universit\'e Libre de Bruxelles,\\
	Campus de la Plaine, Boulevard du Triomphe, 1050 Bruxelles, Belgium}
\\[0.5em]
    {\sl\small $^b$ Institut Denis Poisson, CNRS, Universit\'e de Tours,\\ Parc de Grandmont, 37200 Tours, France}
\\[0.5em]
	{\sl\small $^c$ Fachbereich Mathematik, Universit\"at Hamburg\\
	Bundesstra\ss e 55, 20146 Hamburg, Germany}
\end{center}
\footnotetext{\raggedright Emails: 
			{\tt johannes.berger@ulb.be},
			{\tt azat.gainutdinov@cnrs.fr}, 
			{\tt ingo.runkel@uni-hamburg.de}
        }

\begin{abstract}
We consider the link and three-manifold invariants 
in~\cite{DGGPR}, which are defined
in terms of certain non-semisimple 
    ribbon categories~$\cat$ (locally-finite for link invariants, finite for manifold invariants)
together with a choice of tensor ideal and modified trace.
If the ideal is all of~$\cat$, these invariants agree with those defined by Lyubashenko in the 90's. 
We show that in that case the invariants
depend on the objects labelling the link only through their simple composition factors.
In order to detect non-trivial extensions one needs to pass to proper ideals. 

We compute examples of link and three-manifold invariants for $\cat$ being the category of $N$ pairs of symplectic fermions.
Using a quasi-Hopf algebra realisation of~$\cat$,
we find that the Lyubashenko-invariant of a lens space is equal to the order of its first homology group to the power $N$, a relation we conjecture to hold for all three-manifolds with finite first homology group.
For $N \ge 2$, $\cat$ allows for tensor ideals~$\tensIdeal$ with a modified trace which are different from all of $\cat$ and from the projective ideal.
Using the theory of pull-back traces and symmetrised cointegrals,
we show that the link invariant obtained from $\tensIdeal$ can distinguish a 
continuum of indecomposable but reducible objects which all have the same composition series.
\end{abstract}

\newpage
\tableofcontents

\newpage 

\section{Introduction}

While a large body of research on semisimple Reshetikhin-Turaev type invariants 
of three-manifolds
has been accumulated since their inception in the 90ies, much less is known about quantum invariants obtained from non-semisimple input data. After initial successes of obtaining three-manifold invariants from non-semisimple Hopf algebras by Hennings \cite{H96}, and from modular tensor categories which need not be semisimple by Lyubashenko \cite{Lyu-inv-MCG}, the next significant step forward came only with the introduction of modified traces on tensor ideals and using these to define link invariants by Geer, Kujawa, Patureau-Mirand, and Turaev
\cite{GPT,GKP_generalized-trace-and-mod-dim-rib-cat}. 
This led to the definition of three-manifold invariants 
by Costantino, Geer, and Patureau-Mirand \cite{CGP},
and three-dimensional topological field theories via non-semisimple modular tensor categories and variants thereof \cite{Blanchet:2014ova,DGP-renormalized-Hennings,DGGPR}. The non-semisimple topological field theories of \cite{DGGPR} recover the original Reshetikhin-Turaev theory in the semisimple case, as well as the mapping class group actions found by Lyubashenko in the non-semisimple case \cite{DGGPR2}.

\medskip

Our focus here is on non-semisimple link and three-manifold invariants, 
and not on topological field theory. 
The purpose of the present paper is two-fold:
\begin{enumerate}
	\item[(a)] We exhibit a number of general properties of the non-semisimple invariants defined in \cite{DGGPR} which set them apart from the semisimple case, and from Lyubashenko's invariant. 
	
	\item[(b)] We take one example of a modular tensor category -- 
 that of symplectic fermions \cite{Runkel:2012cf,FGR2} -- as input datum and provide explicit computations of some basic invariants. We hope this example based exposition can be useful as a guide to the rather technical definition of the invariants in \cite{DGGPR}.
\end{enumerate}	
If we should single out one property as particularly interesting then it would be that 
\textit{non-semisimple invariants are able to distinguish a continuum of indecomposable but reducible objects which all have the same composition series.}
In other words, these invariants are able to detect extensions and not just the composition series, in contrast to the Lyubashenko invariant.

\medskip

This paper splits into two parts of roughly equal length.
The first part consists of Sections~\ref{sec:link_invariants_from_ideals} and \ref{sec:invariants_of_manifolds} and contains the result for point (a) and summarises the findings of the explicit computations in (b). The second part, Sections \ref{sec:qHopf}--\ref{sec:lens_space_symp_ferm}, provides all the technical and lengthy details that went into the results in (b).

We will now give a quick overview of the contents of the first part of the paper. 
Throughout this introduction we fix 
\begin{align*}
	\cat ~:~ \text{a locally finite ribbon tensor category over an alg.\ closed field $k$} \ .
\end{align*}
In particular, $\cat$ is rigid and abelian, each object has finite length, Hom-spaces are finite-dimensional, and the tensor unit is simple. The ribbon structure endows $\cat$ with a pivotal structure and a braiding, which may be degenerate.

\subsection{Invariants of ribbon graphs}

In Sections~\ref{sec:link_invariants_from_ideals} we focus on invariants of ribbon graphs in $\field[R]^3$ (or, equivalently, in $S^3$). These depend on the choice of a tensor ideal $\tensIdeal$ in $\cat$ and on a modified trace on $\tensIdeal$. Let us describe some properties of these ingredients and the resulting invariants in turn.

\subsubsection*{Tensor ideals}

A tensor ideal $\tensIdeal \subset \cat$ is a full subcategory that is closed
under tensoring with arbitrary objects in $\cat$ and under taking direct summands in the sense that if $X \oplus Y \in \tensIdeal$, then also $X,Y \in \tensIdeal$. 
The two canonical examples are $\tensIdeal = \cat$ and $\tensIdeal = \ProjIdeal$, the full subcategory of projective objects in $\cat$. In fact,
any non-zero tensor ideal $\tensIdeal$ satisfies
\begin{align*}
	\ProjIdeal \subset \tensIdeal \subset \cat \ .
\end{align*}
We note that $\ProjIdeal$ can be zero for $\cat$ non-semi\-simple (Example~\ref{ex:no-ideals}). But if $\ProjIdeal \neq 0$, then $\cat$ admits a non-zero proper ideal if and only if $\cat$ is not semisimple (\Cref{prop:proj_is_smallest}).
Thus the freedom to choose a tensor ideal is present only for $\cat$ non-semisimple.

In this work, the main sources of tensor ideals $\tensIdeal$ that are not $\cat$ or $\ProjIdeal$, and that we call \textit{intermediate}, are given by subcategories  of all objects in $\cat$ with the property that their image under a monoidal functor $\cat \to \mathcal{D}$ is in a tensor ideal of~$\mathcal{D}$, e.g.\ in $\mathrm{Proj}(\mathcal{D})$.
We refer to such ideals as \textit{pullback ideals}.

\subsubsection*{Modified traces}

A modified trace on a tensor ideal $\tensIdeal$ is a family of
 linear maps
\begin{align*}
	\modTr =
    \left\{ \modTr_X \colon \End_{\cat}(X) \to \field \right\}_{X \in \tensIdeal}
\end{align*}
which induce symmetric pairings $\cat(X,Y) \times \cat(Y,X) \to k$ for $X,Y \in \tensIdeal$, and which satisfy a partial trace compatibility condition with the categorical trace on~$\cat$, 
see~\cite{GKP_generalized-trace-and-mod-dim-rib-cat}.
In fact, the categorical trace $\trCat_X(f)$, $f \in \End_{\cat}(X)$ of $\cat$ -- defined in terms of the duality maps of~$\cat$ -- 
    is the unique-up-to-scalar modified trace on all of~$\cat$. 
The usefulness of modified traces derives from the observation that while the categorical trace is identically zero on any proper tensor ideal (\Cref{prop:categorical_trace_vanishes_on_all_proper_ideals}), it may be possible to find non-zero modified traces on a given tensor ideal $\tensIdeal$. 

An important way to obtain modified traces on intermediate tensor ideals in $\cat$ 
is via pullback of the tensor ideal $\ProjIdeal[\mathcal{D}]\subset \mathcal{D}$ along a pivotal monoidal functor $\cat \to \mathcal{D}$, see \cite{FOG} and \Cref{prop:pullback_ideal}.
Indeed, suppose $\mathcal{D}$ has enough projective objects and is unimodular (i.e.\ the projective cover $P_\tensUnit \to \tensUnit$ of the tensor unit also has socle $\tensUnit$). 
Then there is a unique-up-to-scalar modified trace on $\ProjIdeal[\mathcal{D}]$ and for each non-zero such trace the above symmetric pairing is non-degenerate, see
    \cite[Thm.\,1]{BBGa} for the case of Hopf algebras
    and \cite[Cor.\,5.6]{GKP_m-traces} for the general case.

\subsubsection*{Invariants of $\tensIdeal$-admissible ribbon graphs}

Let $\tensIdeal \subset \cat$ be a tensor ideal and $\modTr_X$ a modified trace on $\tensIdeal$. 
A $\cat$-coloured ribbon graph $T$ in $\field[R]^3$ is called 
$\tensIdeal$-admissible if at least one edge is coloured by an object 
$X \in \tensIdeal$. For such a $T$ one obtains an invariant $\renRTinvariant(T)$ of $T$ by 
cutting an $\tensIdeal$-labelled edge and taking the modified trace
(see \cite{GPT,DGGPR} and \Cref{prop:renormalized_link_invariant}). In particular, this prescription is independent of which $\tensIdeal$-labelled edge is chosen in the cutting procedure. 
If $T$ is a framed link, the invariant $\renRTinvariant(T)$ depends on objects colouring components not used in the cutting procedure only through their class in the Grothendieck ring of~$\cat$ (\Cref{prop:renormalized_link_invariant_factors_through_Gr}). If $\tensIdeal = \cat$, this also holds for the component used in the cutting procedure. Thus, in order to detect any structure of the objects colouring the link other than their composition series, one needs to work with a proper ideal $\tensIdeal \subsetneq \cat$. This illustrates once more that tensor ideals and modified traces are crucial for non-semisimple invariants.

\subsubsection*{Symplectic fermion example}

To see how the non-semisimple invariants 
$\renRTinvariant(T)$
behave in a concrete example, we consider the category $\hmodM[\sympFerm]$ of finite-dimen\-sion\-al left 
modules over $\sympFerm = \sympFerm(N, \beta)$, the
\emph{symplectic fermion} quasi-Hopf algebra.
As an algebra it is the semidirect product of $\field[Z]_2$ with the direct sum of a Gra\ss mann algebra and a Clifford algebra in $2N$ generators each. 
The parameter $\beta  \in \field[C]$ satisfies $\beta^4 = (-1)^N$. 

    Symplectic fermions were introduced in the context of two-dimensional conformal field theory \cite{Kausch:1995py,Gaberdiel:1996np} and the vertex operator algebra was studied in \cite{Abe:2005}. A corresponding 
    modular 
tensor category was proposed in \cite{Davydov:2012xg,Runkel:2012cf} and was expressed as representations of a 
	factorisable
ribbon quasi-Hopf algebra in \cite{Gainutdinov:2015lja,FGR2}. In \cite{GN,Creutzig:2021cpu} it was shown that for $N=1$,  $\hmodM[\sympFerm]$ is indeed ribbon-equivalent to the representation category of the symplectic fermion vertex operator algebra, see also \cite{McRae} for recent progress on general $N$.
    
The categories $\hmodM[\sympFerm]$ 
are comparatively easy to handle. For example, they always have four simple objects, but the maximal 
composition length of indecomposable projective modules
increases with $N$ and is given by $2^{2N}$. The category $\hmodM[\sympFerm]$ is a finite ribbon tensor category and is modular, i.e.\ the braiding satisfies a non-degeneracy condition.

The symplectic fermion quasi-Hopf algebra $\sympFerm$ can be understood as a version of the quantum group
    of type $B_N$ 
at $q=i$ \cite{Flandoli:2017xwu}. Note that the parameter is the rank while the root of unity is kept fixed. In the other class of standard examples, namely small quantum groups for $sl(2)$, as in~\cite{Lyu-inv-MCG, FGST, Creutzig:2017khq},
the parameter is the order of the root of unity. These two classes of examples probe different aspects of non-semisimple invariants: For $B_N$ the number of simple objects is always $4$ and the maximal composition length 
   of their projective covers
increases, while for $sl(2)$ the maximal composition length
of projective covers
is~$4$, but the number of simple objects increases.

\medskip 

For $\hmodM[\sympFerm]$ we compute invariants of the unknot 
    and the Hopf link with different framings,
and of torus knots for the tensor ideals $\hmodM[\sympFerm]$, $\ProjIdeal[{\hmodM[\sympFerm]}]$ and, importantly, also for a choice of intermediate ideal $\tensIdeal$ (in the case $N=2$). 
The ideal $\tensIdeal$ and a modified trace $\modTr^{\tensIdeal}$ on it are obtained via the pullback construction
    specialised to quasi-Hopf algebras, see Section~\ref{sec:pullback-qHopf}.

The results are listed in \Cref{table:invariants} from Section~\ref{sec:examples_link_invariants}. 
Most interestingly, the intermediate ideal we consider contains a continuum of mutually non-isomorphic indecomposable objects which all have the same composition series made out of 4 simple objects. 
The family we consider can be conveniently parameterised by a complex $2 \times 2$-matrix
$\parMat \in \operatorname{Mat}_2(\field[C])$,
and we denote the corresponding representation of $\sympFerm$ by $P_\parMat \in \tensIdeal$.
For a ribbon graph $T$ coloured by such $P_\parMat$, it turns out that the corresponding invariants  $\renRTinvariant[{\modTr^{\tensIdeal}}](T)$
can depend continuously on~$\parMat$. For example the invariant of $O(n,P_\parMat)$, the $n$-framed unknot  coloured by~$P_\parMat$, is
\begin{align*}
\renRTinvariant[{\modTr^{\tensIdeal}}]\bigl(O(n,P_\parMat)\bigr) = 2n(1+\det(\parMat))~.
\end{align*}
Such a dependence on continuous parameters can never happen  
for the extremal choices $\tensIdeal=\cat$ and $\tensIdeal=\ProjIdeal$ for the tensor ideal in a finite tensor category, 
see \Cref{rem:no-contiua-for-C-and-ProjC} for more details.

Let us also mention that in some cases, the invariants based on the intermediate ideal $\tensIdeal$ are topologically stronger (i.e.\ able to distinguish more knots) than those based on the ideal $\ProjIdeal$, see a discussion in the last item (3) in  Section~\ref{sec:examples_link_invariants}.

\subsection{Invariants of three-manifolds}
In Section~\ref{sec:invariants_of_manifolds} we work with stronger assumptions:
\begin{align*}
	\cat \text{ is in addition finite, unimodular, and twist-nondegenrate} \ .
\end{align*}
Finiteness means that there are only finitely many isomorphism classes of simple objects and that $\cat$ has enough projective objects. Unimodularity was already mentioned above, and twist-nondegeneracy will be explained below. All of these conditions hold for modular tensor categories, see~\cite[Prop.\,2.6]{DGGPR}, and in particular for $\hmodM[\sympFerm]$.

Fix a tensor ideal $\tensIdeal \subset\cat$ and a modified trace $\modTr$ on $\tensIdeal$.
We consider invariants of (closed, oriented) three-manifolds with embedded $\tensIdeal$-admissible
ribbon graphs.
We equally consider Lyubashenko invariants of three-manifolds.
For the latter there are no restrictions on the embedded ribbon graph, in particular it is allowed to be empty.

\subsubsection*{Surgery invariants}
The invariants we study are obtained as surgery invariants, and we need to extend our definition of the invariants
$\renRTinvariant$
of	$\tensIdeal$-admissible
ribbon graphs to so-called bichrome graphs. These consist of an ordinary ribbon graph, called ``blue'' and the surgery link called ``red''. The red component does not carry further labels.
The invariant of bichrome graphs needs another ingredient
to evaluate the red part,
namely the universal Hopf algebra $\coend$ in $\cat$, defined as a coend, and the integral $\intLyu \colon \tensUnit \to \coend$, unique up to a scalar. We review the slightly intricate construction of the corresponding invariant $\renLRTinvariant(T)$ from \cite{DGP-renormalized-Hennings,DGGPR} in \Cref{sec:invariants_of_bichrome_graphs}.

Let $M$ be an oriented closed three-manifold and $T$ a (blue) $\tensIdeal$-admissible ribbon graph embedded in $M$. One can show that 
\begin{align*}
	\invDGGPR (M,T)
	= \DD^{-1-\ell(L)} \anomaly^{-\sigma(L)} \renLRTinvariant( L \cup T )
	\notag
\end{align*}
is a topological invariant of the pair $(M,T)$. Here, $L$ is a surgery link representing~$M$, and $\ell(L)$ and $\sigma(L)$ are the number of components and the signature of the surgery link $L$, respectively.
At this point one needs $\cat$ to be twist-nondegenerate, which means that $\renLRTinvariant(O_\pm) \neq 0$ for $O_\pm$ the red $\pm1$-framed unknot. This is necessary to achieve invariance under the Kirby 1-move (stabilisation).
See \cite{DGGPR} and \Cref{def:invariant_general} for details and for the definition of the constants $\DD$ and $\anomaly$.

The original non-semisimple invariant defined by Lyubashenko \cite{Lyu-inv-MCG} is recovered by taking $\tensIdeal = \cat$ and $\modTr_X$ to be the categorical trace $\trCat_X$: 
\begin{align*}
\invLyu(M,T) = \invDGGPRtrace{\trCat}(M,T)\ .
\end{align*}
The strong point of Lyubashenko's invariant is that it can be defined for just the three-manifold $M$ without an embedded ribbon graph (since one can always insert a loop labelled by $\tensUnit$). On the downside, $\invLyu(M,T)$ is often zero. For example, if $\cat$ is modular, then 
$$\invLyu(S^2 \times S^1,\emptyset) = 0$$ if and only if $\cat$ is non-semisimple (\Cref{prop:Lyubashenko_problems}), and $\invLyu(M,T) = 0$
for all $M$
if $T$ contains an edge labelled by an object in a proper tensor ideal (\Cref{cor:Lyu-zero-on-ideal}).

The price to pay to do better than $\invLyu(M,T)$ is that one needs to include an $\tensIdeal$-admissible ribbon graph in $M$. Furthermore, the ribbon graph must 
wrap non-trivially around the surgery link of $M$,
 or else the invariant reduces to $\invLyu(M,\emptyset)$ times the link invariant $\renRTinvariant(T)$ discussed above (\Cref{prop:relation_between_Lyu_and_rLyu}).

\subsubsection*{Lens space invariants in the symplectic fermion example}

Fix two positive coprime integers $p,q$.
The lens spaces $\mathfrak{L}(p, q)$ are by definition a quotient of $S^3 = SU(2)$ by a $q$-dependent action of $\field[Z]/p\field[Z]$. A surgery presentation can be obtained from a continued fraction expansion of $\frac{p}{q}$. We give a general expression for 
$\invDGGPR (\mathfrak{L}(p, q),T)$, where $T$ is a loop bounding a disc (thus reducing to Lyubashenko's invariant 
as given 
for $\mathfrak{L}(p, q)$ already in \cite{Lyu-inv-MCG}), and where $T$ wraps a non-trivial cycle, i.e.\ is linked non-trivially with $L$ (\Cref{prop:invariants_lens_spaces_general}).
In the symplectic fermion example $\cat = \hmodM[\sympFerm]$, $\sympFerm = \sympFerm(N, \beta)$ we find
\begin{align*}
	\invLyu(\mathfrak{L}(p, q)) = p^N \ .
\end{align*}
Since $|H_1(\mathfrak{L}(p, q))| = p$, this leads us to
\begin{align*}
\text{conjecture:} \quad
\invLyu(M) = \begin{cases} 0 &; H_1(M) \text{ infinite}
\\
|H_1(M)|^N &; \text{else}
\end{cases}
\end{align*}
for all closed 3-manifolds $M$.
This conjecture for symplectic fermions fits into a series of observations relating non-semisimple and semisimple invariants which we briefly recall in \Cref{sec:Examples:lens_spaces}, where we also discuss a possible generalisation to other modular tensor categories. 

Next we choose the projective ideal and $T$ the loop wrapping a (specific) non-trivial cycle coloured by the projective cover $P_\tensUnit$ of the tensor unit.  We get  
\begin{align*}
	\invDGGPR\bigl(\mathfrak{L}(p, q),T(P_\tensUnit)\bigr) = \tfrac12 c + \tfrac12 \beta^2 q^{N} ~,
\end{align*}
where $c \in \{0\} \cup \{ \beta^m | m\in \field[Z] \}$ is determined recursively from the continued fraction expansion (\Cref{prop:invariants_lens_spaces_symp_ferm}).
Note that this is no longer just an invariant of the lens space $\mathfrak{L}(p, q)$, but it depends on our choice of cycle wrapped by $T$. Indeed, $\mathfrak{L}(p, q) \cong \mathfrak{L}(p, q+p)$ as manifolds, but the above expression is not invariant under this substitution.

The most interesting case is again that of the intermediate ideal (again $N=2$). In that case we find for $T$ the non-trivial loop as above and labelled by $P_\parMat$, $\parMat \in \mathrm{Mat}_2(\field[C])$ (\Cref{thm:lens-for-intermediate}):
\begin{align*}
	\invDGGPRtrace{\modTr^{\tensIdeal}}\bigl(\mathfrak{L}(p, q),T(P_\parMat)\bigr) = -2pq(1+\det(\parMat))~.
\end{align*}

To the best of our knowledge, our results provide the first explicit computation of the behaviour of quantum invariants under variation of the choice of tensor ideal, including an intermediate tensor ideal.
It would also be interesting to calculate invariants associated to
ideals with modified traces 
    in pivotal tensor categories as
constructed in~\cite{GKP_m-traces}, but this is outside the scope of our paper.

\subsection{Detailed computations}

In \Cref{sec:qHopf}, we review in detail our conventions for
(pivotal,
quasi-triangular, ribbon)
quasi-Hopf algebras $H$.
Then we recall the 
construction of
modified traces on the projective ideal of the category $\hmodM[H]$ of finite-dimensional left $H$-modules.
This uses the theory~\cite{BGR1} of symmetrised (co)integrals of $H$
which is a quasi-Hopf generalisation of the construction in~\cite{BBGa}. 
Using the right integral of $H$, in Section~\ref{sec:Lyu-coend-int-for-qHopf} we calculate the integral $\intLyu \colon \tensUnit \to \coend$ of the universal Hopf algebra $\coend\in\cat$, following~\cite{BGR2}, and characterise the twist-nondegeneracy condition in quasi-Hopf terms.
We then describe in Section~\ref{sec:pullback-qHopf}  how to get
intermediate
tensor ideals and modified traces from suitable quasi-Hopf subalgebras, 
see the precise statement in \Cref{prop:pullback_modTr_along_restriction_qHopf}.

In \Cref{sec:qHopf_applications} we review the definition of the family of non-semisimple symplectic fermion quasi-Hopf
algebras from \cite{FGR2}.
We show that $\hmodM[{\sympFerm}]$ is prime in the sense of
\cite{Laugwitz-Walton}.
Then we return to the pullback construction, and give an explicit example of a
sub-quasi Hopf algebra $A$ in $H = \sympFerm(2, \beta)$.
We introduce a family of $H$-modules $P_\parMat$, bijectively parameterised by complex $(2
\times 2)$-matrices $\parMat$, which are lifts of the projective cover of the tensor unit in
$\hmodM[A]$.
After having established some properties of these modules, we compute the pullback
modified trace on the pullback ideal of $\ProjIdeal[{\hmodM[A]}]$
in $\hmodM[H]$.

\Cref{sec:explicit_computation_link_invariants} consists of the computations of the
various
(framed)
link invariants first presented at the end of
\Cref{sec:link_invariants_from_ideals}.

\Cref{sec:lens_space_symp_ferm} contains expressions for lens space invariants (both for empty manifolds and with $\tensIdeal$-admissible ribbon graphs) in terms of quasi-Hopf data, see \Cref{prop:lens-quasi-Hopf-general}, where all tensor ideals are treated on equal footing via certain (possibly degenerate) bilinear form on the centre of the quasi-Hopf algebra. Then, we finally give the explicit computation 
for $\hmodM[{\sympFerm}]$ using the $SL(2,\mathbb{Z})$ action on the centre of~$\sympFerm$ computed in~\cite{FGR2}.

\medskip

\noindent
\textbf{Acknowledgements.}
  The authors would like to thank Christian Blanchet, Marco de Renzi, Ehud Meir, Vincentas Mulevi\v{c}ius,
Kenichi Shimizu,
  and Joost Vercruysse for answers, comments, remarks, and helpful discussions.
We also thank the referees for many helpful suggestions and improvements to the paper.  
  JB is supported  by the F\'ed\'eration Wallonie-Bruxelles through a postdoctoral fellowship in the framework of the \emph{actions de recherche concert\'ees}-grant
  ``From algebra to combinatorics, and back''.
The work of AMG was supported by the CNRS, and partially by the ANR grant JCJC ANR-18-CE40-0001 and the RSF Grant No.\ 20-61-46005. 
AMG is also grateful to Hamburg University for its kind hospitality in 2022.
  IR is partially supported by the Deutsche Forschungsgemeinschaft via the Cluster of
  Excellence EXC 2121 ``Quantum Universe'' - 390833306.

\medskip

\noindent
\textbf{Conventions.}
  Throughout an algebraically closed field $\field$ of
arbitrary characteristic 
  is fixed, and all linear structures
  will be considered over it.
  All functors between linear categories are assumed to be linear.

\section{Link invariants from tensor ideals and modified traces}
\label{sec:link_invariants_from_ideals}

In this section we start by reviewing some parts of the theory of tensor ideals and modified traces and how to use these to define link invariants, as developed in 
\cite{GPT,GKP_generalized-trace-and-mod-dim-rib-cat,GKP_ambi-objects-and-trace-functions-for-nonssi-cats,GPV_traces-on-ideals-in-pivotal-categories}. 
We prove some general properties of the resulting invariants. The main result in this section is the observation that -- in the example we consider -- working with an intermediate tensor ideal
allows one to detect internal structure of indecomposable
(but reducible) 
objects not visible to invariants build from either of the two canonical ideals, i.e.\ the whole category or the subcategory of projective objects. 
Invariants for the intermediate ideal may also distinguish knots which the two canonical ideals cannot, in our example certain knots from their mirrors.

\subsection{Tensor categories}\label{sec:finitetenscat}

Following \cite{EGNO} by a \emph{tensor category} we mean  a linear abelian
category which
\begin{itemize}
	\item
	is locally finite (finite-dimensional Hom-spaces and
	every object has finite length)
	\item 
	is monoidal with bilinear tensor product $\tensor$ and simple tensor unit
	$\tensUnit$,\footnote{In \cite[Def.\,4.1.1]{EGNO} it is instead assumed that $\mathrm{End}(\tensUnit)=\field$, however the tensor unit is necessarily simple due to~\cite[Thm.\,4.3.8]{EGNO}.}
	
	\item
	is rigid, i.e.\ each object has a left and a right dual with corresponding evaluation
	and coevaluation morphisms, see below.
\end{itemize}

Recall that an object $P \in \cat$ is projective iff the associated  Hom-functor $\cat(P, \placeholder)$ is exact. A \textit{projective cover} of an object $X \in \cat$ is an epimorphism 
$$
    p_X \colon P_X \to X
$$ 
with $P_X$ projective, such that every epimorphism $P \to X$ with $P$ projective factors epimorphically through $p_X$.
We note that we do not require the existence of projective covers in $\cat$. However, if $\tensUnit$ does have a projective cover $P_\tensUnit$, and if the socle (the largest semisimple subobject) of $P_\tensUnit$ is also $\tensUnit$, then $\cat$ is called \textit{unimodular}.

A tensor category with finitely many isomorphism classes of simple objects and where every object has a projective cover
is called a \emph{finite tensor category}.   
A \emph{fusion category} is a semisimple finite tensor category.
Without loss of generality, we will assume monoidal categories to be strictly
unital and associative.\footnote{
	The strictness assumption is made only in Sections~\ref{sec:link_invariants_from_ideals} and \ref{sec:invariants_of_manifolds} for simplicity of exposition.
	In the technical example computations in Sections \ref{sec:qHopf}--\ref{sec:lens_space_symp_ferm} we work with quasi-Hopf algebras, and their representation categories are non-strict.}

\smallskip

For the rest of this section $\cat$ is a  tensor category in the above sense.
We denote by $\Irr$ a set of representatives of isomorphism classes of simple
objects, and agree that $\tensUnit \in \Irr$.
A choice of projective cover (if it exists) of a simple object $U \in \Irr$ shall be denoted by
$\projCoverMorphism{U} \colon \projCover{U} \to U$.
Our notation for left and right duals of an object $X$ is $\dualL{X}$ and $\dualR{X}$,
and the evaluation and coevaluation morphisms are
\begin{align*}
	\evL_X \colon \dualL{X} \tensor X \to \tensUnit
	, \quad
	\evR_X \colon X \tensor \dualR{X} \to \tensUnit
\end{align*}
and
\begin{align*}
	\coevL_X \colon \tensUnit \to X \tensor \dualL{X} 
	, \quad
	\coevR_X \colon \tensUnit \to \dualR{X} \tensor X 
	\ ,
\end{align*}
respectively.
We will employ graphical calculus in the form of string diagrams, which in our convention are read
from bottom to top.
Thus e.g.\ the left evaluation and coevaluation morphisms for $X$ are drawn as
\begin{align*}
	\evL_X =
	\raisebox{.5em}{
		\ipic{-0.5}{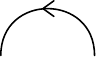}{1.3}
		\put (-65,-30) {$\dualL{X}$}
		\put (-08,-30) {$X$}
	}
	\quad \qandq \quad
	\coevL_X =
	\raisebox{.5em}{
		\ipic{-0.8}{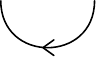}{1.3}
		\put (-65,10) {$X$}
		\put (-05,10) {$\dualL{X}$}
	}
	\ .
\end{align*}

\smallskip
To $\cat$ is associated its \emph{Grothendieck ring} $\grothring$.
This is, as an abelian group, generated by objects of $\cat$ modulo the relations $X +
Z = Y$ in $\grothring$ iff there is a short exact sequence $0 \to X \to Y
\to Z \to 0$ in $\cat$;
we denote the class of $X \in \cat$ in $\grothring$ by $\grothringclass{X}$.
The multiplication determined by $\grothringclass{X} \cdot \grothringclass{Y} =
\grothringclass{X \tensor Y}$ is well-defined because $\tensor$ is exact by
\cite[Prop.\,4.2.1]{EGNO}.
The Grothendieck ring is in fact freely generated over $\field[Z]$ by $\Irr$:
every object $X \in \cat$ possesses a finite composition series, in which the simple
object $U$ occurs  $\grMultiplicity{X}{U}$ times; if $U$ has a projective cover then $\grMultiplicity{X}{U} = \dim_{\field} \cat(\projCover{U}, X)$,
see \cite[Sec.\,1.8]{EGNO}.
Thus the class of any $X \in \cat$ can be written as a sum over classes of simple
objects $U$ with multiplicity $\grMultiplicity{X}{U}$.
In particular, there are \emph{structure constants} $\fusioncoeff{UV}{W} =
\grMultiplicity{U \tensor V}{W} \in \field[N]$ such that
\begin{align}
	\grothringclass{X} \cdot \grothringclass{Y}
	~ = \hspace{-1em}
	\sum_{U, V, W \in \Irr} 
	\hspace{-1em}
	\grMultiplicity{X}{U} \, \grMultiplicity{Y}{V} \,
	\fusioncoeff{UV}{W} \ \grothringclass{W}
	\label{eq:fusion_coeffs_definition}
\end{align}
for all $X, Y \in \cat$.

\subsection{Tensor ideals}
A \emph{right tensor ideal} $\tensIdeal$ of $\cat$ (or \emph{right ideal} for short) 
is a full subcategory such that
\begin{itemize}
	\item
	for all $X \in \tensIdeal$ and $V\in \cat$, also $X \tensor V \in \tensIdeal$, and
	
	\item
	it is closed under retracts, i.e.\ if $Z \in \tensIdeal$ and $Z \cong X \oplus Y$, then also $X, Y \in
	\tensIdeal$.
\end{itemize}
Note that by the second condition, ideals are replete, i.e.\ if $X \cong Y$ with $X \in \tensIdeal$, then $Y \in \tensIdeal$.
The definition of \emph{left} and \emph{two-sided ideals} in $\cat$ is analogous.
We will refer to two-sided ideals just as \emph{ideals}.

\begin{example}
	Two trivial examples of ideals are the subcategory consisting of only the zero object,
	as well as the category $\cat$ itself.
	An important example of an ideal is the subcategory of projective objects, which we
	denote by $\ProjIdeal$. Indeed, closure under tensor products from both sides follows from \cite[Prop.\,4.2.12]{EGNO}, and closure under retracts is the statement that direct summands of projective objects are again projective.
\end{example}

\begin{example}\label{ex:no-ideals}
It can happen that $\cat$ has no nonzero projective objects, i.e.\ that $\ProjIdeal = 0$, and in fact that there are no nonzero proper right tensor ideals at all. As a simple example consider the Hopf algebra $H = \mathbb{C}[X]$, the polynomial ring in one variable, with coproduct $\Delta(X) = X \otimes 1 + 1 \otimes X$. Let $\cat$ be the tensor category of finite-dimensional $H$-modules. The nonzero indecomposable objects are $J_m(q)$, given by an $m$-dimensional vector space on which $X$ acts via a single Jordan block with $q \in \mathbb{C}$ along the diagonal.
The projective cover of each $J_m(q)$ would be the infinite Jordan block $\mathbb{C}[X] \otimes J_1(q)$, which is not contained in $\cat$. To see that there are no proper non-zero right ideals, note that the tensor product decomposes as\footnote{To see this, first check that $J_1(q) \otimes J_m(0) \cong J_m(q)$, and then note that $J_m(0)$ is just the $m$-dimensional representation of a nilpotent part of $sl_2$.}
\begin{equation*}
    J_m(p) \otimes J_n(q) = \bigoplus_{k = 1}^{\min(m,n)} J_{m+n+1-2k}(p+q) 
    \quad \text{for all $m,n \in \mathbb{Z}_{>0}$ and $p,q \in \mathbb{C}$ .}
\end{equation*}
Hence, if a right ideal $\tensIdeal$ contains the indecomposable $J_m(p)$, it also contains the tensor unit $J_1(0)$ as it occurs as a direct summand in $J_m(p) \otimes J_m(-p)$, so that $\tensIdeal = \cat$.
\end{example}

\begin{remark}\label{rem:C-braided}
$\cat$ is \emph{braided} if it is equipped 
	with a natural isomorphism with components 
	\begin{align*}
		\braiding_{X,Y} = \ipic{-0.4}{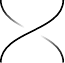}{1.5}
		\put (-52,-29) {$X$}
		\put (-05,-29) {$Y$}
		\put (-48,032) {$Y$}
		\put (-05,032) {$X$}
		~~\colon X \tensor Y \xrightarrow{\sim} Y \tensor X
	\end{align*}
	called the \emph{braiding}, satisfying the hexagon axioms, see
	e.g.~\cite[Ch.\,8]{EGNO}. 
	The inverse is drawn with the opposite crossing.
	Since ideals are replete, there is no distinction between left, right and two-sided ideals in a braided category.
\end{remark}

Let $\tensIdeal$ be a (two-sided) ideal in $\cat$. 
Then $\tensIdeal$ is closed under taking duals, i.e.\ $X \in \tensIdeal$ implies $\dualL{X}, \dualR{X} \in \tensIdeal$, see \cite[Lem.\,3.1.2]{GKP_generalized-trace-and-mod-dim-rib-cat}. For example, the zig-zag identities for the left duality maps exhibit $\dualL{X}$ as a retract of $\dualL{X} \otimes X \otimes \dualL{X}$. 

The next proposition shows that the projective ideal $\ProjIdeal$ is the smallest
non-zero ideal.
    It is a slight adaptation of \cite[Lem.\,4.4.1]{GKP_generalized-trace-and-mod-dim-rib-cat} as we do not assume $\cat$ to be braided.
To give the precise statement, we need some notation.
We say that a right
tensor ideal $\tensIdeal$ is \emph{generated by $P \in \cat$}, written
$\tensIdeal = \generatedIdeal{P}$, if every object $Q \in \tensIdeal$ can be
obtained as a retract of $P \tensor X$ for some $X \in \cat$. 
For example, the ideal
generated by $\tensUnit$ is the entire category.
The ideals of $\cat$ ordered by inclusion of full subcategories form a poset, which immediately yields the notion of a \emph{sub-ideal}.
Note that the intersection $\tensIdeal \cap \tensIdeal[J]$ of two ideals (which is a well-defined operation on
replete subcategories) is a sub-ideal in $\tensIdeal$ and in $\tensIdeal[J]$.
\begin{samepage}
	\begin{proposition}\label{prop:proj_is_smallest}
		Let $\cat$ be a  tensor category. Then:
		\begin{enumerate}[font=\upshape]
        	\item
			If $\tensIdeal$ is a  non-zero right or left 
			ideal in $\cat$, then $\ProjIdeal$
			is a sub-ideal of $\tensIdeal$.
In particular, $\ProjIdeal$ contains no right or left 
			sub-ideal other than zero and itself.	
\end{enumerate}
Suppose now that $\ProjIdeal \neq 0$. Then:
\begin{enumerate}[font=\upshape]
\setcounter{enumi}{1}
\item
$\ProjIdeal = \generatedIdeal{P}$ for any non-zero projective object $P$.	
            
\item
			$\cat$ admits a non-zero proper ideal if and only if $\cat$ is not semisimple.
		\end{enumerate}
	\end{proposition}
\end{samepage}

\begin{proof} 
To see $(1)$, assume first that $\tensIdeal$ is a right ideal in $\cat$. Let  $X \in \tensIdeal$ and $P \in \ProjIdeal$. Since $\tensIdeal$ is a right ideal and $\ProjIdeal$ is an ideal, we have $X \tensor \dualR{X} \tensor
	P \in \tensIdeal \cap \ProjIdeal$. To show that this intersection is $\ProjIdeal$, consider the epimorphism:
$$
X \tensor \dualR{X} \tensor	P \xrightarrow{\evR_X \otimes \id_P} P \ ,
$$
where we used that the right evaluation morphism $\evR_X$ is epi due to simplicity of the tensor unit and that the tensor product is exact, i.e.\ it preserves epimorphisms.
This epimorphism has a section due to projectivity of $P$, and thus $P$ is a retract of an object in $\tensIdeal$. The statement for left ideals in $\cat$ is proven similarly using  left duals and left evaluation maps.

The show $(2)$ observe that $\generatedIdeal{P}$ is a non-zero right
subideal in $\ProjIdeal$ because it contains $P$, and therefore by $(1)$ this subideal consists of all projective objects.
	
	Lastly, for $(3)$ note that
	\begin{align*}
		\cat \text{ is semisimple} 
		&\overset{(*)}{\iff} \tensUnit \in \ProjIdeal
		\\
		&\overset{\text{by }(2)}{\iff} \ProjIdeal = (\tensUnit) = \cat
		\\
		&\overset{\text{by }(1)}{\iff} 
		\text{ the smallest non-trivial ideal is $\cat$ itself}
		\ ,
	\end{align*}
	where the implication $\impliedby$ of $(*)$ follows from $X \cong X \tensor
	\tensUnit$ for every $X \in \cat$ together with the fact that the ideal of projectives is replete.
 \end{proof}

We note that the assumption $\ProjIdeal \neq 0$ in~\Cref{prop:proj_is_smallest}\,(3) is necessary as we saw in \Cref{ex:no-ideals}.

\begin{remark}
	\label{rem:ideals_in_graded_categories}
	Suppose that $\cat$ is a
tensor category which is graded as an abelian
	category by a set $S$, i.e.\ $\cat = \oplus_{s \in S}\,\cat_s$
	such that there are no non-zero morphisms between objects in $\cat_s$ and $\cat_{s'}$ for $s \neq s'$.
	While an ideal $\tensIdeal$ is in general not an abelian subcategory, it is closed
	under retracts, and thus \emph{inherits the grading} in the sense that every object $X
	\in \tensIdeal$ can be written as $\oplus_{s\in S} X_s$ with $X_s \in \tensIdeal_s$.
Let now $\tensIdeal \subset \cat$ be a non-zero right (or left, or two-sided) ideal. By \Cref{prop:proj_is_smallest}\,(1)
we have $\ProjIdeal \subset \tensIdeal \subset \cat$ and hence also $\ProjIdeal_s \subset
	\tensIdeal_s \subset \cat_s$ for each $s \in S$. If $\cat_s$ is semisimple for some $s \in S$, then $\ProjIdeal_s = \cat_s$ and hence also $\tensIdeal_s = \cat_s$.
\end{remark}

\subsection{Modified traces on tensor ideals}
Recall that a \emph{pivotal structure} on a
tensor category is a monoidal
natural isomorphism $\pivotalStruct\colon \id_{\cat} \To \ddualL{(\placeholder)}$.
Equivalently, it is a monoidal natural isomorphism between the left and the right dual
functor.

Let from now on $\cat$ be pivotal.
For simplicity of exposition, we fix a left duality functor, and take as right dual of an object $X$ the object $\dualL{X}$ with (co)evaluations induced by the pivotal
structure, meaning that e.g.\ $\evR_{X} = \evL_{\dualL{X}} \circ (\pivotalStruct_{X}
\tensor \id_{\dualL{X}})$.
Then one can define the \emph{right (quantum \emph{or} categorical) trace} of an
endomorphism $f \in \End_{\cat}(X)$ as the number
\begin{align}
	\trCat[r, \cat]_X(f) = 
	\evR_{X}
	\circ 
	(f \tensor \id_{\dualL{X}})
	\circ \coevL_X
	\in \End_{\cat}(\tensUnit)
	\ .
	\notag
\end{align}
One can similarly define the left categorical trace $\trCat[l, \cat]$ using the other pair of duality morphisms. 
Later we will only deal with ribbon categories, see \Cref{rem:modified_traces_in_ribbon_cats}, and both traces then agree. 
We will thus just write $\trCat_X(f)$. 
The \emph{(quantum \emph{or} categorical) dimension} 
of $X \in \cat$ is
$\dimCat(X) =  \trCat _X(\id_X)$.

\begin{lemma}
	\label{prop:categorical_trace_vanishes_on_all_proper_ideals}
	Let $\tensIdeal$ be a proper right ideal in the pivotal 
tensor category
	$\cat$.
	Then the right categorical trace of $\cat$ vanishes identically on $\tensIdeal$, that
	is $\trCat|_{\tensIdeal} \equiv 0$.
\end{lemma}

\begin{proof}
	Let $X \in \tensIdeal$.
	The right trace of an endomorphism $f$ of $X$ is a morphism
	\begin{align*}
		\trCat_X(f) = \tensUnit \to X \tensor \dualL{X} \to \tensUnit
		\ .
	\end{align*}
	If this is non-zero, then $\tensUnit$ is a retract of $X \tensor \dualL{X}$.
	Since $X \in \tensIdeal$ and $\tensIdeal$ is a right ideal, also $X \tensor\dualL{X} \in \tensIdeal$. But
	ideals are closed under retracts, and so $\tensUnit \in \tensIdeal$.
	Thus $\tensIdeal$ contains the ideal generated by $\tensUnit$, which is $\cat$.
\end{proof}

If the category is not semisimple, or equivalently if it admits a non-trivial proper ideal, then to obtain a version of \emph{trace} and \emph{dimension} which does not vanish on the ideal, the notion of a modified trace was introduced and studied in
\cite{GPT,GKP_generalized-trace-and-mod-dim-rib-cat,GKP_ambi-objects-and-trace-functions-for-nonssi-cats,GPV_traces-on-ideals-in-pivotal-categories}.
A \emph{right modified trace} on a right ideal $\tensIdeal$ of $\cat$ is a family of linear maps
\begin{align*}
	\left\{\modTr_X \colon \End_{\cat}(X) \to \field \right\}_{X \in \tensIdeal}
\end{align*}
satisfying the following two axioms:
\begin{enumerate}
	\item
	\emph{Cyclicity:}
	For every pair of morphisms $f,g$ between objects of $\tensIdeal$
	\begin{align*}
		\modTr_X ( X \xrightarrow{f} Y \xrightarrow{g} X )
		= \modTr_Y ( Y \xrightarrow{g} X \xrightarrow{f} Y )~.
	\end{align*}
	\item
	\emph{right partial trace property:}
	for $X \in \tensIdeal$,  $V\in \cat$, and $f \in \End_{\cat}(X\tensor V)$
	\begin{align*}
		\modTr_{XV} (f)
		= 
		\modTr_X
		\bigg(
    X \xrightarrow{\id_X \coevL_V} 
    X V \dualL{V}
    \xrightarrow{f \id_{\dualL{V}}}
    X V \dualL{V}
    \xrightarrow{\id_X \evR_V}
    X
		\bigg)~.
	\end{align*}
\end{enumerate}
Above, and also in many places below, we have omitted the $\tensor$-symbol for better readability.

The right partial trace property is simply the natural compatibility condition between
the right modified and categorical traces, and it establishes the `multiplicative'
property $\modTr_{XV} (f \tensor g) = \trCat_V(g) \ \modTr_X(f)$.

Analogously one defines \emph{left modified traces} on left ideals, where in (2) one uses the left partial trace.
On two-sided ideals, one can have both left and right modified traces;
a left modified trace on an ideal $\tensIdeal$ does not have to be right, but if it is, we will simply speak of a \emph{two-sided modified trace}, or \emph{modified trace} for
short.

\begin{remark}
	\label{rem:modified_traces_in_ribbon_cats}
	Adding to Remark~\ref{rem:C-braided}, 
    we call $\cat$ \textit{ribbon}, if it is braided and there is a natural family of isomorphism $\theta_X\colon X \to X$ such that $c_{Y,X} \circ c_{X,Y} \circ (\theta_X \otimes \theta_Y) = \theta_{X \otimes Y}$ and $\theta_{X^*} = (\theta_X)^*$. 
    For a ribbon category $\cat$, 
    similarly to the left/right categorical traces, 
    a left modified trace on a (necessarily two-sided) ideal is automatically right, and vice versa.
	The proof can easily be adapted from
	e.g.~\cite[Thm.\,3.3.1]{GKP_generalized-trace-and-mod-dim-rib-cat}.
\end{remark}

A (left, or right, or two-sided) modified trace $\modTr$ on an ideal $\tensIdeal$
induces canonical pairings
\begin{align}
	\cat(X, V) \times \cat(V, X) \to \field,
	\quad
	(f, g) \mapsto \modTr_X (g \circ f)
	\qquad
	\text{for }
	X \in \tensIdeal, V \in \cat \ ,
	\notag
\end{align}
and one calls $\modTr$ \emph{non-degenerate} if all of these pairings are
non-degenerate.

\begin{example}
	\label{example:modified_traces}
	\begin{enumerate}
		\item
		The categorical trace $ \trCat $ of $\cat$ serves as a modified trace on the ideal
		$\tensIdeal = \cat$.
		
		\item
		\label{example:modified_traces:categorical_trace_unique}
		Any modified trace $\modTr$ on $\cat$ is a scalar multiple of the categorical
		trace $ \trCat $:
		from $\tensUnit \tensor X = X$, 
        together with cyclicity and the partial trace
		property, 
  one finds $\modTr_X(f) =
		\modTr_{\tensUnit}(\id_{\tensUnit})  \trCat _X(f)$.
		
		\item
		If $\cat$ is unimodular, i.e.\ 
        $\cat( \tensUnit, \projCover{\tensUnit} ) \cong
		\field$, then $\ProjIdeal$ admits a unique (up to scalar) non-degenerate modified
		trace~\cite[Cor.\,5.6]{GKP_m-traces}.
Fix a non-zero $\injHullMorphism{\tensUnit} \in \cat( \tensUnit, \projCover{\tensUnit} )$. By non-degeneracy, we must have 
$\modTr_{\projCover{\tensUnit}}
(\injHullMorphism{\tensUnit} \circ \projCoverMorphism{\tensUnit}) \neq 0$, and we can normalise the modified trace so that $\modTr_{\projCover{\tensUnit}}
(\injHullMorphism{\tensUnit} \circ \projCoverMorphism{\tensUnit}) =1$.
 \end{enumerate}
\end{example}

\begin{corollary}\label{cor:cat-tr-nondeg}
Suppose $\ProjIdeal \neq 0$. Then
	the categorical trace $ \trCat $ is non-degenerate iff $\cat$ is semisimple.
\end{corollary}
\begin{proof}
	By \Cref{prop:categorical_trace_vanishes_on_all_proper_ideals}, $ \trCat $
	vanishes on any proper ideal, so if $ \trCat $ is non-degenerate, the only possible proper
	ideal of $\cat$ is $(0)$, and we can conclude from \Cref{prop:proj_is_smallest}\,(3) that
	$\cat$ is semisimple.
	Conversely, if $\cat$ is semisimple, then it is unimodular and $\cat = \ProjIdeal$.
	By points (2) and (3) in \Cref{example:modified_traces}, there exists a
	non-degenerate modified trace on $\cat$ which is a scalar multiple of $ \trCat $.
\end{proof}

\begin{remark}
In addition to \Cref{cor:cat-tr-nondeg}, we observe that in the non-semisimple case $\trCat$ is maximally degenerate in the following sense.
	Let $\tensIdeal$ be a right ideal in $\cat$ and $\modTr$ a right modified trace on $\tensIdeal$. Define
	\begin{align*}
		\mathrm{Ann}(\tensIdeal,\modTr)
		:= \{ X \in \tensIdeal \,|\, \modTr_X(f) = 0 \text{ for all } f \in \End_{\cat}(X) \} ~.
	\end{align*}
	Then $\mathrm{Ann}(\tensIdeal,\modTr)$ is again a right ideal. Indeed, it is closed under retracts by cyclicity of $\modTr$ and the partial trace property shows that for $X \in \mathrm{Ann}(\tensIdeal,\modTr)$ and $V \in \cat$ also $X \otimes V\in \mathrm{Ann}(\tensIdeal,\modTr)$. Since the categorical trace is non-zero on the tensor unit, $\mathrm{Ann}(\cat,\trCat)$ is a proper
    right ideal, and by \Cref{prop:categorical_trace_vanishes_on_all_proper_ideals} every other right ideal is contained in it. 
	If $\cat$ is semisimple, we have $\mathrm{Ann}(\cat,\trCat) = (0)$ by \Cref{cor:cat-tr-nondeg}. 
	If $\cat$ is not semisimple, we have the inclusions
	\begin{align*}
		(0)
        \,\subset\,
        \ProjIdeal
		\,\subset\,
		\tensIdeal
		\,\subset\,
		\mathrm{Ann}(\cat,\trCat) 
		\,\subsetneq\,
		\cat
	\end{align*}
	of right ideals, 
	for any proper and non-zero right ideal $\tensIdeal$.
\end{remark}

\medskip
In \cite{FOG}, the following `pullback construction' of tensor ideals and modified traces is introduced.
Recall that for a
monoidal functor $F \colon \cat[C] \to \cat[D]$ between rigid
categories, there is a canonical natural isomorphism $\xi_X \colon F(\dualL{X}) \to
\dualL{(FX)}$, 
for an explicit expression see
\eqref{eq:canonical_iso_duals_strong_monoidal_functor}.
If $\cat[C]$ and $\cat[D]$ are pivotal, then we say that $F$ is \emph{pivotal} if we
have commuting diagrams for all $X \in \cat[C]$:
\begin{equation}
\label{eq:comm_diag:preserve_pivotal_structures}
\begin{tikzcd}
FX \ar[r, "F \pivotalStruct^{\cat[C]}_X"]
\ar[d, swap, "\pivotalStruct^{\cat[D]}_{FX}"]
& F(\ddualL{X})
\ar[d, "\xi_{\dualL{X}}"]
\\
\ddualL{(FX)}
\ar[r, swap, "\dualL{(\xi_X)}"]
&
\dualL{(F(\dualL{X}))}
\end{tikzcd}
\end{equation}

The next statement is a special case of the construction in \cite[Prop.~2.6\,\&\,2.7]{FOG} for submodule categories and module traces.

\begin{proposition}
	\label{prop:pullback_ideal}
	Let $F \colon \cat \to \cat[D]$ be a monoidal functor between two
	tensor categories, and let $\tensIdeal$ be a right ideal of $\cat[D]$.
	\begin{enumerate}
		\item
		The pullback $F^\ast \tensIdeal$ of $\tensIdeal$ along $F$ --- i.e.\ the full
		subcategory consisting of $X \in \cat$ such that $F(X) \in \tensIdeal$ --- is a
		right ideal of $\cat$.
		
		\item
		Let the categories in addition be pivotal, and suppose that $F$ is
		pivotal.
		If $\modTr$ is a right modified trace on $\tensIdeal$, then the pullback $F^\ast
		\modTr$ of $\modTr$ along $F$ --- i.e.\ the family of linear maps defined by 
		\begin{align}
			\big( F^\ast \modTr \big)_X (f)
			= \modTr_{FX} \big( Ff \big)
			\notag
		\end{align}
		for $X \in F^\ast \tensIdeal$ --- is a right modified trace on $F^\ast \tensIdeal$.
	\end{enumerate}
\end{proposition}

\begin{proof}
Since $F \colon \cat[C] \to \cat[D]$ is a monoidal functor, $A
\rightact X = A \tensor FX$ for $X \in \cat[C]$, $A \in \cat[D]$, defines a right
$\cat[C]$-module structure on $\cat[D]$, and $F$ is a right $\cat[C]$-module functor
with respect to $\rightact$. 
Note that the associators and unitors of $\rightact$ are induced by the multiplication
and unit isomorphisms of $F$.
A right ideal $\tensIdeal \leq \cat[D]$, being closed under tensor products from the
right, is a right submodule category which is also closed under retracts, and so
\cite[Prop.~2.6]{FOG} implies that $F^* \tensIdeal$ is a right ideal.

Now let $\modTr$ be a right modified trace on $\tensIdeal \leq \cat[D]$, and let $F$ preserve pivotal structures.
We claim that these conditions guarantee that $\modTr$ is a $\cat[C]$-module trace on
the $\cat[C]$-module endocategory $(\tensIdeal, \id)$ in the sense
of~\cite[Def.~2.4]{FOG}.
To see this, note first that the canonical natural isomorphism $\xi_X \colon
F(\dualL{X}) \to \dualL{(FX)}$ is given by the composition
\begin{align}
	F(\dualL{X})
	&
	\xrightarrow{\id \tensor \coev_{FX}}
	F(\dualL{X}) \tensor FX \tensor \dualL{(FX)}
	\xrightarrow{F_2(\dualL{X}, X) \tensor \id}
	F(\dualL{X} \tensor X) \tensor \dualL{(FX)}
	\notag \\
	&
	\xrightarrow{F(\ev_X) \tensor \id}
	F\tensUnit \tensor \dualL{(FX)}
	\xrightarrow{F_0\inv \tensor \id}
	\dualL{(FX)}
	\ ,
	\label{eq:canonical_iso_duals_strong_monoidal_functor}
\end{align}
where $F_2$ and $F_0$ are, respectively, the multiplication and the unit isomorphisms of the monoidal functor $F$.
Then one inserts $\xi_X \circ \xi_X\inv$ into the right partial trace condition
satisfied by $\modTr$ for any endomorphism of $A \rightact X$, and uses the
commutativity of the diagram \eqref{eq:comm_diag:preserve_pivotal_structures} to
conclude that $\modTr$ is a module trace.
Thus the claim about pullback modified traces follows from the more general
construction in~\cite[Prop.~2.7]{FOG}.
\end{proof}

Similar statements hold of course for all variations, e.g.\ the pullback of a left
modified trace on a left ideal.
Note also that for $\cat$ ribbon, the pullback of \emph{any} modified trace is
automatically two-sided.
\Cref{prop:pullback_ideal} will be our source of non-trivial intermediate ideals with modified trace. Namely, we consider $F^\ast \tensIdeal \subset \cat$ for  $F : \cat \to \cat[D]$ pivotal, $\cat[D]$ an appropriate unimodular tensor category, and $\tensIdeal$ the projective ideal in~$\cat[D]$, see \Cref{sec:Invariants_of_ribbon_graphs} for the application to link invariants and \Cref{sec:pullback-qHopf} for details in the case of quasi-Hopf algebras.

\begin{remark}\label{rem:tens-ideal-module}
    First paragraph of the proof of \Cref{prop:pullback_ideal} suggests a  more general source of tensor ideals in $\cat$. Consider a $\cat$-module functor $F\colon \cat \to \mathcal{M}$, for $\mathcal{M}$ a right $\cat$-module category with the action functor $\rightact$ exact in the second variable,
    and $\cat$ is considered as the right regular $\cat$-module. 
    Such  $\cat$-module functors are parametrised by objects $M\in\mathcal{M}$, i.e.\ $F$ can be written in the form $F_M: X \mapsto  M \rightact X$. Assume furthermore, $\mathcal{I}$ is a $\cat$-submodule category
   of~$\mathcal{M}$  closed under retracts. One example of such a $\cat$-submodule is $\ProjIdeal[\mathcal{M}]$, due to rigidity of $\cat$ and  the assumption on~$\rightact$. Then, $F_M^*\mathcal{I}$ is a right $\cat$-submodule of the regular module category $\cat$, by \cite[Prop.~2.6]{FOG}, which is the same as
   right tensor ideal in $\cat$, and this is a two-sided ideal for braided $\cat$. 
By \cite[Prop.~2.7]{FOG}, a modified trace on $F_M^*\mathcal{I}$ can again be obtained by pulling back a module trace on $\mathcal{I} \subset \mathcal{M}$.
\end{remark}

\subsection{Invariants of \texorpdfstring{$\cat$}{C}-coloured ribbon graphs and links}
\updatelabelname{Invariants of C-coloured ribbon graphs and links}
\label{sec:Invariants_of_ribbon_graphs}
Let now $\cat$ in addition be ribbon, and let us denote by $\CategoryColoredRibGraph$
the rigid ribbon category 
of ribbon graphs embedded in $\field[R]^2 \times [0,1]$. Recall that a ribbon graph consists of oriented framed strands coloured by objects of $\cat$ and of rectangles called coupons where strands can end, coloured by morphisms in $\cat$. More precisely, $\CategoryColoredRibGraph$ has (see \cite[Ch.\,I.2]{Turaev-book} for full details)
\begin{itemize}
	\item
	objects: 
	Finite tuples 
	$(\vvec{V}, \vvec{\varepsilon}) = ( (V_1, \varepsilon_1), \ldots,
	(V_m, \varepsilon_m))$ with $V_k \in \cat$ and
	$\varepsilon_k\in\{+,-\}$.
	We denote by $\lvert \vvec{V} \rvert = m$ the length of the tuple.
	
	\item
	morphisms: 
	Each $(\vvec{V}, \vvec{\varepsilon}) \in \CategoryColoredRibGraph$ determines in a
	natural way a set of $\cat$-coloured framed oriented points in $\field[R]^2$, say,
	along the $x$-axis.
	A morphism 
	$T\colon (\vvec{V}, \vvec{\varepsilon}) \to (\vvec{V'}, \vvec{\varepsilon'})$ 
	is then an isotopy class\footnote{
The isotopy can move the ribbon graph freely in the interior $\field[R]^2 \times (0,1)$, but keeps boundary points and framings fixed. E.g.\ coupons can be rotated and translated arbitrarily, dragging the attached ribbons along accordingly.
    } of $\cat$-coloured ribbon graphs in $\field[R]^2 \times
	[0,1]$ from $(\vvec{V}, \vvec{\varepsilon}) \times \{0\}$ to
	$(\vvec{V'}, \vvec{\varepsilon'}) \times \{1\}$ such that framings,
	orientations, and labels match --- e.g.\ an incoming boundary vertex $(V,+)$ means
	that the corresponding strand is coloured with $V$ and oriented upward.
\end{itemize}
The tensor product acts on objects by concatenation of lists, and the tensor unit of
$\CategoryColoredRibGraph$ is the empty tuple $\emptyset$.

We will draw morphisms in $\CategoryColoredRibGraph$ as blue $\cat$-coloured ribbon
graphs so as not to confuse them with string diagrams in $\cat$, for which we use
black.

The \emph{Reshetikhin-Turaev functor} $\RTfunctor \colon
\CategoryColoredRibGraph \to \cat$ is a ribbon functor which acts on objects as
\begin{align}
	\RTfunctor( \vvec{V}, \vvec{\varepsilon} )
	= \bigotimes_{i = 1}^{ \lvert \vvec{V} \rvert}
	V_i^{\varepsilon_i}
	\quad \text{where }
	V^+ = V
	\text{ and }
	V^- = \dualL{V}
	\label{eq:FC-functor-on-obj}
\end{align}
and on morphisms via evaluation of the corresponding string diagram, see \cite{RT90} and~\cite[Ch.\,I]{Turaev-book}.

Fix now an ideal $\tensIdeal$ in $\cat$ and a modified trace $\modTr$ on
$\tensIdeal$.
A morphism in $\CategoryColoredRibGraph$ is \emph{closed} 
if it is an endomorphism of the tensor unit $\emptyset$ of $\CategoryColoredRibGraph$, 
and a  closed  $\cat$-coloured ribbon
graph $T$ is \emph{$\tensIdeal$-admissible} if it has at least one edge coloured by an
object in $\tensIdeal$.
If the ideal $\tensIdeal$ is clear from the context, we will sometimes just say
\emph{admissible}.
Given an $\tensIdeal$-admissible graph $T$ with an edge coloured by $X \in \tensIdeal$,
a \emph{cutting presentation} of $T$ is a $\cat$-coloured ribbon graph $T_X \in
\CategoryColoredRibGraph( (X,+), (X,+) )$ such that
\begin{align*}
	T = \trCat[{ \CategoryColoredRibGraph }]_{(X,+)} (T_X)
	\ .
\end{align*}
Note that cutting presentations are not necessarily unique.
Nevertheless, the properties of modified traces allow for the following:

\begin{theorem}
	\label{prop:renormalized_link_invariant}
	Let $T$ be an $\tensIdeal$-admissible $\cat$-coloured ribbon graph.
	Then the number
	\begin{align*}
		\renRTinvariant(T) := \modTr_X \big( \RTfunctor (T_X) \big)
		\quad
		\text{ for some cutting presentation $T_X$ of $T$ }
	\end{align*}
	is well-defined, i.e.\ an invariant of the isotopy class of $T$.
\end{theorem}

This is shown in \cite{GPT,GPV_traces-on-ideals-in-pivotal-categories}, see also \cite[Thm.\,3.3]{DGGPR} for a proof in the present setting (in \cite{DGGPR} finiteness is assumed, but the proof does not use it).
As in \cite{DGGPR} we call $\renRTinvariant$ the \emph{renormalised invariant of
	$\tensIdeal$-admissible $\cat$-coloured ribbon graphs}.
If $\tensIdeal = \cat$ and $\modTr =  \trCat $, then $\renRTinvariant[ \trCat ] =
\RTfunctor$, i.e.\ the renormalised invariant based on the categorical trace is
just the standard Reshetikhin-Turaev invariant of closed $\cat$-coloured ribbon
graphs.

Let $L^M_X$ denote a closed $\cat$-coloured ribbon graph containing a component which is a knot with colour $M \in \cat$. By this we mean that this component is an embedded ribbon loop which does not touch any coupons, but which is allowed to link with the rest of the ribbon graph. Furthermore, $L^M_X$ contains an edge (distinct from the knot) coloured by $X \in \tensIdeal$. 

\begin{proposition}
	\label{prop:ren_link_inv_factors_INITIAL}
    For any exact sequence $0 \to A \to B \to C \to 0$ in $\cat$ we have
	\begin{align}
		\renRTinvariant ( L^B_X )
		=
		\renRTinvariant ( L^A_X ) + \renRTinvariant ( L^C_X )
		\ .
		\notag
	\end{align}
\end{proposition}

Before proving this, let us state an obvious consequence:
intuitively, it tells us that given an admissible graph, any component that `looks like
a loop' need only be coloured by simple objects, as long as we keep one edge coloured
by something from the ideal on which our modified trace acts.
More precisely, we have the following corollary.

\begin{corollary}
	\label{prop:renormalized_link_invariant_factors_through_Gr}
	The renormalised Reshetikhin-Turaev invariant of $L^B_X$
	depends only on the class of $B$ in the Grothendieck ring $\grothring$ of $\cat$,
	that is
	\begin{align}
		\renRTinvariant (L^B_X)
		=
		\sum_{U \in \Irr} \grMultiplicity{B}{U} \cdot \renRTinvariant (L^U_X)
		\ .
		\notag
	\end{align}
\end{corollary}

\begin{proof}[Proof of \Cref{prop:ren_link_inv_factors_INITIAL}]
	In $\CategoryColoredRibGraph$ we can without loss of generality represent $L_X^B$ as
	\begin{align}
		L^B_X
		=
		\ipic{-0.4}{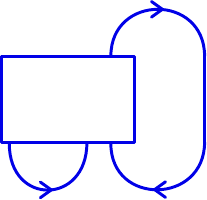}{0.6}
		\put (-50,000) {\textcolor{blue}{$\widetilde{L}^B_X$}}
		\put (-70,-20) {\textcolor{blue}{$B$}}
		\put (003,000) {\textcolor{blue}{$X$}}
		\put (015,000) {,}
		\notag
	\end{align}
	and cutting off the top cap and the bottom left cup, we obtain a
	morphism
	\begin{align}
		\eta(X)_B
		=
		\RTfunctor 
		\left(
		~
		\ipic{-0.4}{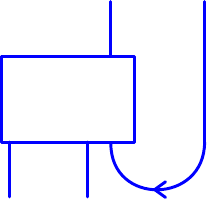}{0.6}
		\put (-50,000) {\textcolor{blue}{$\widetilde{L}^B_X$}}
		~
		\right)
		\in \cat( \dualL{B} \tensor B, X \tensor \dualL{X} )
		\notag
	\end{align}
	after evaluating with the Reshetikhin-Turaev functor.
	By our requirement on the $B$-coloured edge to be part of a knot, this actually defines a dinatural transformation
	$ \eta(X) \in \Dinat( 
	\dualL{\placeholder} 
	\tensor \placeholder, X \tensor
	\dualL{X} )$.
	If $0 \to A \to B \to C \to 0$ is any short exact sequence, by
	\cite[Lem.\,2.1]{GR-proj} we have
	\begin{align}
		\eta(X)_B \circ \coevR_B
		=
		\eta(X)_A \circ \coevR_A + \eta(X)_C \circ \coevR_C
		~.
		\notag
	\end{align}
	From monoidality of $\RTfunctor$ and linearity of the modified trace we therefore
	get
	\begin{align}
		\modTr_X
		\bigg(
		\RTfunctor
		\bigg(
		~
		\ipic{-0.4}{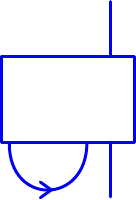}{0.6}
		\put (-30,000) {\textcolor{blue}{$\widetilde{L}^B_X$}}
		~
		\bigg)
		\bigg)
		&= 
		\modTr_X
		\bigg(
		\RTfunctor
		\bigg(
		~
		\ipic{-0.4}{./img/proof_factor_through_GR_3.pdf}{0.6}
		\put (-30,000) {\textcolor{blue}{$\widetilde{L}^A_X$}}
		~
		\bigg)
		\bigg)
		+
		\modTr_X
		\bigg(
		\RTfunctor
		\bigg(
		~
		\ipic{-0.4}{./img/proof_factor_through_GR_3.pdf}{0.6}
		\put (-30,000) {\textcolor{blue}{$\widetilde{L}^C_X$}}
		~
		\bigg)
		\bigg)
		\ ,
		\notag
	\end{align}
	as claimed.
\end{proof}

\begin{remark}
	\label{rem:some_label_name_dont_know}
	\begin{enumerate}
		\item
		\Cref{prop:ren_link_inv_factors_INITIAL} and its
		\Cref{prop:renormalized_link_invariant_factors_through_Gr} also hold for more general
		types of $\cat$-coloured graphs.
		Indeed, looking at the proof, we see that it is enough to require the $M$-coloured
		edge of $L_X^M$ to be part of a natural transformation, i.e.\ natural in $M$.

		\item
		\label{rem:RT_invariant_factors_through_GR}
The invariant $\renRTinvariant(T)$ for an $\tensIdeal$-admissible ribbon graph $T$ does
not change if we add a $\tensUnit$-coloured loop $O_\tensUnit$ to $T$:
$\renRTinvariant(T) = \renRTinvariant(T \sqcup O_\tensUnit)$.
For the ideal $\tensIdeal = \cat$ every object is admissible, in particular the tensor
unit $\tensUnit$. For a $\cat$-coloured link $L$ we can write
\begin{align*}
	\RTfunctor(L) =
	\renRTinvariant[{ \trCat }](L) = 
	\renRTinvariant[{ \trCat }](L \sqcup O_\tensUnit)
\end{align*}
and apply \Cref{prop:renormalized_link_invariant_factors_through_Gr} with $X=\tensUnit$.
This shows that the Reshetikhin-Turaev invariant $\RTfunctor$ depends on the objects
colouring $L$ only up to their class in the Grothendieck ring.
	\end{enumerate}
\end{remark}

\subsection{Examples of link invariants for symplectic fermions}
\label{sec:examples_link_invariants}

We will now give some examples of renormalised Reshetikhin-Turaev
invariants.

\begin{figure}[tb]
	\centering
	\begin{subfigure}{0.3\textwidth}
		\begin{align*}
			\ipic{-0.5}{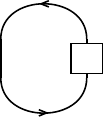}{1.5}
			\put (-19,-04) {$\ribTwist_{\!X}^{\,n}$}
		\end{align*}
		\caption{$\framedUnknot{n}{X}$}
	\end{subfigure}
	\begin{subfigure}{0.3\textwidth}
		\begin{align*}
			\ipic{-0.5}{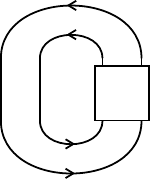}{1.0}
			\put (-24,-05) {$\braiding_{X, X}^{m}$}
		\end{align*}
		\caption{$\torusknot{m}{X}$}
	\end{subfigure}
	\begin{subfigure}{0.3\textwidth}
		\begin{align*}
			\ipic{-0.5}{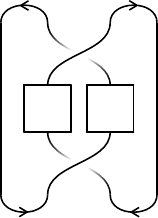}{1.0}
\put (-60,-05) {$\ribTwist^b_U$}
\put (-30,-05) {$\ribTwist^a_X$}
            \put (-53,-45) {$X$}
    		\put (-37,-45) {$U$}
		\end{align*}
		\caption{$\hopflink{X,a}{U,b}$}
	\end{subfigure}
	\caption{The three types of links used in the computation of example invariants.
		Here, $X \in \tensIdeal$, $U \in \cat$, 
  $n,a,b \in \field[Z]$, 
  $m \in 2 \field[Z]+1$.  
	}
	\label{fig:example-links}
\end{figure}

We consider the following families of $\cat$-coloured links, where $X, U \in \cat$ (see
Figure~\ref{fig:example-links}):
\begin{itemize}
	
	\item
	$\framedUnknot{n}{X}$, $n \in \field[Z]$:
	the unknot with framing number $n$.
	
	\item
	$\torusknot{m}{X}$, $m \in 2\field[Z]+1$: the $(2, m)$-torus knot.
	This is the braid closure, i.e.\ the categorical trace in $\CategoryColoredRibGraph$, of $\braiding_{X, X}^{m}$.
	
	\item
 $\hopflink{X,a}{U,b}$, $a,b \in \field[Z]$:
	the Hopf link of $X$ with $U$ with framing $a$ and $b$, respectively, i.e.\ the braid closure of $(\ribTwist_{X}^{\,a} \otimes \ribTwist_{U}^{\,b}) \circ \braiding_{U, X} \circ
	\braiding_{X, U}$ in $\CategoryColoredRibGraph$.
	
\end{itemize}
Note that $\framedUnknot{\pm 1}{X}$ is isotopic to $\torusknot{\pm 1}{X}$, and that
$\torusknot{3}{X}$ is the $X$-coloured trefoil.  
Replacing $+$ with $-$ in $\framedUnknot{\pm n}{X}$ or $\torusknot{\pm m}{X}$ turns
each knot into its mirror image.
By an appropriate rotation one sees that $\hopflink{X,a}{U,b}$ and  $\hopflink{U,b}{X,a}$ are isotopic.

For each link $L$ in the list above, we compute the invariant $\renRTinvariant(L)$,
where we vary the input category $\cat$, the tensor ideal $\tensIdeal$ with modified
trace $\modTr$, and the colours $X,U$ of the link components.
The object $X$ will always be in $\tensIdeal$, while on $U$ no such restriction is
imposed.
However, by \Cref{prop:renormalized_link_invariant_factors_through_Gr} it is sufficient
to restrict oneself to simple $U$. We have the relations
\begin{align}
    \renRTinvariant(\framedUnknot{\pm 1}{X}) &= \renRTinvariant(\torusknot{\pm 1}{X}) \ ,
\nonumber\\
  \renRTinvariant(\hopflink{X,a}{U,b}) &= \renRTinvariant(\hopflink{U,b}{X,a}) \ ,
\nonumber\\
    \renRTinvariant(\framedUnknot{1}{X \otimes U}) &= \renRTinvariant(\hopflink{X,1}{U,1}) \ ,
\label{eq:knot-relations-invs}    
 \end{align}
where the last line follows from $(\ribTwist_{X} \otimes \ribTwist_{U}) \circ \braiding_{U, X} \circ \braiding_{X, U} = \ribTwist_{X \otimes U}$.

\medskip

Let us consider the ingredients $\cat$, $\tensIdeal$, $\modTr$ and $X$ in turn.

\subsubsection*{The ribbon category $\cat$:}

We take $\cat$ to be the category $\hmodM[\sympFerm]$ of finite-dimen\-sion\-al left 
modules over $\sympFerm = \sympFerm(N, \beta)$, the (finite-dimensional)
\emph{symplectic fermion} quasi-Hopf algebras from \cite{FGR2}.
In particular, $\hmodM[\sympFerm]$ is a finite tensor category.
The parameters $N \in \field[Z]_{>0}$ and $\beta  \in \field[C]$ are subject to 
$\beta^4 = (-1)^N$.
We review the definition of $\sympFerm(N, \beta)$ in detail in \Cref{sec:qHopf_applications}. 

The value of $\beta$ affects the ribbon quasi-Hopf structure of $\sympFerm$, but not the
algebra structure.
Independently of $N$, the algebra $\sympFerm(N, \beta)$ has precisely four simple modules, denoted
$X_0^\pm$ and $X_1^\pm$.
The simple objects $X_1^\pm$ are projective, but $X_0^\pm$ are not.
The length of the composition series of the projective covers
$P_0^\pm$ of $X_0^\pm$ is $2^{2N}$, so in particular it grows with $N$.
The category $\hmodM[\sympFerm]$ is unimodular, i.e.\ the tensor unit $X_0^+$ is also the socle of $P_0^+$.

\subsubsection*{The tensor ideal $\tensIdeal$ and the modified trace $\modTr$:}

We consider three different choices of tensor ideal in $\cat=\hmodM[\sympFerm]$:
\begin{enumerate}
	\item $\hmodM[\sympFerm]$ with the categorical trace $ \trCat $.
	\item $\ProjIdeal[{\hmodM[\sympFerm]}]$ with the unique-up-to-scalar non-degenerate modified trace $\modTr$ from \Cref{example:modified_traces}\,(3).
 \item $\relativeProjectives{A}{H}$ with pullback modified trace $\pullbackTrace[{\modTr^A}]$, where  we take 
 $$H:= \sympFerm(2, \beta)$$
 and $A \subset H$ is a pivotal unimodular quasi-Hopf subalgebra.
\end{enumerate}
In point (3),
the intermediate ideal
$\tensIdeal =
	\relativeProjectives{A}{H}$ is the pullback ideal (as in \Cref{prop:pullback_ideal}) of
	$\ProjIdeal[{\hmodM[A]}]$ along the restriction functor $\restrictionFunctor$
 and $\pullbackTrace[{\modTr^A}]$ is the pullback of the unique-up-to-scalar modified trace on $\ProjIdeal[{\hmodM[A]}]$.
 The full details can be found in \Cref{sec:symp_ferm_pullback}.

For $N=2$ the three ideals are properly contained in each other:
\begin{equation}	
\ProjIdeal[{\hmodM[H]}] ~\subsetneq~ 
\relativeProjectives{A}{H} ~\subsetneq~ \hmodM[H]~.
\end{equation}
Corresponding examples exist for any $N \ge 2$, see \Cref{sec:other-A}.

\subsubsection*{The colouring objects:}

Depending on the tensor ideal in question, the colouring objects are as follows:
\begin{enumerate}
	\item $\hmodM[\sympFerm]$:
	By \Cref{rem:some_label_name_dont_know}\,(\ref{rem:RT_invariant_factors_through_GR}), colouring link components by
	$\Irr[{\hmodM[\sympFerm]}]$ is sufficient for knowing the invariants for arbitrary colours in $\hmodM[\sympFerm]$. Hence in this case we restrict ourselves to the four simple objects $X_0^\pm$, $X_1^\pm$ of $\hmodM[\sympFerm]$.
	
	\item $\ProjIdeal[{\hmodM[\sympFerm]}]$:
	We know the invariants for $P \in \ProjIdeal[{\hmodM[\sympFerm]}]$ once we know the
	invariants for all projective indecomposables, i.e.\ for the projective covers of 
	simple objects.
	
	\item $\relativeProjectives{A}{H}$:
	We will consider a family
	of objects $P_\parMat \in \relativeProjectives{A}{H}$ indexed by $\parMat \in
	\operatorname{Mat}_2(\field[C])$.
	The modules $P_\parMat$ and $P_{\parMat'}$ always have the same class in the Grothendieck ring of $\hmodM[\sympFerm]$,
	but are isomorphic only if $\parMat = \parMat'$, see \Cref{sec:symp_ferm_pullback} for their
	construction and structure.
\end{enumerate}

\newcommand{\myLinewidth}{0.9pt}
\newcommand{\myLinewidthThicker}{1.5pt}
\renewcommand{\arraystretch}{1.5}
\newcolumntype{?}{!{\vrule width \myLinewidth}}
\newcolumntype{^}{!{\vrule width \myLinewidthThicker}}

\captionsetup{format=hang}

\afterpage{
	\begin{landscape}
		\begin{table}[t]
			\caption{Some link invariants based on the categories of modules over the symplectic fermion quasi-Hopf algebras.
				The blue expressions have been interpolated from \texttt{Mathematica}-output for several values of $N$ and $m$.
				We require $a,b,n \in \field[Z]$ and $m \in 2\field[Z]+1$.
			}
			\updatelabelname{%
				Renormalised RT invariants based on SF.
			}
			
			\label{table:invariants}
			\begin{tabular}{c^c?c?c}
				\makecell{%
					ideal $\tensIdeal$ in $\cat$
					, \\ modified trace $\modTr$
				}
				& ~$\hmodM[\sympFerm]$, $ \trCat $
				& $\ProjIdeal[{\hmodM[\sympFerm]}]$, $\modTr$
				&  
				\makecell{%
					$\relativeProjectives{A}{H}$, $\pullbackTrace[{\modTr^A}]$
					\\ 
					(for $N = 2$)
				}
				\\ \Xhline{\myLinewidthThicker}
				\multirow{2}{*}{
					$X$: $\renRTinvariant( \framedUnknot{n}{X} )$
				}
				& $X_0^\varepsilon$: $\varepsilon$
				& $P_0^\varepsilon$: $\tfrac{\varepsilon}{2} n^N \beta^{-2}$
				& \multirow{2}{*}{
					$P_\parMat$: $2 n (1 + \det\parMat)$
				} 
				\\ \cline{2-3}
				& $X_1^\varepsilon$: $0$ 
				& $X_1^\varepsilon$: $2^{-N-1} \varepsilon^{n+1} \beta^{-n}$
				&  
				\\ \hline
				\multirow{2}{*}{
					$X:$ $\renRTinvariant ( \torusknot{m}{X} )$
				}
				& $X_0^\varepsilon$: $\varepsilon$
				& \textcolor{blue}{$P_0^\varepsilon$: 
					$\tfrac{\varepsilon}{2} m^N \beta^{-2}$}
				& \multirow{2}{*}{
					$P_\parMat$: $2 m (1 + \det\parMat)$
				} \\ \cline{2-3}
				& $X_1^\varepsilon$: $0$
				& \textcolor{blue}{$X_1^\varepsilon$: $2^{-N-1} m^N \beta^{m-2}$}
				&
				\\ \hline
				\multirow{2}{*}{
     $(X, U):$  $\renRTinvariant ( \hopflink{X,a}{U,b} )$
				}
				& \multirow{4}{*}{%
					$(X^\nu_i, X^\rho_j)$: $\delta_{i, 0} \delta_{j, 0} \, \nu \rho$
				}
				& 
    $(P^\nu_0, X_0^\rho)$: $\tfrac{1}{2} \nu \rho \, a^N \beta^{-2}$
				& 
    \multirow{2}{*}{$(P_\parMat, X_0^\pm)$: $\pm 2 a (1 + \det\parMat)$}
				\\ \cline{3-3}
				& 
				& 
$(X^\nu_1, X_0^\rho)$: $\nu 2^{-N-1} (\nu \beta^{-1})^a$
				&
				\\ \cline{3-4}
				& 
    & 
    $(P^\nu_0, X_1^\rho)$: $\rho 2^{N-1} (\rho \beta^{-1})^b$
				& \multirow{2}{*}{$(P_\parMat, X_1^\pm)$: $0$}
				\\ \cline{3-3}
				& 
    & $(X^\nu_1, X_1^\rho)$: $\frac12 (\nu \beta^{-1})^a (\rho \beta^{-1})^b$
				&
			\end{tabular}
		\end{table}
	\end{landscape}
}

The link invariants obtained from these input data are listed in Table~\ref{table:invariants}.  The explicit computations of these invariants can be found in
\Cref{sec:explicit_computation_link_invariants}. 
Let us collect some observations about the results in that table. We order the discussion by column, i.e.\ by the choice of tensor ideal.
\begin{enumerate}
	\item $\hmodM[\sympFerm]$: 
	The simple objects $X_0^\pm$ generate a copy of the category of super-vector spaces in $\hmodM[\sympFerm]$. Since that category is symmetric, the invariant can only detect the parity of $X_0^\pm$, but none of the topology of the link. By \Cref{prop:categorical_trace_vanishes_on_all_proper_ideals}, any invariant involving a ribbon coloured by the simple objects $X_1^\pm$ is zero, as these are projective.

	\item $\ProjIdeal[{\hmodM[\sympFerm]}]$:
	Consider the results for a fixed choice of the ribbon category, i.e.\ for fixed $N$ and
	$\beta$.
	If $N$ is odd, we see that for fixed choice of colour $P_0^\pm$, the invariant
	$\renRTinvariant$ can distinguish all the framed unknots $\framedUnknot{n}{X}$ and all the torus knots $\torusknot{m}{X}$.
	By contrast, for a semisimple ribbon category, these  invariants would always be periodic\footnote{ 
This follows from the fact that in a (finitely) semisimple ribbon category, for each object~$X$, the twist $\theta_X$ has finite order (and hence also the double braiding has finite order), see \cite[Thm.\,3.1.19]{BK} -- that theorem assumes modularity, but the part of the proof concerning the order of the twist does not use that assumption. Alternatively, one can specialise the statement for finite braided tensor categories in \cite{Etingof:2002} to the semisimple case.
}
 in $n$ and $m$.
	
	For even $N$ and $\beta = \pm 1$, $\renRTinvariant$ can not distinguish the framed
	unknot from its mirror knot, and ditto for the torus knots. 
	\item $\relativeProjectives{A}{H}$:
	This is the most interesting case.
	First note that here $N=2$ and we just saw that for the choice $\beta = \pm 1$, the
	projective ideal could not tell apart a knot and its mirror for unknots
	and torus knots.
	For the intermediate ideal, these can now be distinguished even for $\beta = \pm 1$.
	
	More remarkably, for the intermediate ideal the modified trace can distinguish a
	continuum of representations in $\relativeProjectives{A}{H}$.
	For example, $\renRTinvariant[{\pullbackTrace[{\modTr^A}]}]$ can detect the value
	$\lambda \in \field[C]$ for the choice
	$\parMat = \begin{psmallmatrix} \lambda & 0 \\ 0 & 1\end{psmallmatrix}$.

On the other hand, this example is somewhat special in the sense that link invariants reduce to the knot invariant of the component labelled by $P_\parMat$ times the quantum dimensions of the other labels due to a degeneracy explained in \Cref{rem:link-only-knot-matters}.

The ideal $\relativeProjectives{A}{H}$ also contains modules which are $H$-pro\-jective, but it turns out that if we label any component of a link with such a module, the invariant is zero by \Cref{cor:proj-inv-zero}, in contrast to when we start from the ideal $\ProjIdeal[{\hmodM[\sympFerm]}]$ as in point (2).
\end{enumerate}

There are a number of simple consistency checks for the expressions given in Table~\ref{table:invariants}:
\begin{itemize}
    \item The identities listed in \eqref{eq:knot-relations-invs} are satisfied. For the last identity in \eqref{eq:knot-relations-invs} on needs to know the tensor products. For example $X_1^\nu \otimes X_1^\rho \cong P_0^{\nu\rho}$, see \Cref{sec:simple-projective}.
    \item In the middle column for the framed Hopf link, when both labels are projective, one can choose which one to cut and evaluate the modified trace on. Combining this with  \Cref{prop:renormalized_link_invariant_factors_through_Gr} and the identity $\grothringclass{P_0^\nu} = 2^{2N - 1} ( \grothringclass{X_0^+} + \grothringclass{X_0^-})$ in the Grothendieck ring (see \Cref{sec:simple-projective}) gives
    \begin{align*}
          &\renRTinvariant(\hopflink{P_0^\nu,a}{X_1^\rho,b}) 
          \\
          &= 2^{2N-1}  \renRTinvariant(\hopflink{X_1^\rho,b}{X_0^+,a}) 
          + 2^{2N-1}\renRTinvariant(\hopflink{X_1^\rho,b}{X_0^-,a})  \ ,
    \end{align*}
which is indeed the case in Table~\ref{table:invariants}.
\end{itemize}

\begin{remark}\label{rem:no-contiua-for-C-and-ProjC}
	Assume that we have a family $X_{\lambda}$ of mutually non-isomorphic objects in $\cat$ parametrised by $\lambda \in \field[C]$, such that the Grothendieck classes $\grothringclass{X_\lambda}$ are independent of $\lambda$.
	Let $L_\lambda$ be a link with one component coloured by $X_\lambda$ and all other components (if there are any) coloured in an arbitrary but fixed way. Then:
	\begin{itemize}
		\item $\tensIdeal = \cat$: By \Cref{rem:some_label_name_dont_know}\,(\ref{rem:RT_invariant_factors_through_GR}) the invariant $\renRTinvariant[{ \trCat }](L_\lambda)$ depends on the colouring
		only through the class in the Grothendieck ring, and is therefore independent of $\lambda$.
		\item $\tensIdeal = \ProjIdeal$: 
		We first note that the $X_\lambda$ cannot all be projective. 
Indeed, there are only finitely many distinct indecomposable 
projective objects with composition factors determined by
summands in the fixed Grothendieck class $\grothringclass{X_\lambda}$ (by local finiteness of $\cat$
and Fitting's Lemma, an indecomposable projective is uniquely determined by its head). Thus, for $\renRTinvariant[{ \modTr }](L_\lambda)$ to be $\tensIdeal$-admissible,
    $L_\lambda$ must have at least two components, one of which is coloured by $X_\lambda$, and one by a projective object. But then by \Cref{prop:renormalized_link_invariant_factors_through_Gr} again only the Grothendieck class of $X_\lambda$ is relevant and so $\renRTinvariant[{ \modTr }](L_\lambda)$ is independent of~$\lambda$.
	\end{itemize}
	Thus in order to have a chance to see the parameter $\lambda$, we need the $X_\lambda$ to all be in an ideal $\tensIdeal$ with $\ProjIdeal \subsetneq \tensIdeal \subsetneq \cat$. Case (3) above illustrates that then it can indeed be possible to distinguish such a continuum of objects with fixed Grothendieck class.
\end{remark}

\section{Invariants of closed 3-manifolds}
\label{sec:invariants_of_manifolds}
In this section we review the construction of invariants of closed connected  oriented
$3$-manifolds with so-called bichrome graphs inside of them \cite{DGP-renormalized-Hennings,DGGPR}, and provide
the results of an example computation for lens spaces.
In contrast to the previous section, here we need to work with \emph{finite} tensor categories. More specifically, 
the construction takes as input datum a twist non-degenerate and unimodular finite ribbon
category.

\subsection{Twist non-degeneracy, unimodularity, and factorisability}
\label{sec:twist-non-deg_etc}
Recall the definition of ends and coends from e.g.\ \cite[Ch.\,IX]{catWorkMath}.
In a finite tensor category $\cat$, the coend
\begin{align}
	\coend = \int^{X \in \cat} \dualL{X} \tensor X
	\qtextq{with dinatural transformation}
	\dinatLyu_X \colon \dualL{X} \tensor X \to \coend
	\notag
\end{align}
exists (see \cite[Cor.\,5.1.8]{KerlerLyubashenko} and \cite[Thm.\,3.6]{Shimizu:2014}), and is a coalgebra in $\cat$ \cite[Sec.\,5.2]{KerlerLyubashenko}.
Similarly, the end $\cocoend = \int_{X \in \cat} X \tensor \dualL{X}$ exists and is an algebra;
denote its universal dinatural transformation by 
$\dinatLyu^{\cocoend} : \cocoend \to X \tensor \dualL{X}$.
If $\cat$ is braided, then both $\coend$ and $\cocoend$ can be given the structure of a
Hopf algebra,
but we will only use the structure maps of $\coend$:
\begin{alignat}{6}
	&\text{mult.:} &\multL &\colon \coend \tensor \coend \to \coend
	& \hspace{4em} &\text{unit:} &\unitL &\colon \tensUnit \to \coend
	\nonumber
	\\
	&\text{comult.:} \quad &\comultL &\colon \coend \to \coend \tensor \coend 
	&&\text{counit:} \quad &\counitL &\colon \coend
	\to \tensUnit
	\ .
	\nonumber 
\end{alignat}
The dual $\dualL{\coend}$ carries the natural structure of the dual Hopf algebra, and
is in fact isomorphic to $\cocoend$ as a Hopf algebra.
There is also a Hopf pairing 
\begin{equation}
\label{eq:Hopf-pairing}
\hopfPair \colon \coend \tensor \coend \to
\tensUnit~, 
\end{equation}
which means that the morphism $\coend \to \dualL{\coend}$ built using $\hopfPair$ is a morphism of Hopf algebras.
See \cite[Sec.\,3]{FGR1} for a detailed review using the present conventions.

\smallskip

A finite tensor category is called \emph{unimodular} if the projective cover $\projCover{\tensUnit}$ of the tensor unit is self-dual, or equivalently if the Hom-space $\tensUnit \to \projCover{\tensUnit}$ is one-dimensional (rather than zero-dimensional).

Let $\cat$ be a unimodular braided finite tensor category. Then $\coend$ admits a 
\emph{two-sided integral} and a \emph{two-sided cointegral} 
with trivial object of integrals,
see e.g.\ \cite[Sec.\,4.2.3]{KerlerLyubashenko} for Hopf algebra (co)integrals in general braided categories.
This means there is a unique (up to scalar) non-zero morphism $\intLyu \in \cat( \tensUnit, \coend )$ satisfying 
\begin{align*}
	\multL \circ (\intLyu \tensor \id_\coend)
    = \intLyu \circ \counitL
    = \multL \circ (\id_\coend \tensor \intLyu)
	\ ,
\end{align*}
see \cite[Rem.\,6.2]{BGR2} for a proof using monadic cointegrals.
Dually, the cointegral $\cointLyu \in \cat( \coend, \tensUnit )$ is determined up to scalar by
\begin{align}\label{eq:Lambda-co}
	(\cointLyu \tensor \id_\coend) \circ \comultL
	= \unitL \circ \cointLyu
	= (\id_\coend \tensor \cointLyu) \circ \comultL
	\ ,
\end{align}
see \cite[Cor.\,4.10]{Sh-integrals}.\footnote{We note that this reference uses the notion dual to $\coend$, namely the end instead of the coend, and the defining equations on categorical integrals \cite[Eqs.\,(4.4)\,\&\,(4.5)]{Sh-integrals} in the unimodular case are equivalent to~\eqref{eq:Lambda-co}.}
We can and will normalise $\cointLyu$ such that \cite[Thm.\,3.3]{BKLT}
\begin{align} \label{eq:Lam-co-norm}
	\cointLyu \circ \intLyu =
	\id_{\tensUnit}\ .
\end{align}

Define the morphism $\mathcal{Q} \colon \coend \tensor \coend \to \coend \tensor \coend$ via the dinatural transformation
\begin{align}
	\dualL{X} X \, \dualL{Y} Y \xrightarrow{
		\id_{\dualL{X}} \tensor (\braiding_{\dualL{Y}, X} \circ \braiding_{X,
			\dualL{Y}}) \tensor \id_Y
	}
	\dualL{X} X \, \dualL{Y} Y 
	\xrightarrow{\dinatLyu_X \tensor \dinatLyu_Y}
	\coend \tensor \coend
	\ .
	\label{eq:curly-Q-def}
\end{align}
In terms of $\mathcal{Q}$, the Hopf pairing is given by $\hopfPair = (\counitL \otimes \counitL) \circ \mathcal{Q}$. We define the morphism $\modS_\coend : \coend \to \coend$ as
\begin{align}
	\modS_\coend = (\counitL \tensor \id_\coend) \circ \mathcal{Q} \circ (\id_\coend
	\tensor \intLyu)
	\ .
	\label{eq:S_L-def}
\end{align}

\begin{remark}\label{rem:factorisable-cat}
If $\modS_\coend$ is invertible, the unimodular braided finite tensor category $\cat$ is called \emph{factorisable}. Equivalently, $\cat$ is factorisable iff $\hopfPair$ in \eqref{eq:Hopf-pairing} is non-degenerate, i.e.\ if the induced morphism of Hopf algebras is an
isomorphism $\coend \cong \cocoend$.
To see that these two conditions are indeed equivalent, one can use the identity $\modS_\coend = (\hopfPair \otimes \id) \circ (\id \otimes \comultL) \circ (\id \otimes \intLyu)$, which is immediate from the definitions, together with non-degeneracy of the copairing $\comultL \circ \intLyu$
(see e.g. \cite[Cor.\,4.2.13]{KerlerLyubashenko} for the latter).
\end{remark}

Now let $\cat$ in addition be ribbon with ribbon twist $\ribTwist$.
Let  $\modT_\coend$ be the unique endomorphism of $\coend$ such that 
\begin{align}
	\modT_\coend \circ \dinatLyu_X = \dinatLyu_X \circ (
	\id_{\dualL{X}} \tensor \ribTwist_X )
	\label{eq:T_L-def}
\end{align}
for all $X \in \cat$. $\modT_\coend$ is invertible, and
it determines the \emph{stabilisation coefficients} $\stabCoeff \in \field$ of $\cat$
via
\begin{align}\label{eq:ribboncat-twistnondeg}
	\counitL \circ \modT^{\pm 1}_\coend \circ \intLyu 
	=
	\stabCoeff \cdot \id_{\tensUnit}
	\ .
\end{align}
We call $\cat$ \emph{twist non-degenerate} if the numbers $\stabCoeff$ are non-zero.

By postcomposition with $\modS_\coend$, $\modT_\coend$ we obtain linear endomorphisms of $\cat(\tensUnit, \coend)$, which will be useful to express the lens space invariants in \Cref{sec:Examples:lens_spaces}. For the explicit quasi-Hopf algebra computations it will be easier to have linear endomorphisms of $\End(\id_{\cat})$ instead. To relate the two, 
recall from e.g.~\cite[Sec.\,2]{GR-proj} that for a unimodular finite tensor
category $\cat$, the linear maps
\begin{align*}
	\End(\id_{\cat}) 
	\xrightarrow{\psi} \cat(\coend, \tensUnit)
	\xrightarrow{\rho} \cat(\tensUnit, \coend)
	\ ,
\end{align*}
given by
\begin{align}
	\label{eq:isomorphism_EndidC_Hom1L}
	\psi(\kappa) \circ \dinatLyu_X = \ev_X \circ (\id_{\dualL{X}} \tensor \kappa_X)
	\qandq
	\rho(g) = (g \tensor \id_{\coend}) \circ \Delta_\coend \circ \intLyu
\end{align}
are isomorphisms. 
By pullback of $\modS_\coend \circ (-)$, $\modT_\coend \circ (-)$ along $\rho \circ \psi$ we get linear endomorphisms of $\End(\id_{\cat})$, which we denote by $\modS_{\cat}$, $\modT_{\cat}$. Explicitly, for $\alpha \in \End(\id_{\cat})$,
\begin{align}\label{eq:SC-TC-explict}
\rho(\psi(\modS_{\cat}(\alpha))) = \modS_\coend \circ \rho(\psi(\alpha)) ~,
\end{align}
and similar for $\modT_{\cat}$.

\medskip

A \emph{modular} tensor category is a factorisable ribbon finite tensor category.
According to \cite[Lem.\,2.7]{DGGPR}, 
a unimodular ribbon finite tensor category
$\cat$ is modular if and only if there
exists a non-zero number $\modularityParameter \in \field^\times$ such that the
integral and cointegral of $\coend$ are related via the Hopf pairing $\hopfPair$ as
\begin{align}
	\label{eq:definition_modularity_parameter}
	\hopfPair \circ ( \intLyu \tensor \id_\coend ) 
	= \modularityParameter \cdot \cointLyu \ .
\end{align}
The number $\modularityParameter$ is called the \emph{modularity parameter}.
It is further shown in \cite[Cor.\,4.6]{DGGPR} that $\modularityParameter = \stabCoeffM \stabCoeffP$
holds in a modular tensor category.
Conversely, by \cite[Prop.\,2.6]{DGGPR} a modular tensor category is automatically
unimodular and twist non-degenerate.

\begin{remark}\label{rem:modular-remark}
	Let $\cat$ be a modular tensor category.
	\begin{enumerate}
		\label{rem:modular_actions}
		\item
		One can constrain the normalisation freedom for
		$\intLyu$ by imposing 
\begin{equation}\label{eq:norm-omega}
\hopfPair \circ (\intLyu \tensor \intLyu) = \id_{\tensUnit}\ .
\end{equation}
		This fixes $\intLyu$ up to a sign, and together with $\cointLyu \circ \intLyu =
		\id_{\tensUnit}$ implies $\modularityParameter = 1$.
		Note, however, that this is not the usual normalisation used for modular fusion
		categories as in \cite{BK,Turaev-book}, c.f.\ \cite[Rem.\,3.10]{GR-nonssi-Verlinde}.
		\item 
		Recall from e.g.~\cite{Lyu-inv-MCG,FGR1} that the modular
		group $\slTwoZ$ acts projectively on $\cat(\tensUnit, \coend)$, the $S$- and $T$-generators
    acting by composition with $\modS_{\coend}$ and $\modT_{\coend}$ as given in
    \eqref{eq:S_L-def} and \eqref{eq:T_L-def}.
		Correspondingly, $\modS_{\cat}$ and $\modT_{\cat}$ provide a projective 
		$\slTwoZ$-action on $\End(\id_{\cat})$.
		
		\item
		For any $V \in \cat$, the \emph{internal character} is the morphism $\chi_V =
		\dinatLyu_V \circ \coevR_V$ in $\cat(\tensUnit, \coend)$
		\cite{Fuchs:2010mw,Shimizu:2015}, and we denote by $\phi_V = (\rho \circ
		\psi)\inv(\chi_V)$ the natural transformation corresponding to $\chi_V$ under the
		isomorphism \eqref{eq:isomorphism_EndidC_Hom1L}.
		For example, $\chi_\tensUnit = \unitL$ is by definition the unit of the Hopf algebra
		$\coend$, and using the normalisation \eqref{eq:Lam-co-norm} one finds
		$\rho(\cointLyu)=\unitL$. Thus $\phi_{\tensUnit} = \psi\inv(\cointLyu)$.
		
		As shown in e.g.~\cite{GR-proj,FGR1}, the $S$-transformation $\modS_{\cat}$ acting on
		$\phi_V$ yields the natural transformation
		\begin{align}
			\modS_{\cat}(\phi_V)_X
			=
			\ipic{-0.5}{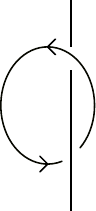}{1.0}
			\put (005,000) {$V$}
			\put (015,000) {$,$}
			\put (-10,-50) {$X$}
			\label{eq:openHopfLink_definition}
		\end{align}
		which is also called the \emph{open Hopf link operator} \cite{Creutzig:2016fms}.
		
		\item
		By \cite[Thm.\,6.1]{GR-proj}, the non-degenerate modified trace on $\ProjIdeal$
		can be chosen such that
		$\modTr_{\projCover{\tensUnit}}((\phi_{\tensUnit})_{\projCover{\tensUnit}}) = 1$.
		We remark that $(\phi_{\tensUnit})_{\projCover{\tensUnit}}$ factors through~$\tensUnit$, see \cite[Prop.\,6.5]{GR-proj}.
	\end{enumerate}
\end{remark}

\subsection{Bichrome graphs and their invariants}
\label{sec:invariants_of_bichrome_graphs}
We now recall the category of \emph{partially $\cat$-coloured bichrome ribbon graphs} from \cite{DGP-renormalized-Hennings,DGGPR}.
This is the category $\CategoryBichromeGraph$ whose objects are the same as those of
$\CategoryColoredRibGraph$, see \Cref{sec:Invariants_of_ribbon_graphs}, but morphisms
are now isotopy classes of \emph{bichrome graphs}.
These are ribbon graphs which have
\begin{itemize}
	\item
	two types of edges:
	blue edges labelled by objects in $\cat$;
	and red edges which carry no label
	
	\item
	two types of coupons:
	blue coupons coloured by morphisms in $\cat$, such that all incident edges
	are blue and the labels of coupons and edges are compatible;
	bichrome coupons, which are unlabelled and only come in the two types
	\begin{align*}
		\ipic{-0.5}{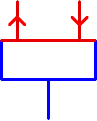}{0.9}
		\put (-18,-22) {\textcolor{blue}{$\cocoend$}}
		\quad \qandq \quad
		\ipic{-0.5}{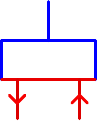}{0.9}
		\put (-18,016) {\textcolor{blue}{$\coend$}}
		\ .
	\end{align*}
\end{itemize}

It is clear that $\CategoryColoredRibGraph$ is a subcategory of
$\CategoryBichromeGraph$:
it has the same objects, and its morphisms are precisely the bichrome graphs which 
are purely blue.
Moreover, there exists an extension $\LRTfunctor \colon \CategoryBichromeGraph \to
\cat$ of the Reshetikhin-Turaev functor $\RTfunctor$, called the
\emph{Lyubashenko-Reshetikhin-Turaev functor} in \cite{DGP-renormalized-Hennings,DGGPR}.
The subscript $\intLyu$ signifies that the integral is used in the definition of
$\LRTfunctor$:
what the functor does is, essentially, to colour the red parts of the graph by
$\intLyu$.
We will now describe this functor algorithmically and in an example, more
details and proofs of the properties stated can all be found in \cite[Sec.\,3.1]{DGGPR}.

On objects, $\LRTfunctor$ agrees with $\RTfunctor$ as given in \eqref{eq:FC-functor-on-obj}.
For its action on morphisms, consider as an example the bichrome graph
\begin{align}
	T = 
	\ipic{-0.5}{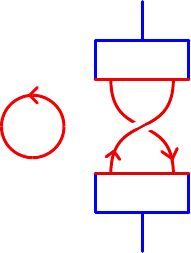}{0.8}
	\put (-16,-48) {\textcolor{blue}{$\cocoend$}}
	\put (-16,040) {\textcolor{blue}{$\coend$}}
	\notag
\end{align}
A \emph{cycle} of $T$ is a (maximal) subset of red edges in $T$ that belong to the same
component in the red graph obtained by replacing the bichrome coupons by cups and caps,
as appropriate, and throwing away the blue parts.
Let $n$ be the number of cycles in $T$, and let $s$ and $t$ be the source and target
object of $\LRTfunctor(T)$, respectively.
In our running example: $n = 2$, $s = \cocoend$, $t = \coend$.
The morphism $\LRTfunctor(T) \in \cat( s, t )$ is obtained using the following recipe.

\begin{enumerate}[{label=\textit{Step \arabic*.}}]
	\item
	For each cycle $c$ of $T$, pick an edge, then `bend' it down and `cut' it
	open so that the result is a
	so-called $n$-bottom graph $\widetilde{T}$.
	In our example one obtains a 2-bottom graph:
	\begin{align}
		T \leadsto
		\widetilde{T} = 
		\ipic{-0.3}{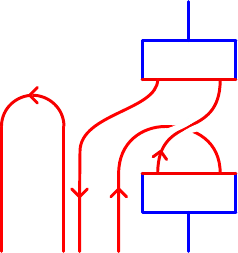}{0.8}
		\put (-16,-28) {\textcolor{blue}{$\cocoend$}}
		\put (-16,060) {\textcolor{blue}{$\coend$}}
		\notag
	\end{align}
	The new graph $\widetilde T$ has the property that when closing the cut red edges again with cups,
	one recovers the original graph $T$.
	
	\item
	Change the colour of the red part of $\widetilde{T}$ to blue,
	replacing a bichrome coupon by the appropriate universal dinatural transformation $j$ or $j^{\cocoend}$.
	The formerly red part then gives rise to a transformation $\eta_{\widetilde{T}}$
	which is dinatural in each of the $n$ pairs of consecutive red strands.
Namely $(\eta_{\widetilde{T}})_{X,Y} = \RTfunctor(\widetilde{T}_{X,Y})$, where 
	\begin{align}
\widetilde{T}_{X,Y} ~=~ 
		\ipic{-0.3}{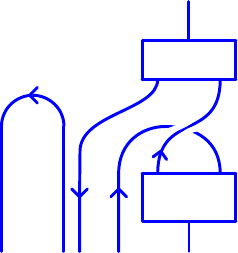}{0.8}
		\put (-22,070) {\textcolor{blue}{$\coend$}}
		\put (-22,043) {\textcolor{blue}{$\dinatLyu_Y$}}
		\put (-25,-11) {\textcolor{blue}{$\dinatLyu^{\cocoend}_Y$}}
		\put (-97,-40) {\textcolor{blue}{$\dualL{X}$}}
		\put (-77,-40) {\textcolor{blue}{$X$}}
		\put (-65,-40) {\textcolor{blue}{$\dualL{Y}$}}
		\put (-50,-40) {\textcolor{blue}{$Y$}}
		\put (-23,-40) {\textcolor{blue}{$\cocoend$}}
		\quad \text{for all } X, Y \in \cat.
		\notag
	\end{align}

	\item
	Use the universal property and the Fubini theorem for coends to obtain from
	$\eta_{\widetilde{T}}$ 
 a morphism $f_{\widetilde{T}} \in \cat( \coend^{\tensor n}
	\tensor s, t)$:
	\begin{align}
		\ipic{-0.3}{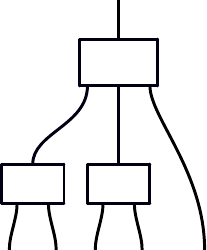}{0.8}
		\put (-37,070) {$\coend$}
		\put (-40,041) {$f_{\widetilde{T}}$}
		\put (-73,-07) {$\dinatLyu_X$}
		\put (-40,-07) {$\dinatLyu_Y$}
		\put (-05,-40) {$\cocoend$}
		:=
		\ipic{-0.3}{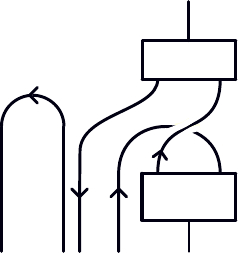}{0.8}
		\put (-22,070) {$\coend$}
		\put (-22,043) {$\dinatLyu_Y$}
		\put (-25,-11) {$\dinatLyu^{\cocoend}_Y$}
		\put (-23,-40) {$\cocoend$}
		\notag
	\end{align}
	
	\item
	Finally, define $\LRTfunctor(T)$ by inserting $n$ copies of the integral,
	\begin{align}
		\LRTfunctor(T) \vcentcolon = 
		f_{\widetilde{T}}
		\circ (\intLyu^{\tensor n} \tensor \id_s)
		\colon s \to t 
		\ .
		\notag
	\end{align}
\end{enumerate}
In \cite[Prop.\,3.1]{DGGPR} it is shown that this is a well-defined construction
which yields a ribbon functor.

\smallskip

Fix now an ideal $\tensIdeal$ in $\cat$ and a modified trace $\modTr$ on
$\tensIdeal$.
It is straightforward to generalise the notions of \emph{closed} and \emph{admissible
	$\cat$-coloured ribbon graphs} to bichrome graphs, and the concept of a (non-unique)
\emph{cutting presentation} of an admissible bichrome graph also easily translates.
This leads to what is called the \emph{renormalised invariant of admissible closed
	bichrome graphs} in \cite{DGGPR}:

\begin{theorem}[{\cite[Thm.\,3.3]{DGGPR}}]
	\label{prop:renormalized_bichrome_invariant}
	Let $T$ be an $\tensIdeal$-admissible 
	closed bichrome graph.
	Then the number
	\begin{align*}
		\renLRTinvariant(T) := \modTr_X \big( \LRTfunctor (T_X) \big)
		\quad
		\text{ for some cutting presentation $T_X$ of $T$ }
	\end{align*}
	is well-defined, and in particular an invariant of the isotopy class of $T$.
\end{theorem}

\begin{remark}\label{rem:need-not-be-integral}
	The proofs of \cite[Prop~3.1,~Thm~3.3]{DGGPR} do not rely on $\intLyu$ being an
	integral for $\coend$, and so one could have likewise defined invariants (of
	bichrome graphs) using any morphism in $\cat(\tensUnit, \coend)$.
	However, morphisms that lead to well-defined manifold invariants have to satisfy extra conditions arising from the Kirby moves.
	The integral $\intLyu$ satisfies these conditions, and more general such morphisms were investigated in~\cite{Virelizier-Kirby} under the name \emph{Kirby elements}.
\end{remark}

The invariants $\LRTfunctor$ and $\renLRTinvariant$ are both defined on purely blue closed admissible bichrome
graphs, but they are in general different.
To see this, let $\tensIdeal = \ProjIdeal[{\cat}]$ 
and let $\modTr$ be the unique non-degenerate modified trace on $\ProjIdeal$ from
\Cref{example:modified_traces}\,(3), with normalisation $\modTr_{\projCover{\tensUnit}}
(\injHullMorphism{\tensUnit} \circ \projCoverMorphism{\tensUnit}) = 1$.\footnote{
In case $\cat$ is in addition modular, one can choose $\injHullMorphism{\tensUnit}$ to satisfy $(\phi_{\tensUnit})_{\projCover{\tensUnit}} = \injHullMorphism{\tensUnit} \circ \projCoverMorphism{\tensUnit}$ \cite[Prop.\,6.5]{GR-proj}, and the normalisation of the modified trace then agrees with the one given in \Cref{rem:modular-remark}\,(4).
}
Then, unless $\cat$ is semisimple, $\renLRTinvariant(T)$ is different from
$\LRTfunctor(T)$: 
the bichrome graph $T$ given by the trace (in 
$\CategoryBichromeGraph$)  of $\injHullMorphism{\tensUnit} 
\circ \projCoverMorphism{\tensUnit}$ satisfies
$\renLRTinvariant(T) = \modTr_{\projCover{\tensUnit}}(\injHullMorphism{\tensUnit} \circ
\projCoverMorphism{\tensUnit}) = 1$
and $\LRTfunctor(T) = \projCoverMorphism{\tensUnit} \circ \injHullMorphism{\tensUnit} =
0$.

More generally, we record the following
\namecref{prop:LRTfunctor_on_admissibles_is_zero} for later use,
which is a direct consequence of \Cref{prop:categorical_trace_vanishes_on_all_proper_ideals}:

\begin{lemma}
	\label{prop:LRTfunctor_on_admissibles_is_zero}
	Let $T$ be a closed bichrome graph containing an edge coloured by an object in a proper tensor ideal. Then $\LRTfunctor(T) = 0$.
\end{lemma}

\subsection{Definition and properties of the manifold invariant}

Throughout, all manifolds will be oriented.
Let $\cat$ be twist non-degenerate and unimodular.
In order to define invariants of $3$-manifolds with embedded $\cat$-coloured bichrome
graphs, we first recall some notation following~\cite{DGGPR}.
For a framed link $L$ in the $3$-sphere $S^3$ we denote the manifold resulting from
surgery on $L$ by $M_L$.
The number of components of the link is $\ell(L)$, and its signature is $\sigma(L)$.
Let $M$ be a closed connected 3-manifold, containing a closed bichrome graph $T$.
If $M \cong M_L$, then we say that $L$ is a surgery link for $M$;
we denote by $T \subset M_L$ also the pullback of $T$ along the homeomorphism, and by
$L \cup T \subset S^3$ the corresponding bichrome graph in $S^3$, regarding $L$ as red.

Since $\cat$ is twist non-degenerate, we can choose invertible numbers $\DD,
\anomaly \in \field$ with, recall the definition of $\stabCoeff$ in~\eqref{eq:ribboncat-twistnondeg},
\begin{align}
	\DD^2 = \stabCoeffM \stabCoeffP
	\qandq
	\anomaly = \frac{\DD}{\stabCoeffM} = \frac{\stabCoeffP}{\DD}
	\notag
	\ .
\end{align}

\begin{theorem}[{\cite[Thm.\,3.8]{DGGPR}}]
	\label{def:invariant_general}
	Let $\cat$ be a unimodular and twist non-degenerate finite ribbon
	tensor category, and let
	$\tensIdeal$ be an ideal of $\cat$ with modified trace $\modTr$.
	If $M \cong M_L$ is a closed connected 3-manifold containing an $\tensIdeal$-admissible
	bichrome
	ribbon graph $T$, then
	\begin{align}\label{eq:3mfd-inv}
		\invDGGPR (M,T)
		= \DD^{-1-\ell(L)} \anomaly^{-\sigma(L)} \renLRTinvariant( L \cup T )
	\end{align}
	is a topological invariant of the pair $(M,T)$.
\end{theorem}

Following \cite{DGGPR} call $\invDGGPR$ the \emph{renormalised Lyubashenko invariant}
of admissible closed $3$-manifolds.

\begin{remark}
	\begin{enumerate}
		\item 
		For $H$ a ribbon Hopf algebra, $\cat = \hmodM$ the category of finite-dimensional
		left $H$-modules, and $\tensIdeal$ the projective ideal in $\hmodM$, the invariant
		$\invDGGPR(M,T)$ was first considered in \cite{DGP-renormalized-Hennings}, see \cite[Sec.\,2.4]{DGGPR2} for the equivalence between \cite{DGGPR} and \cite{DGP-renormalized-Hennings} in the Hopf algebra case.

		\item
		The notation $\invDGGPR$ is chosen because of the dependence of the invariant on both
		$\intLyu$ and $\modTr$.
		In fact, it also implicitly depends on a choice of sign for $\DD$.
		For if we change $\intLyu \to \intLyu' = x \intLyu$ for some $x \in \field^\times$, 
		then $\stabCoeff' = x \stabCoeff$, and consequently $\DD' = \pm x \DD$.
		If we impose the condition $\DD' = x \DD$, i.e.\ $\intLyu'/\DD' = \intLyu/\DD$, then $\anomaly' = \anomaly$ and we have
		\begin{align*}
			L'_{x \intLyu, y \modTr + z \modTr'}
			= \frac{1}{x} (y \ L'_{\intLyu, \modTr} + z \ L'_{\intLyu, \modTr'})
		\end{align*}
		for all $x,y,z \in \field[C]$ and modified traces $\modTr$ and $\modTr'$.
	\end{enumerate}  
\end{remark}

The statement of \Cref{prop:renormalized_link_invariant_factors_through_Gr} generalises to the renormalised Lyubashenko invariant, and for later reference let us restate it here in full: 

\begin{corollary}
	\label{cor:manifold_invariants_factor_through_Gr}
	Suppose the bichrome graph $T^B_X$ in $M$ has an edge labelled by $X \in \tensIdeal$ and a disjointly embedded knot coloured with $B \in \cat$ (which may link non-trivially with the remaining part of the graph $T^B_X$).
 Then 
	\begin{align}
		\invDGGPR (M, T^B_X)
		= \sum_{U \in \Irr} \grMultiplicity{B}{U} \cdot \invDGGPR (M, T^U_X)
		\ .
		\notag
	\end{align}
\end{corollary}

Lyubashenko's original invariant \cite{Lyu-inv-MCG}---or rather its extension allowing
for $\cat$-coloured ribbon graphs inside the $3$-manifolds---is the following special
case.

\begin{definition}[Lyubashenko invariant]
	\label{def:Lyubashenko_invariant}
	Let everything be as in \Cref{def:invariant_general}, but specialise to the ideal $\tensIdeal
	= \cat$.
	By \Cref{example:modified_traces} (2), without loss of generality the modified
	trace is the categorical trace $ \trCat$.
	We call 
	\begin{align}
		\invLyu(M,T)
		= \invDGGPRtrace{\trCat}(M,T)
		\notag
	\end{align}
	the \emph{Lyubashenko invariant} of the pair $(M,T)$.
\end{definition}

In case $\cat$ is semisimple, it is immediate from \cite[Sec.\,2.9]{DGGPR} that
$\invLyu$ coincides with the Reshetikhin-Turaev invariant 
as presented for example in~\cite[Thm.\,4.1.12]{BK}.  
For $\cat$ not semisimple, \Cref{prop:LRTfunctor_on_admissibles_is_zero} implies that $\invLyu (M, T)$ vanishes whenever $T$ has an edge coloured by 
an object in a proper ideal.
From \Cref{cor:manifold_invariants_factor_through_Gr} we see that  if $T$ is a link (i.e.\ there are no coupons),
$\invLyu (M, T)$ factors through $\grothring$ with respect to all components of $T$ simultaneously.

\medskip 
The invariants from \Cref{def:invariant_general} are related to the Lyubashenko
invariant as in the following proposition, which is an immediate consequence of
\cite[Prop.\,3.11]{DGGPR}.\footnote{
	There is a typo in \cite[Prop.\,3.11]{DGGPR}: the right hand side should have an extra factor of $\DD$, i.e.\ the identity should read, in the present notation, $\invDGGPR(M \# M',T')  
	= \DD \,          
	\invLyu(M) 
	\,
	\invDGGPR(M', T')$.
Also, \cite[Prop.\,3.11]{DGGPR} assumes that there is no ribbon graph in $M$, but the same proof works in the present situation where $M$ contains $T$.
}

\begin{proposition}
	\label{prop:relation_between_Lyu_and_rLyu}
	Let $M$ be a closed connected 3-manifold containing a bichrome graph
	$T = \widetilde{T} \sqcup T_{\textup{adm}}$ where $T_{\textup{adm}}$ is
		admissible and contained within an open ball in $M$ not intersecting
		$\widetilde{T}$.
		Then
		\begin{align}
			\invDGGPR(M,T) 
			= \DD \,          
			\invLyu(M, \widetilde{T}) 
			\,
			\invDGGPR(S^3, T_{\textup{adm}})
			\ .
			\notag
		\end{align}
In particular, if $\renLRTinvariant(T_{\textup{adm}}) = 1$, 
and so by \eqref{eq:3mfd-inv} $\invDGGPR(S^3, T_{\textup{adm}}) = \DD^{-1}$,
we have
		\begin{align}
			\invLyu(M,\widetilde{T})
			= \invDGGPR(M, \widetilde{T} \sqcup T_{\textup{adm}})
			\ .
			\notag
					\end{align}
		\end{proposition}					
Combining \Cref{prop:relation_between_Lyu_and_rLyu} with \Cref{prop:LRTfunctor_on_admissibles_is_zero} gives:

\begin{corollary}
	\label{cor:Lyu-zero-on-ideal}
	Suppose the bichrome graph $T$ inside the 3-manifold $M$ can be written as $T_1
	\sqcup T_2$, where $T_1$ contains an edge coloured by an object in any proper tensor ideal, and $T_2$ is admissible and
	contained within an open ball in $M$ that does not intersect $T_1$.
	Then $\invDGGPR (M,T) = 0$.
\end{corollary}

As another corollary to the above proposition, consider a situation where the ribbon
graph $T$ in $M$ can be written as $T = T' \cup T''$, such that $T'$ is disjoint from
$T''$ but not necessarily contained in a 3-ball disjoint from $T''$.
Assume that $T'$ is an embedded blue loop labelled by $X \in \tensIdeal$ with one coupon
$f : X \to X$ which factors through $\tensUnit$. 

\begin{corollary}
	\label{prop:nilpotent_knot_yields_Lyu}
	We have $\invDGGPR(M, T) = \modTr_X(f) \invLyu(M, T'')$.
\end{corollary}

\begin{proof}
As $f$ factors through $\tensUnit$, we can replace $f$ by two coupons, one from $X$ to $\tensUnit$ and one from $\tensUnit$ to $X$. This makes $T'$ contractible, and we can disentangle it from $T''$. Replacing the two new coupons again by $f$, we end up with an unknot $U_f$ with coupon $f$ embedded in a 3-ball disjoint from $T''$. 
One can now apply \Cref{prop:relation_between_Lyu_and_rLyu} with $T_{\textup{adm}} = U_f$ and $\widetilde{T} = T''$.
\end{proof}

We finish this section by showing that for $\cat$ modular, see above \eqref{eq:definition_modularity_parameter}, Lyubashenko's invariant detects semisimplicity of $\cat$ in the sense that the invariant of $S^2 \times S^1$ is non-zero if and only if $\cat$ is semisimple.

\begin{proposition}
	\label{prop:Lyubashenko_problems}
	Let $\cat$ be a modular tensor category.
	Then $\invLyu(S^2 \times S^1, \emptyset) \neq 0$ if and only if $\cat$ is semisimple.
\end{proposition}

\begin{proof}
	The surgery presentation of $S^2 \times S^1$ is a zero-framed unknot. From this one easily computes
	\begin{align}
		\invLyu(S^2\times S^1, \emptyset) 
		= \counitL \circ \intLyu
		\ ,
		\notag
	\end{align}
see e.g.\ \cite[Sec.\,5.3]{Lyu-inv-MCG}. As $\cat$ is modular, the Hopf pairing $\hopfPair$ is non-degenerate and hence by \eqref{eq:definition_modularity_parameter} transforms the integral $\Lambda$ into a cointegral,
\begin{equation}
    \counitL \circ \intLyu = \hopfPair \circ (\intLyu \otimes \unitL) = \modularityParameter \cdot \cointLyu \circ \unitL \ .
\end{equation}
It is shown in \cite[Prop.\,5.6]{Shimizu:2015} that $\cointLyu \circ \unitL \neq 0$ if and only if $\cat$ is semisimple. 
Note that \cite{Shimizu:2015} uses the so-called adjoint algebra $A$ to define the integral of $\cat$ while our $\coend\cong A^*$ as coalgebras in $\cat$ and our cointegral $\cointLyu$ corresponds to 
Shimizu's categorical integral, by comparing~\eqref{eq:Lambda-co} with~\cite[Eq.\,(5.7)]{Shimizu:2015}.
\end{proof}

One source of modular tensor categories are Drinfeld centres $\cat[Z](\cat)$
of pivotal finite tensor categories $\cat$
    satisfying a sphericality condition,
see \cite[Thm.\,5.11]{Sh-ribbon}, and these centers 
allow for a non-degenerate two-sided modified 
trace on the ideal of projective objects~\cite[Cor.\,4.7]{GR-proj}.
It is thus useful to know when exactly such centres are semisimple, or by~\Cref{prop:Lyubashenko_problems}, when the corresponding invariant of $S^2\times S^1$ is non-zero. For this, recall first  that the \emph{global dimension} of a 
    ribbon
fusion category~$\cat$ is $\dim \cat = \sum_{U \in \Irr} (\dim_{\cat} U)^2$.
For $\cat$ ribbon, semisimplicity of $\cat[Z](\cat)$ and of  $\cat$ are related as follows.

\begin{lemma}
	For  a ribbon finite tensor category $\cat$, the following are equivalent:
	\begin{enumerate}
		\item
		The centre $\cat[Z](\cat)$ is semisimple.
		\item
		$\cat$ is semisimple  and $\dim \cat \neq 0$.\footnote{Over an algebraically closed field of characteristic 0, the global dimension is automatically non-zero \cite[Thm.\,7.21.12]{EGNO}.}
	\end{enumerate}
\end{lemma}

\begin{proof}
	$(2)\Rightarrow (1)$ is shown in 
\cite[Thm.\,1.2]{Mueger}, see also  \cite[Cor.\,2.4]{BV-centers}
which uses  the theory of Hopf monads \cite{BV-hopfmonads}.
	For $(1)\Rightarrow (2)$, we first show that $\cat$ is semisimple.
	Since $\cat$ is braided, it embeds into $\cat[Z](\cat)$ via a fully faithful monoidal
	functor.
	Explicitly, the embedding sends objects $X \in \cat$ 
	to $(X, \braiding_{X, -})$ with $\braiding$ the braiding in $\cat$.
	If $\cat[Z](\cat)$ is semisimple, therefore, $\cat$ is  equivalent to a
	full (replete) abelian subcategory of a semisimple abelian category and hence itself
	semisimple abelian.
	In particular,  $\cat$ is a
ribbon
    fusion category.
     This, together with the assumption that $\cat[Z](\cat)$ is semisimple, implies that $\dim \cat \neq 0$ by \cite[Cor.\,2.3]{BV-centers}.
\end{proof}

\begin{remark}
	Suppose that $\cat$ is non-semisimple
        and modular.
	As a consequence of 
	\Cref{prop:Lyubashenko_problems} we get that
		the Lyubashenko invariant $\invLyu$ of closed $3$-manifolds can not be extended to a (rigid) TFT, see e.g.\ the introduction of \cite{Kerler-connectivity}. 
		The reason is that in a rigid TFT, the invariant of $S^2 \times S^1$ gives the dimension of the state space of $S^2$. If this is zero, the TFT functor is zero on all non-empty bordisms as one can always write such a bordism as a composition with a 3-ball, seen as a bordism $\emptyset \to S^2$. In particular, it would be zero on $S^3$, but by construction $\invLyu(S^3, \emptyset) = \DD^{-1} \neq 0$.
        Nonetheless, one can construct a 3d\,TFT on a reduced set of so-called admissible bordisms, we refer to \cite{DGGPR} for details.
\end{remark}

\subsection{Example: Lens spaces}
\label{sec:Examples:lens_spaces}

For this section let us fix two coprime positive integers $p,q$. 
The \emph{lens space} $\mathfrak{L}(p,q)$ is defined as the quotient of $S^3 = \big\{ (z,w) \in \field[C]^2 \,\big|\, |z|^2 + |w|^2 = 1 \big\}$ by the free $\field[Z]/p\field[Z]$ action
\begin{align}
	k.(z,w) = \big( e^{2\pi i k/p} z , e^{2\pi i k q/p} w \big) ~.
\end{align}
Since $S^3$ is simply connected, it follows
that the fundamental group of $\mathfrak{L}(p,q)$ is $\field[Z]/p\field[Z]$, and so the parameter $p$ is an invariant of $\mathfrak{L}(p,q)$. The parameter $q$ on the other hand is not, as the presentation above shows $\mathfrak{L}(p,q) = \mathfrak{L}(p,q+p)$.
See e.g.\ \cite[Ch.\,IV \S 11]{Prasolov-Sossinsky} for more on lens spaces.

A surgery presentation of $\mathfrak{L}(p,q)$ can be given as follows, see e.g.\ \cite[Ch.\,VI \S 17]{Prasolov-Sossinsky} or \cite[Sec.\,5.3]{Lyu-inv-MCG}.\footnote{
	One obtains an isotopic surgery link if one exchanges over- and underbraidings, so
	the surgery descriptions of $\mathfrak{L}(p, q)$ in \cite{Lyu-inv-MCG} and
	\cite{Prasolov-Sossinsky} are the same.
	Similarly, one can reverse the order of $a_1,\dots,a_n$.
}
Fix a (generalised) continued fraction expression
\begin{align}
	\frac{p}{q}
	= \raisebox{+1.0cm}{
		$a_n - \cfrac{1}{a_{n-1} - \cfrac{1}{\ddots - \cfrac{1}{a_2 - \cfrac{1}{a_1}}}}$
	}
	=: [a_n; a_{n-1}, \ldots, a_1]
	, \quad a_i \in \field[Z]
	\ .
	\label{eq:continued-fraction-notation}
\end{align}
Then we have the surgery presentation
\begin{align}
	\mathfrak{L}(p, q) = 
	\ipic{-0.4}{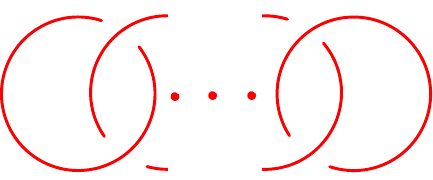}{0.7}
	\put (-130,-30) {\textcolor{red}{$a_1$}}
	\put (-100,-30) {\textcolor{red}{$a_{2}$}}
	\put (-060,-30) {\textcolor{red}{$a_{n-1}$}}
	\put (-030,-30) {\textcolor{red}{$a_n$}}
	\label{eq:Lpq-surgery-link}
\end{align}
in $S^3$.
From now on, we will require the continued fraction expression to satisfy the following
\begin{align}
	\text{\textit{Assumption:}} \quad
	\tfrac{p_i}{q_i} := [a_i; a_{i-1},
	\ldots, a_1] > 0 \quad
	\text{for $i=1,\dots,n$} \ .
	\label{eq:cont-frac-pos-assum}
\end{align}
This assumption can always be satisfied. For example, choose $a_n$ to be the ceiling of $\frac{p}{q}$ (smallest integer $\ge \frac{p}{q}$). Then in the next step one starts from $(a_n-\frac{p}{q})^{-1} \ge 0$. If this is zero, one is done, else take $a_{n-1}$ to be its ceiling, and so on.

We will show in the appendix (\Cref{prop:continued_fractions_main_statement}) that under the above assumption, the signature of the surgery link in \eqref{eq:Lpq-surgery-link} is
\begin{align}\label{eq:Lpq-surgery-link-signature}
	\sigma = \sigma(\mathfrak{L}(p, q)) = n \ .
\end{align}

With these data fixed, we will consider two types of admissible bichrome graphs
embedded in the lens space $\mathfrak{L}(p, q)$, corresponding to whether or not the
graph is contained within a ball.
More precisely, let $\alpha \in \End(\id_{\cat})$ and $P \in \tensIdeal$.
Define the bichrome graphs 
\begin{align}
	\mathfrak{L}_{*}(\alpha, P) 
	= \ipic{-0.4}{./img/lens_space_no_loop.pdf}{0.7}
	\put (-130,-30) {\textcolor{red}{$a_1$}}
	\put (-100,-30) {\textcolor{red}{$a_{2}$}}
	\put (-060,-30) {\textcolor{red}{$a_{n-1}$}}
	\put (-030,-30) {\textcolor{red}{$a_n$}}
	\quad
	\ipic{-0.4}{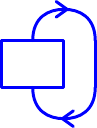}{0.9}
	\put (-35,002) {\textcolor{blue}{$\alpha_P$}}
\end{align}
and 
\begin{align}
	\mathfrak{L}_{\circ}(\alpha, P)
	= \ipic{-0.4}{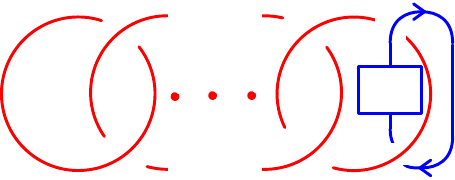}{0.9}
	\put (-180,-35) {\textcolor{red}{$a_1$}}
	\put (-150,-35) {\textcolor{red}{$a_{2}$}}
	\put (-090,-35) {\textcolor{red}{$a_{n-1}$}}
	\put (-050,-35) {\textcolor{red}{$a_n$}}
	\quad
	\put (-35,005) {\textcolor{blue}{$\alpha_P$}}
	\label{eq:Lens-circ-def}
\end{align}
in $S^3$.
Then we consider as manifold the lens space with embedded blue ribbon graph obtained by
surgery on the red part of the respective link, and denote it by the same symbol as the
link.
Note that in $\mathfrak{L}_*(\alpha, P)$, the blue graph is embedded along a
contractible cycle, and that $\invDGGPR(\mathfrak{L}_*(\alpha, T)) = \invDGGPR(\mathfrak{L}_\circ(\alpha, T))$ for any transparent object $T$.

Let $\cat$ be unimodular and twist-nondegenerate, and recall
the definition of $\modS_\coend, \modT_{\coend} : \coend \to \coend$ from \eqref{eq:S_L-def} and \eqref{eq:T_L-def}.
To state the resulting invariants of $\mathfrak{L}_{*/\circ}(\alpha, P)$, let us define
the morphism
\begin{align}
	f(a) =
	\big(
	\prod_{i=n}^1
	\modT_{\coend}^{a_i} \circ \modS_{\coend}
	\big) (\unitL)
	=
	\big(
	\modT_{\coend}^{a_n} \circ \modS_{\coend}
	\circ \ldots \circ
	\modT_{\coend}^{a_1} \circ \modS_{\coend}
	\big) (\unitL)
	~ \in ~ \cat(\tensUnit, \coend)
	\ ,
	\label{eq:definition_morphism_representing_Lpq}
\end{align}
where we abbreviate $a = (a_1, \ldots, a_n)$.
We also need the natural transformation 
$Q \colon \coend \tensor \id_{\cat} \To \id_{\cat}$ with components determined by the dinatural transformation
\begin{align}
	\dualL{X} X V 
	\xrightarrow{\id_{\dualL{X}} \tensor (\braiding_{V, X} \braiding_{X, V})}
	\dualL{X} XV 
	\xrightarrow{\evL_X \tensor \id_V}
	V
	\label{eq:definition_natural_transformation_Lid_id}
\end{align}
for $V\in \cat$.

\begin{proposition}
	\label{prop:invariants_lens_spaces_general}
Let $\cat$ be unimodular and twist-nondegenerate, 
	$\tensIdeal \subset \cat$ a tensor ideal with modified trace $\modTr$, $P \in \tensIdeal$ and $\alpha \in \End(\id_{\cat})$.
	We have
	\begin{align}
		\invDGGPR(\mathfrak{L}_{*}(\alpha, P))
		&= \anomaly^{-\sigma} \DD^{-1 -n} \cdot \modTr_{P}(\alpha_{P})
		\cdot \counitL \circ f(a) ~,
		\notag
		\\
		\invDGGPR(\mathfrak{L}_\circ(\alpha, P)) 
		&=
		\anomaly^{-\sigma} \DD^{-1 -n}
		\cdot
		\modTr_P\big(\alpha_P \circ Q_P \circ (f(a) \tensor \id_P)\big)
		\ .
		\notag
	\end{align}
\end{proposition}

\begin{proof}
	Let us start with $\mathfrak{L}_\circ(\alpha, P)$.
	To compute $\renLRTinvariant$,
by the procedure in \Cref{sec:invariants_of_bichrome_graphs}	
	we need to evaluate the
	$n$-dinatural transformation
	\begin{align}
		\label{eq:lens-computation-aux1}
		\ipic{-0.5}{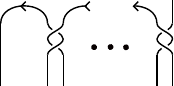}{1.5}
		\put (-137,-20) {$a_1$}
		\put (-075,-20) {$a_2$}
		\put (-025,-20) {$a_n$}
		\put (-005,-42) {$P$}		
	\end{align}
    and then compose with $\intLyu^{\tensor n}  \tensor \id_P$.
	Here we have disregarded the coupon $\alpha_P$, as we may just reinsert it on top
	once we are done.
	Now one finds that this dinatural transformation corresponds to the morphism
	\begin{align}
		\label{eq:lens-computation-aux2}
		\ipic{-0.4}{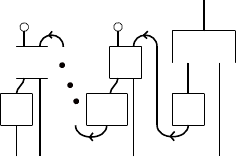}{1.5}
		\colon \coend^{\tensor n} \tensor P \to P
		\put (-250,-15) {$\modT^{a_1}$}
		\put (-235,020) {$\mathcal{Q}$}
		\put (-189,-15) {$\modT^{a_{n-1}}$}
		\put (-168,020) {$\mathcal{Q}$}
		\put (-125,-15) {$\modT^{a_n}$}
		\put (-110,030) {$Q_P$}
		\ ,
	\end{align}
	where $\mathcal{Q}$ was defined in \eqref{eq:curly-Q-def}.
From the explicit expression for $\modS_{\coend}$ in \eqref{eq:S_L-def} it is easy to see that
$\modS_{\coend} \circ \unitL = \intLyu$.
Thus, precomposing the above morphism with $\intLyu^{\tensor n}  \tensor \id_P$ and replacing the first factor of $\Lambda$ with $\modS_{\coend} \circ \unitL$
yields $Q_P \circ (f(a) \tensor \id_{P})$.
This establishes the expression for $\invDGGPR(\mathfrak{L}_\circ(\alpha, P))$. 
	
	In order to obtain $\invDGGPR(\mathfrak{L}_*(\alpha, P))$ one simply repeats the
	above computation without the double braiding between $\coend$ and $P$ in
	\eqref{eq:lens-computation-aux1}.
	This amounts to replacing $Q_P$ in \eqref{eq:lens-computation-aux2} by 
	$\counitL \otimes \id_P$.
\end{proof}

\begin{remark}\label{rem:L-vs-Lyu}
By \Cref{prop:nilpotent_knot_yields_Lyu} we have
	\begin{align}
		\invDGGPR(\mathfrak{L}_*(\alpha, P)) 
		= \modTr_P(\alpha_P) \cdot \invLyu(\mathfrak{L}(p, q))
		\ .
		\label{eq:lens-L-vs-Lyu}
	\end{align}
	Comparing to \Cref{prop:invariants_lens_spaces_general}, we read off
	\begin{align}
		\invLyu(\mathfrak{L}(p, q))
		= \anomaly^{-\sigma} \DD^{-1 -n}
		\cdot \counitL \circ f(a) 
		\ ,
		\label{eq:lens-Lyu=epsof(a)}
	\end{align}
	in agreement with 
 \cite[Sec.\,5.3]{Lyu-inv-MCG}, 
 where this invariant was first computed. Strictly speaking, we still need to show that $\modTr_P(\alpha_P) \neq 0$ for some choice of $\tensIdeal$, $P$ and $\alpha$. This is the case e.g.\ for $\tensIdeal = \cat$, $P = \tensUnit$ and any $\alpha$ with $\alpha_\tensUnit = \id_\tensUnit$, or for $\tensIdeal = \ProjIdeal$, $P = \projCover{\tensUnit}$ and the unique $\alpha \in \End(\id_{\cat})$  which sends the top of $\projCover{\tensUnit}$ to its socle, and which is zero on the projective covers of the other simples.
\end{remark}

We end this section by computing the invariants in \Cref{prop:invariants_lens_spaces_general} in the example $\cat = \hmodM[\sympFerm]$, the representations of the symplectic fermion
quasi-Hopf algebra $\sympFerm = \sympFerm(N, \beta)$ already mentioned in \Cref{sec:examples_link_invariants}, and which will be described in detail in~\Cref{sec:qHopf}.
In this example, $\cat$ is modular (and hence twist-nondegenerate, see the sentence before \Cref{rem:modular-remark}),
and we use the Lyubashenko integral with normalisation~\eqref{eq:norm-omega}; 
this also fixes a canonical
normalisation for the modified trace,
see \Cref{rem:modular_actions}\,(4).

\subsubsection*{Categorical trace}

Let us start with the case where the ideal is $\tensIdeal = \cat$ and we use the categorical trace $\trCat$. By \Cref{def:Lyubashenko_invariant}, in this case we get the original Lyubashenko invariant. We show in \Cref{sec:cat-tr-Lyu-Lens} that
\begin{equation}\label{eq:Lyu-inv-lens}
	\invLyu(\mathfrak{L}(p, q))
	= 
	p^N
	\ .
\end{equation}
Including blue ribbons as in $\mathfrak{L}_{*/\circ}(\alpha, X)$ does not add anything new. Indeed, by \Cref{cor:manifold_invariants_factor_through_Gr} it is enough to consider $X \in \cat$ simple. There are four simple objects in $\hmodM[\sympFerm]$, namely $X_0^\pm$ and $X_1^\pm$. Of these, $X_1^\pm$ are projective and so the invariant is zero by \Cref{cor:Lyu-zero-on-ideal} (where we choose $T_2 = \emptyset$). The object $X_0^+$ is the tensor unit, and since the Grothendieck class of its projective cover is $[P_0^+] = 2^{2N-1} ( [X_0^+]+[X_0^-] )$ (see \Cref{sec:mod-cat-QM}), the invariant of $X_0^-$ must be minus that of $X_0^+$. Altogether, for $x = *,\circ$,
\begin{align*}
	\invLyu(\mathfrak{L}_x(\id, X_0^\pm)) = \pm p^N
	\quad , \qquad 
	\invLyu(\mathfrak{L}_x(\id, X_1^\pm))= 0
	\ .
\end{align*}
In particular, $\invLyu$ cannot detect whether or not the blue ribbon wraps a nontrivial cycle in $\mathfrak{L}(p, q)$.

\subsubsection*{Relation to semisimplification}

The result~\eqref{eq:Lyu-inv-lens} suggests that for symplectic fermions, the Lyubashenko invariant of a 3-manifold is just a classical invariant, namely
the order of the $1^\mathrm{st}$ homology group $H_1(M)$ of the manifold $M$, in power of the ``rank'' of the category, $N$ in our case. 
We conjecture that for all closed 3-manifolds~$M$,
\begin{align*}
\invLyu(M) = \begin{cases} 0 &; H_1(M) \text{ infinite}
\\
|H_1(M)|^N &; \text{else}
\end{cases}
\end{align*}
This fits into a series of results and conjectures on the relation between non-semisimple and semisimple invariants. The first instance was proven in~\cite{CKS}: the Hennings invariant for small quantum $sl(2)$ for an odd order root of unity is proportional to the $SO(3)$ Chern-Simons invariant at the corresponding level. The proportionality coefficient is $|H_1(M)|$ if it is finite, and zero else. Related results and conjectures have been given in \cite{CYZ12,CGP2,CDGS}.
While similar, our case of $\cat = \hmodM[\sympFerm]$ is not directly covered by those previous observations. Indeed, it is the first higher rank example for manifolds $M$ which are not integer homology spheres (i.e.\ for which $|H^1(M)|>1$).

    When comparing semisimple and non-semisimple invariants, one cannot just take the semisimplification itself. Indeed,
let $\cat'_{\mathrm{ss}}$ be the semisimplification obtained by dividing out negligible morphisms, see~\cite[Thm.\,2.9]{Barrett:1993zf},
\cite[Prop.\,1.4.2]{Br-Css}, \cite[Sec.\,3.4]{Virelizier-Kirby} and \cite[Sec.\,2]{EO}. 
The projection $\cat \to \cat'_{\mathrm{ss}}$ is a ribbon functor. In $\cat'_{\mathrm{ss}}$, indecomposable objects of non-zero dimension become simple objects \cite[Prop.\,2.4]{EO}. This may lead to an uncountably infinite number of simples (this happens e.g.\ for $\hmodM[\sympFerm]$ with $N \ge 2$). Instead, let use introduce the \textit{minimal semisimplification} $\cat_{\mathrm{ss}} \subset \cat'_{\mathrm{ss}}$ to be the full subcategory tensor-generated by the images of the simple objects in $\cat$ under the projection.

The minimal semisimplification of $\cat = \hmodM[\sympFerm]$ is $\cat_{\mathrm{ss}} = \mathsf{SVect}$. This category is not modular but still twist-nondegenerate. Hence it defines Reshetikhin-Turaev invariants 
    $\mathrm{RT}_{\cat_{\mathrm{ss}}}(M)$
of closed three-manifolds $M$.
    With appropriate overall normalisation, for $\cat_{\mathrm{ss}} = \mathsf{SVect}$ these are all equal to one. 
Motivated by these examples,
it would be very interesting to see if the following general statement holds or not:
\begin{quote}
Let $\cat$ be a modular tensor category. Then its semisimplification $\cat_{\mathrm{ss}}$ is twist-nondegenerate and the corresponding invariants are related as $\invLyu(M) = |H_1(M)|^n\,\mathrm{RT}_{\cat_{\mathrm{ss}}}(M)$
for some positive integer number $n$ if $|H_1(M)|$ is finite. Otherwise $\invLyu(M)$ is zero.
\end{quote}

\subsubsection*{Modified trace on projective ideal}

Next we turn to the projective ideal $\tensIdeal = 
\ProjIdeal[{\hmodM[\sympFerm]}]$. There is no need to consider $\mathfrak{L}_*(\alpha, P)$ as by \eqref{eq:lens-L-vs-Lyu}  the corresponding invariant is given by $\invLyu(\mathfrak{L}(p, q))$ times a topology-independent prefactor.

Recall from \Cref{sec:examples_link_invariants} that $P_0^\pm$ denotes the projective cover of $X_0^\pm$.
In the next theorem, by $|t|$ for $t \in \field[Z]_2^N$ we mean the sum $\sum_{j=0}^N t_j$.

\begin{theorem}
	\label{prop:invariants_lens_spaces_symp_ferm}
	We have
	\begin{align*}
		\invDGGPR(\mathfrak{L}_\circ(\id, P_0^\pm))
		= \pm \tfrac12\big( c^0 + \beta^2 q^{N} \big) 
		\ ,
	\end{align*}
	where $c^0 \in \{0\} \cup \{ \beta^m | m\in \field[Z] \}$.
	Furthermore, for $t \in \field[Z]_2^N$ with $|t| \geq 1$, there are $\alpha_t \in \End(\id_{\cat})$ 
	(independent of $p,q$) 
	such that
	\begin{align}
		\invDGGPR(\mathfrak{L}_\circ(\alpha_t, P_0^\pm))
		&=
		\pm (-1)^{|t|} \beta^{2}
		2^{- |t| - 1}
		p^{|t|}
		q^{N - |t|}
		\notag
		\ .
	\end{align}
\end{theorem}

The arduous proof will be presented in \Cref{prop:invariants_lens_spaces_symp_ferm_extended_version} below, which in fact contains some additional invariants to those presented here.

\begin{remark}
We see that $\invDGGPR(\mathfrak{L}_\circ(\id, P_0^\pm))$ and $\invDGGPR(\mathfrak{L}_\circ(\alpha_t, P_0^\pm))$ (for $N \geq 2$ and $1 \le |t| \le N$ ) are not invariant 
 under $q \leadsto q+p$, and so are not invariants of the lens space $\mathfrak{L}(p,q)$ but only of the lens space with 
the embedded ribbon loop.
\end{remark}

\subsubsection*{Pullback trace on intermediate ideal}

Finally, we also consider the intermediate ideal 
    $\relativeProjectives{A}{H}$ with pullback modified trace $\pullbackTrace[{\modTr^A}]$,
as in the computation of link
invariants in \Cref{sec:examples_link_invariants}.

\begin{theorem}\label{thm:lens-for-intermediate}
	Let $H = \sympFerm(2, \beta)$, and consider the $\operatorname{Mat}_2(\field[C])$-indexed 
	$H$-modules $P_\parMat$ already described in \Cref{sec:examples_link_invariants}.
	Then
	\begin{align*}
		\invDGGPRtrace{\pullbackTrace[{\modTr^A}]} 
		(\mathfrak{L}_{\circ}(\id, P_\parMat))
		= - 2 
		p q (1 + \det\parMat)
		\ ,
	\end{align*}
	and there are $\alpha_{j,l} \in \End(\id_{\cat})$ (independent of $p,q$) with $j, l = 1,2$ such that 
	\begin{align*}
		\invDGGPRtrace{\pullbackTrace[{\modTr^A}]} (\mathfrak{L}_{\circ}(
		\alpha_{j, l}, P_\parMat))
		=
		p^2 
		\parMat_{j, l} \ .
	\end{align*}
\end{theorem}

The above theorem is part of the more general statement in \Cref{thm:lens-pullback-trace-SF}
where also explicit expressions for the $\alpha_{j, l}$ are given.
Note that it is possible to recover the matrix $M$ uniquely by computing all four
values of $\invDGGPRtrace{\pullbackTrace[{\modTr^A}]} (\mathfrak{L}_{\circ}(
\alpha_{j, l}, P_\parMat))$,
thereby distinguishing all modules $P_\parMat$.

\bigskip

This ends the first part of the paper, in which we set out general properties of the non-semisimple link- and three-manifold invariants, and where we presented some explicit results in the example of the symplectic fermion quasi-Hopf algebra. 

In the second part of the paper we review the necessary background on quasi-Hopf algebras and provide the technical computations of the various invariants in the symplectic fermion example.

\section{Quasi-Hopf algebras, tensor ideals and modified traces}
\label{sec:qHopf}

This \namecref{sec:qHopf} contains our conventions for quasi-Hopf
algebras, as well as background material on pivotal structures, braidings, and ribbon structures. We discuss modified traces and explain in detail how to define modified traces on tensor ideals described by quasi-Hopf subalgebras.

\subsection{Quasi-Hopf algebras and their modules}
Our conventions for ribbon quasi-Hopf algebras follow \cite{FGR1,BGR2}.

\subsubsection{Quasi-Hopf algebras}
An algebra $H$ is a quasi-Hopf algebra if it comes equipped with elements $\alphaQ,
\betaQ \in H$, $\coassQ \in H^{\tensor 3}$, and linear maps $\Delta: H \to H\tensor
H$, $\counit : H \to \field$, and $S:H \to H$ satisfying the axioms described in this
subsection.
We denote the unit of $H$ by $\oneQ$, or also by $\oneQ_H$ if it can be confused with the tensor unit.
Throughout \Cref{sec:qHopf} we will assume that $H$ is a finite-dimensional $k$-vector space.

We employ the sumless Sweedler notation to write
\begin{align}
	\Delta(h) = h\sweedler{1} \tensor h\sweedler{2}
	\qandq
	(\Delta \tensor \id) (\Delta(h)) 
	= h\sweedler{1,1} \tensor h\sweedler{1,2} \tensor h\sweedler{2}
	\ ,
	\notag
\end{align} etc., for $h \in H$.
Similarly, for $w\in
H^{\tensor 2}$ we write $w = w_1 \tensor w_2$, 
again with implicit summation,
and $w_{21} = \tau(w)$, where $\tau$ is the flip map in vector spaces.
This notation is extended to higher tensor powers in the obvious way.

The coproduct $\Delta$ and the counit $\counit$ are required to be algebra maps.
Furthermore, the counit makes the coproduct counital.
In Sweedler notation, these conditions read
$(hg)\sweedler{1} \tensor (hg)\sweedler{2} 
= h\sweedler{1} g\sweedler{1} \tensor h\sweedler{2} g\sweedler{2}$, $\counit(hg) =
\counit(h)\counit(g)$ and $\counit( h\sweedler{1} ) h\sweedler{2} = h = h\sweedler{1}
\counit( h\sweedler{2} )$ for $h,g \in H$.

The coassociator $\coassQ$ is invertible and we write $\invCoassQ = \coassQ\inv$.
    The components of the coassociator and its inverse will be written as
    \begin{align}
        \coassQ = \coassQ_1 \otimes \coassQ_2 \otimes \coassQ_3
        \quad \text{and} \quad
        \invCoassQ = \invCoassQ_1 \otimes \invCoassQ_2 \otimes \invCoassQ_3,
    \end{align}
    respectively. 
Further, it is normalised, i.e.\ $(\id \tensor \counit \tensor \id)(\coassQ) = \oneQ
\tensor \oneQ$, it makes the coproduct coassociative in the sense that
$(\Delta \tensor \id) (\Delta(h)) \cdot \coassQ = \coassQ \cdot (\id\tensor \Delta )
(\Delta(h))$
for all $h\in H$, and it satisfies a cocycle condition
\begin{align}
	(\Delta \otimes \id \otimes \id)(\coassQ) 
	\cdot (\id \otimes \id \otimes \Delta)(\coassQ) 
	= (\coassQ \otimes \oneQ) 
	\cdot (\id \otimes \Delta \otimes \id)(\coassQ)
	\cdot (\oneQ \otimes \coassQ)\ .
	\notag
\end{align}
Note that our conventions for $\coassQ$ differ from those of e.g.~\cite{HN-integrals,Bulacu_book} in that we use the inverse coassociator.

Finally, the antipode $S$ is an anti-algebra map, and together with the evaluation and
coevaluation elements $\alphaQ$ and $\betaQ$ it satisfies
\begin{align}
	S(h\sweedler{1}) \alphaQ h\sweedler{2}
	= \counit(h) \alphaQ
	, \quad
	h\sweedler{1} \betaQ S(h\sweedler{2})
	= \counit(h) \betaQ \ ,
	\quad h\in H,
	\label{eq:antipodeAxioms}
\end{align}
and
\begin{align}
	S(\coassQ_1) \alphaQ \coassQ_2 \betaQ S(\coassQ_3) = \oneQ
	, \quad
	\invCoassQ_1 \betaQ S(\invCoassQ_2) \alphaQ \invCoassQ_3 = \oneQ
	\ .
	\label{eq:zigzag-qHopf}
\end{align}
From the last two axioms it is easy to see that one may always rescale $\alphaQ$ and
$\betaQ$ such that $\counit(\alphaQ) = 1 = \counit(\betaQ)$.

We will also use the (co)opposite (co)multiplications $\mu\op = \mu \circ \tau$ and
$\Delta\cop = \tau \circ \Delta$, where $\tau$ is again the flip map in $\Vect$.

\subsubsection{Finite tensor category structure of $H$-modules}
A left $H$-module is a left module $V$ over the algebra $H$, and the action of $h \in
H$ on $v \in V$ is denoted by $h.v$.
We write $\hmodM$ for the category of finite-dimensional left $H$-modules.
It is a (non-strict) finite tensor category:
the monoidal product of two modules $V, W$ is the vector space $V \tensor W$ with
the diagonal action $h.(v \tensor w) = h\sweedler{1}.v \tensor h\sweedler{2}.w$, and
the associator of this monoidal structure is
\begin{align}
	U \tensor (V \tensor W) \to (U \tensor V) \tensor W,
	\quad
	u \tensor v \tensor w \mapsto 
	\coassQ_1 . u \tensor \coassQ_2 . v \tensor \coassQ_3 . w
	\ .
	\notag
\end{align}
Next we describe a rigid structure on $\hmodM$.
To this end, denote by 
\begin{align}
	V^* \times V \to \field, 
	\quad (f,v) \mapsto \langle f \mid v \rangle 
	\vcentcolon = f(v)
	\notag
\end{align}
the canonical pairing between a vector space $V$ and its linear dual $V^*$.
An $H$-module $V$ has both a left and a right dual.
They are given by the dual vector space $V^*$, and $h\in H$ acts on the 
left dual $\dualL{V}$ resp.\ the right dual $\dualR{V}$ by
\begin{align}
	(h.f)(v) = \langle f \mid S(h) . v \rangle
	\qtextq{resp.}
	(h.f)(v) = \langle f \mid S\inv(h) . v \rangle
	\ ,
	\notag
\end{align}
for $v\in V$, $f\in V^*$.
The corresponding evaluation is given by\footnote{
    The notation for the (co)evaluation morphisms given here differs from \Cref{sec:finitetenscat}, as we will later fix a convention adapted to a given pivotal structure, see \eqref{eq:pivot-qHopf-ev} and \eqref{eq:pivot-qHopf-coev} below.
}
\begin{align}
	\evL^\mathrm{L}_V(f \tensor v) 
	= 
	\langle f \mid \alphaQ . v \rangle
	\qtextq{resp.}
	\evL^\mathrm{R}_V(v \tensor f) 
	= \langle f \mid S\inv(\alphaQ) . v \rangle
	\ ,
	\label{eq:qHopf-ev}
\end{align}
and the coevaluation by
\begin{align}
	\coevL^\mathrm{L}_V(1) = \sum_{i=1}^{\dim V} \betaQ . v_i \tensor v^i
	\qtextq{resp.}
	\coevL^\mathrm{R}_V(1) = \sum_{i=1}^{\dim V} v^i \tensor S\inv(\betaQ) . v_i
	\ ,
	\label{eq:qHopf-coev}
\end{align}
where $\{v_i\}$ is a basis of $V$ with corresponding dual basis $\{v^i\}$.
That these four maps are indeed morphisms in the category is guaranteed by
\eqref{eq:antipodeAxioms}, and the zig-zag axioms for $\hmodM$ follow from the zig-zag
axioms \eqref{eq:zigzag-qHopf} for $H$.

\subsubsection{Special elements and (co)integrals}
\label{sec:qHopf-coints_and_qRpR}
A \emph{left integral} for a quasi-Hopf algebra is an element $c \in H$ satisfying $hc
= \counit(h) c$ for all $h \in H$.
In other words, $c \in \hmodM(\tensUnit, H)$, regarding $H$ as the left regular
representation.
One similarly defines right integrals.
It is well-known that, in the finite-dimensional case, left (right) integrals exist and
span a one-dimensional space.
A priori, left and right integrals are different;
by dimension considerations, if $c$ is a left integral, then $ch = \modulus(h) c$ for
all $h \in H$, where $\modulus \in H^*$ is an algebra morphism called the
\emph{modulus} of $H$.
Then a left integral is a right integral if and only if $\modulus = \counit$, in which
case one says that $H$ is \emph{unimodular}.\footnotemark\
\footnotetext{%
	The modulus may be identified with the one-dimensional socle of
	$\projCover{\tensUnit}$ \cite[Sec.~6.5]{EGNO}.
	Thus $H$ is unimodular if and only if $\hmodM$ is unimodular.
}

The dual notion of \emph{left/right cointegrals} for $H$ is more involved
\cite{HN-integrals}, since the (linear) dual of a quasi-Hopf algebra is not a quasi-Hopf
algebra in general.
Following \cite{Drinfeld}, define the elements 
\begin{alignat}{3}
	\qR &= \invCoassQ_1 \tensor S\inv(\alphaQ \invCoassQ_3) \invCoassQ_2,
	\quad \quad
	& \pR &= \coassQ_1 \tensor \coassQ_2 \betaQ S(\coassQ_3),
	\notag \\
	\qL &= S(\coassQ_1)\alphaQ \coassQ_2 \tensor \coassQ_3,
	& \pL &= \invCoassQ_2 S\inv(\invCoassQ_1 \betaQ) \tensor \invCoassQ_3
	\label{eq:qRpR}
	\ ,
\end{alignat}
which satisfy a number of useful identities, see  e.g.~\cite[Sec.\,3.4]{BGR2} for a
list in our conventions and \cite[Sec.\,3]{BGR1} for a representation in graphical
calculus.
With these, one may describe the \emph{Drinfeld twist} $\Dt$ of $H$, i.e.\ the
invertible element in $H \tensor H$ implementing the natural isomorphism $\dualL{W}
\tensor \dualL{V} \cong \dualL{(V \tensor W)}$ in $\hmodM$, see
e.g.~\cite[Sec.\,3.4]{BGR2} for an explicit expression.

With the application to link and three-manifold invariants in mind, we will in the
following restrict ourselves to the case that $H$ is unimodular.
In that case, a linear form $\coint^r$ on $H$ is a right cointegral, resp.\ $\coint^l$
a left cointegral, if and only if it satisfies
\begin{align}\label{eq:coint-def}
	(\coint^r \tensor \id) \big( V\cop \Delta(h) U\cop \big) = \coint^r(h) \oneQ
	\qtextq{resp.}
	(\id \tensor \coint^l) \big( V\Delta(h) U \big) = \coint^l(h) \oneQ
\end{align}
for all $h \in H$, see \cite{HN-integrals,BC1,BC2} and \cite[Def.\,3.2]{BGR1} for a
definition in our conventions.
The elements $U^{\textup{(cop)}}, V^{\textup{(cop)}} \in H
\tensor H$ are defined as
\begin{alignat*}{3}
	U &= \Dt\inv (S \tensor S)(\qR_{21}),
	\quad & V &= (S\inv \tensor S\inv)(\Dt_{21} \pR_{21}),
	\\
	U\cop &= (S\inv \tensor S\inv)(\qL_{21} \Dt\inv_{21}),
	\quad & V\cop &= (S \tensor S)(\pL_{21}) \Dt_{21}
	\ .
\end{alignat*}

\subsubsection{Pivotal structure, braiding, and ribbon structure on $\hmodM$}

Here, we briefly review results on pivotal structures, braidings and ribbon twists. For more details we refer to \cite{EGNO} for the categorical structures, and to \cite{Bulacu_book} for their quasi-Hopf algebra counterparts.

Recall that a \emph{pivotal structure} on a rigid monoidal category is a monoidal
natural isomorphism $\id_{\cat} \To \ddualL{(\placeholder)}$.
A \emph{pivot} for $H$ is an invertible element $\pivotQ \in H$ satisfying
$S^2(h) = \pivotQ h \pivotQ\inv$
and
$\Delta(\pivotQ) = \Dt\inv \cdot (S \tensor S)(\Dt_{21}) \cdot (\pivotQ \tensor
\pivotQ)$.
A \emph{pivotal quasi-Hopf algebra} is a quasi-Hopf algebra together with chosen pivot.

\begin{proposition}
	\label{prop:qHA:pivot}
	There is a bijective correspondence 
 	\begin{align*}
		\renewcommand{\arraystretch}{0.9}
		\left\{
		\begin{matrix}
			\text{pivotal structures}
			\\
			\text{on } \hmodM
		\end{matrix}
		\right\}
		\xrightarrow{~\sim~}
		\left\{
		\begin{matrix}
			\text{pivots}
			\\
			\text{for } H
		\end{matrix}
		\right\}
		, \quad 
		\pivotalStruct \longmapsto 
(\pivotalStruct^\Vect_{H})\inv(\pivotalStruct_H(\oneQ))
		, \
	\end{align*}
	where $\pivotalStruct^\Vect_V$ is the canonical isomorphism $V \cong V^{**}$ of
	vector spaces.
The inverse maps $\pivotQ$ to the pivotal structure $\pivotalStruct_V(v) = \pivotalStruct^\Vect_V(\pivotQ . v)$.
	A pivot satisfies $\counit(\pivotQ) = 1$ and $S(\pivotQ) =
	\pivotQ\inv$.
\end{proposition}
In the pivotal case we will choose the (co)evaluation maps in $\hmodM$ such that they are related by the pivotal structure. Our convention is to define the two-sided dual to be $\dualL{V}$. Accordingly we take the left (co)evaluation map from \eqref{eq:qHopf-ev}, \eqref{eq:qHopf-coev} and we write
\begin{align}
	\evL_V := \evL^\mathrm{L}_V
	\quad , \quad 
	\coevL_V := \coevL^\mathrm{L}_V ~.
	\label{eq:pivot-qHopf-ev}
\end{align}
The right-duality on $\dualL{V}$ is given by $\evR_V := \evL^\mathrm{L}_{\dualL{V}} \circ (\pivotalStruct_V \otimes \id_{\dualL{V}})$ and similar for $\coevR_V$. Explicitly,
\begin{align}
	\evR_V(v \tensor f) 
	= \langle f \mid S(\alphaQ)\pivotQ . v \rangle
	\quad , \quad 
	\coevR_V(1) = \sum_{i=1}^{\dim V} v^i \tensor S\inv(\pivotQ\betaQ) . v_i
	~.
	\label{eq:pivot-qHopf-coev}
\end{align}

An \emph{$R$-matrix} for $H$ is an invertible element $\rMatrix \in H\tensor H$
satisfying
$(\Delta \tensor \id) (\rMatrix) 
= \invCoassQ_{231} \rMatrix_{13} \coassQ_{132} \rMatrix_{23} \invCoassQ$, 
$(\id \tensor \Delta) (\rMatrix) 
= \coassQ_{312} \rMatrix_{13} \invCoassQ_{213} R_{12} \coassQ$
and $\rMatrix \Delta(h) \rMatrix\inv = \Delta\cop(h)$.
A \emph{quasi-triangular quasi-Hopf algebra} is a quasi-Hopf algebra together with a
chosen $R$-matrix.

\begin{proposition}
	\label{prop:qHA:braiding}
	There is a bijective correspondence 
	\begin{align*}
		\renewcommand{\arraystretch}{0.9}
		\left\{
		\begin{matrix}
			\text{braidings} \\ \text{on } \hmodM
		\end{matrix}
		\right\}
		\xrightarrow{\sim}
		\left\{
		\begin{matrix}
			R\text{-matrices} \\ \text{for } H
		\end{matrix}
		\right\}
		, \quad
		\braiding \longmapsto \tau \circ \braiding_{H, H}(\oneQ \tensor \oneQ)
	\end{align*}
	where  $\tau$ is the flip map in $\Vect$.
The inverse maps $\rMatrix$ to the braiding
$\braiding_{X, Y}(x \tensor y)
= \rMatrix_2 . y \tensor \rMatrix_1 . x$.
\end{proposition}

Associated to an $R$-matrix is also its monodromy $\monodromy = \rMatrix_{21}
\rMatrix$, which encodes the double braiding in $\hmodM$.
Lastly, since $\hmodM$ is rigid, there is a canonical (but in general non-monoidal)
natural isomorphism $\id_{\cat} \To \ddualL{(\placeholder)}$ built using left
dualities and a braiding;
it is completely determined by an element $\drinfeldElement \in H$, called the
\emph{Drinfeld element} of $H$, see e.g.\ \cite[Sec.\,6.3]{FGR1} for an explicit
expression in the current conventions.

Finally, recall that a \emph{ribbon twist} on a braided rigid monoidal category is a
natural automorphism of the identity functor, satisfying certain axioms.
A \emph{ribbon element} for $H$ is a non-zero central element $\elRibbon \in H$
satisfying
$\Delta(\ribbon) = \monodromy\inv \cdot \ribbon \tensor \ribbon$
and $S(\ribbon) = \ribbon$.
A \emph{ribbon quasi-Hopf algebra} is a quasi-Hopf algebra together with a chosen
ribbon element.

\begin{proposition}[\cite{Sommerhaeuser}]
	A ribbon element is automatically invertible and satisfies $\counit(\ribbon) = 1$,
	and there is a bijective correspondence
	\begin{align*}
		\renewcommand{\arraystretch}{0.9}
		\left\{
		\begin{matrix}
			\text{ribbon twists} \\ \text{on } \hmodM
		\end{matrix}
		\right\}
		\xrightarrow{\sim}
		\left\{
		\begin{matrix}
			\text{ribbon elements} \\ \text{for } H
		\end{matrix}
		\right\}
		, \quad
		\ribTwist \mapsto (\ribTwist_H)\inv(\oneQ) \ .
	\end{align*}
 The inverse maps $\elRibbon$ to the ribbon twist $\ribTwist_V(v) = \elRibbon\inv . v$.
\end{proposition}

Since a ribbon category is pivotal, a ribbon Hopf algebra admits a canonical pivot
which is compatible with the ribbon structure.
In terms of the ribbon and Drinfeld element, it is explicitly given by $\pivotQ
=\drinfeldElement \ribbon\inv$.

\subsection{Modified traces on \texorpdfstring{$\hmodM$}{hM} and
\texorpdfstring{$\ProjIdeal[\hmodM]$}{Proj(hM)}}

Let $(H, \pivotQ)$ be a pivotal quasi-Hopf algebra.

\subsubsection{The categorical trace}
From \eqref{eq:pivot-qHopf-ev} and \eqref{eq:pivot-qHopf-coev} one finds that the right
resp.\ left categorical trace of $\cat = \hmodM$ is given by
\begin{align}
	\trCat[r, \cat]_V(f) = \tr_V \left( \pivotQ \betaQ S(\alphaQ) \circ f \right)
	\qtextq{resp.}
	\trCat[l, \cat]_V(f) 
	= \tr_V \left( \alphaQ S\inv(\pivotQ \betaQ) \circ f \right)
	\label{eq:categorical_trace_qHopf}
\end{align}
for $f \in \End_H(V)$, where by $h = \pivotQ \betaQ S(\alphaQ)$ we mean the linear
endomorphism of $V$ given by acting with $h$, and $\tr$ on the right hand sides is the
usual trace of linear maps. If $H$ is ribbon then  both the categorical traces agree.

\subsubsection{The modified trace on the projective ideal}
Next, consider $\tensIdeal = \ProjIdeal[{\hmodM}]$, the subcategory of projective
$H$-modules.
Since non-degenerate modified traces do not exist in the non-unimodular
setting~\cite{FOG,GKP_m-traces}, we now restrict to the case that $H$ is unimodular.
Building on \cite{BBGa}, in \cite{BGR1} it was shown that non-degenerate right (resp.\
left) modified traces on $\tensIdeal$ are in bijection with linear forms
$\symRightCoint \in H^*$ (resp.\ $\symLeftCoint \in H^*$) satisfying
\begin{align}
	(\symRightCoint \tensor \id) \big( \qR \Delta(h) \pR \big) = \symRightCoint(h)
	\pivotQ\inv
	\qtextq{resp.}
	(\id \tensor \symLeftCoint) \big( \qL \Delta(h) \pL \big) = \symLeftCoint(h) \pivotQ
	\label{eq:symm_coint_equation}
\end{align}
for all $h \in H$.
Moreover, the unique (up to scalar)
solutions to these equations are given by
\begin{align}
	\label{eq:sum-int-from-int}
	\symRightCoint(h) = \coint^r ( \pivotQ h )
	\qtextq{resp.}
	\symLeftCoint(h) = \coint^l ( \pivotQ\inv h )
	, \quad h \in H
	\ ,
\end{align}
where $\coint^r$ (resp.\ $\coint^l$) is a non-zero right (resp.\ left) cointegral of
$H$ defined via~\eqref{eq:coint-def}.
Such a solution is then called a \emph{right (resp.\ left) symmetrised cointegral} for
$H$, the name reflecting the fact that it is
a symmetric form associated to a
cointegral.

Let us describe explicitly the right modified trace associated to the right symmetrised
cointegral $\symRightCoint$.
By linearity of the modified trace it is enough to consider endomorphisms of
indecomposable projective modules, or more generally, any direct summand of the regular $H$-module $H$.
So let $P \in \hmodM$ be a direct summand in $H$ via the split idempotent $H
\xrightarrow{p} P \xrightarrow{i} H$.
The associated right modified trace of $f \in \End_H(P)$ is given by
\begin{align}
	\label{eq:modified_trace_on_indecomp_proj}
	\modTr^r_P(f) 
	= \symRightCoint\big((i \circ f \circ p)(\oneQ)\big)
	= \coint^r\big(\pivotQ \cdot (i \circ f \circ p)(\oneQ)\big)
	\ ,
\end{align}
and is independent of the choice of $i$ and $p$ with $p \circ i = \id_P$, see \cite[Prop.\,4.3\,\& Thm.\,4.5]{BGR1}.
Replacing $\symRightCoint$ by the left symmetrised cointegral $\symLeftCoint$ 
yields the left modified trace.
Note that if $H$ is ribbon, then there is no difference between left and right modified
traces, and so the left and right symmetrised cointegrals agree.
The corresponding left and right cointegrals do not have to agree (but they do e.g.\ if
$\pivotQ$ has order 2).

A special case of \eqref{eq:modified_trace_on_indecomp_proj}, which will often appear
later, is when $f$ is a component of a natural transformation $\xi \in \End(\id_{\cat})$.
It is well-known that $\End(\id_{\cat})$ is
isomorphic to the centre
$Z(H)$ of $H$;
in fact, $\xi$ is given by the action with $\hat{\xi} = \xi_H(\oneQ) \in Z(H)$.
Then we have $\modTr^r_P(\xi_P) = \coint^r(\pivotQ \hat{\xi} e)$ where $e = (i \circ
p)(\oneQ)$ is the idempotent of the direct summand $P$ in $H$.

\subsection{Lyubashenko coend and integral}\label{sec:Lyu-coend-int-for-qHopf}

Let $H$ be a quasi-triangular quasi-Hopf algebra, so that $\hmodM$ is braided.
The Hopf algebra $\coend$ (as recalled in \Cref{sec:twist-non-deg_etc}) exists
and is given by the coadjoint representation --- i.e.\ $\coend \cong H^*$ as vector spaces,
with action given by $(h.f)(v) = f(S(h\sweedler{1}) v h\sweedler{2})$ for $f \in H^*$ and 
$h,v \in H$ --- together with the universal dinatural transformation with components
\begin{align}
	\dinatLyu_X(x^* \tensor y)(h) = x^*(h y)
 \label{eq:dinatCoend_qHopf}
\end{align}
for $X \in \hmodM$, $x^* \in \dualL{X}$, $y \in X$, and $h \in H$.
See \cite[Sec.\,7]{FGR1} for more details.

Let now $H$ be unimodular.
This is equivalent to $\hmodM$ being unimodular, and so $\coend$ admits a two-sided integral $\intLyu : \tensUnit\to\coend$, which in the present case translates into $\intLyu \in H^*$.
It was shown in \cite[Cor.\,6.4]{BGR2} that given a non-zero right cointegral $\coint^r$ for $H$,
a non-zero integral $\intLyu$ for $\coend$
is obtained as  $\intLyu(h) = \coint^r(S(\betaQ) h)$.
This is in fact a bijection with inverse 
\begin{align}
	\label{eq:H-int-from-Lam-int}
	\coint^r(h) = \intLyu(S(\qL_1) h \qL_2) \ .
\end{align}

The condition \eqref{eq:ribboncat-twistnondeg} for twist-nongedeneracy of $\hmodM$ can be expressed in terms of quasi-Hopf data as follows. By \cite[Eqn.\,(8.3)]{FGR1}, the endomorphism $\modT_\coend : \coend \to \coend$ from \eqref{eq:T_L-def} acts on $f \in \coend = H^*$ as $(\modT_\coend(f))(a) = f( \elRibbon\inv a)$, where $a \in H$. The counit $\varepsilon_\coend : \coend \to \tensUnit$ is given by $\varepsilon_\coend(f) = f(\alphaQ)$ for $f \in H^*$, see \cite[Thm.\,7.3]{FGR1}. Combining this with the above expression for the integral $\intLyu$, we obtain
\begin{align}\label{eq:qHopf-twistnondeg}
    \hmodM \text{ twist-nondegenerate }
    \quad
    \Leftrightarrow
    \quad
    \stabCoeff =
    \coint^r(S(\betaQ) \elRibbon^{\mp1} \alphaQ) \neq 0 ~.
\end{align}
We call quasi-Hopf algebras satisfying the non-vanishing condition on the right hand side of~\eqref{eq:qHopf-twistnondeg} \textit{twist-nondegenerate}.

Recall from \Cref{rem:factorisable-cat} the condition for a unimodular braided finite tensor category to be factorisable. We call the unimodular quasi-triangular quasi-Hopf algebra $H$ \textit{factorisable} if $\hmodM$ is factorisable. A characterisation in terms of quasi-Hopf data was given in \cite[Sec.\,7.3]{FGR1} and agrees with the original definition in \cite{BT} as shown in \cite[Sec.\,7.4]{FGR1}.

\smallskip

Next we explain that two a priori different looking ways of normalising the modified trace on the projective ideal in terms of the integral $\intLyu$ actually agree.
We start by fixing an arbitrary non-zero integral $\intLyu : \tensUnit\to\coend$. In case $H$ is modular, one can further fix $\intLyu$ up to a sign by demanding \Cref{rem:modular_actions}\,(1), but here we do not assume this.

From \eqref{eq:H-int-from-Lam-int} and \eqref{eq:sum-int-from-int}, one obtains a (two-sided) symmetrised cointegral by setting
$\symRightCoint(h) = \intLyu(S(\qL_1) \pivotQ h \qL_2)$. Via \eqref{eq:modified_trace_on_indecomp_proj} this in turn fixes a modified trace $\modTr^\mathrm{Hopf}$ on the projective ideal. Altogether, we have determined $\modTr^\mathrm{Hopf}$ from the initial choice of $\intLyu$.
This Hopf algebraic normalisation condition has a categorical counterpart:

\begin{remark}
	Assume for this remark that $\cat$ is modular. 
	Given $\intLyu$, there is a unique cointegral $\cointLyu$ for $\coend$ such that $\cointLyu \circ \intLyu = \id_{\tensUnit}$, cf.~\eqref{eq:Lam-co-norm}. 
	Fix a surjection (unique up to scalar) 
	$\projCoverMorphism{\tensUnit} \colon \projCover{\tensUnit} \to \tensUnit$.
	By \cite[Prop.~6.5]{GR-proj}, this yields a unique morphism $\injHullMorphism{\tensUnit}
	\colon \tensUnit \to \projCover{\tensUnit}$ such that
	\begin{align}
		\cointLyu \circ \dinatLyu_{\projCover{\tensUnit}}
		= \ev_{\projCover{\tensUnit}} \circ
		\big( \id_{\dualL{\projCover{\tensUnit}}}
		\tensor (\injHullMorphism{\tensUnit} \circ \projCoverMorphism{\tensUnit})\big)
		\ .
		\label{eq:cond_inj_hull_morph}
	\end{align}
	Note that the composition $\injHullMorphism{\tensUnit} \circ \projCoverMorphism{\tensUnit}$ is independent of the choice of $\projCoverMorphism{\tensUnit}$. From \cite[Thm.\,6.1\,\&\,Prop.\,6.5]{GR-proj} it follows that for a non-zero modified trace $\modTr$ on $\ProjIdeal$ one has $\modTr_{\projCover{\tensUnit}}(\injHullMorphism{\tensUnit} \circ \projCoverMorphism{\tensUnit}) \neq 0$. We may thus normalise $\modTr$ such that
	\begin{align}\label{eq:cat-trace-norm}
		\modTr_{\projCover{\tensUnit}}(\injHullMorphism{\tensUnit} \circ \projCoverMorphism{\tensUnit}) = 1 \ .
	\end{align}
	This is the same condition as in \Cref{rem:modular_actions}\,(4), again by  \cite[Prop.\,6.5]{GR-proj}. Altogether, starting from $\intLyu$ we obtained a unique choice of modified trace $\modTr$ on $\ProjIdeal$.
\end{remark}

Let us return to the unimodular quasi-triangular quasi-Hopf algebra $H$. We have:

\begin{lemma}\label{lem:tHopf-normal}
	$\modTr^\mathrm{Hopf}$ satisfies \eqref{eq:cat-trace-norm}.
\end{lemma}

Note that this holds without assuming that $\hmodM$ is modular.

\begin{proof}
	Denote by $i$ and $\pi$ 
	a choice of morphisms in $\hmodM$ which give 
	the injection and projection of $\projCover{\tensUnit}$ in
	$H$. The map $i$ is unique up to scalar as $\projCover{\tensUnit}$ is also
an injective hull of $\tensUnit$
and the space of $H$-intertwiners from $\tensUnit$ to ${}_HH$ is the space of left integrals and hence one-dimensional.
	Fix the
	morphism $\projCoverMorphism{\tensUnit} = \counit \circ i \colon \projCover{\tensUnit}
	\to \tensUnit$ and let $\injHullMorphism{\tensUnit}$ be determined by \eqref{eq:cond_inj_hull_morph}. To describe $\injHullMorphism{\tensUnit}$ more explicitly, first note that by \cite[Prop.\,7.8]{FGR1} there is a unique choice of integral $c \in H$ of $H$ such that the cointegral $\cointLyu \in H^{**}$ of $\coend$ satisfies $\cointLyu = \langle - | c \rangle$.
	A short calculation using the normalisation $\counit(\alphaQ) = 1$
 shows that 
 $\injHullMorphism{\tensUnit} = \pi(c)$ 
    (viewing $\pi(c) \in \projCover{\tensUnit}$ as a linear map $k \to \projCover{\tensUnit}$) 
 satisfies \eqref{eq:cond_inj_hull_morph}.
	Note further that $\symRightCoint(c) = \intLyu(c) = \cointLyu \circ \intLyu = 1$.
	Altogether we have (here we write $\oneQ_H \in H$ for the unit of $H$ to distinguish it from the tensor unit $\tensUnit$ of $\hmodM$)
	\begin{align*}
		\modTr^\mathrm{Hopf}_{\projCover{\tensUnit}}
		(\injHullMorphism{\tensUnit} \circ \projCoverMorphism{\tensUnit})
		=
		\symRightCoint(
		(i \circ \injHullMorphism{\tensUnit}
		\circ \projCoverMorphism{\tensUnit} \circ \pi)
		(\oneQ_H)
		)
		=
		\counit(e_{\tensUnit})
		\symRightCoint(c)
		= 1 \ ,
	\end{align*}
	where $e_{\tensUnit} = (i \circ \pi)(\oneQ_H)$.
	The last equality follows from the fact that the counit is only non-zero on the
	idempotent corresponding to the tensor unit in an isotypic decomposition of the regular
	module, which implies $1 = \counit(\oneQ_H) 
 = \counit(e_{\tensUnit})$.
\end{proof}

\subsection{Pullback ideals from quasi-Hopf subalgebras}
\label{sec:pullback-qHopf}
For any quasi-Hopf algebra $H$, a linear subspace $A \subset H$ is a \emph{quasi-Hopf
	subalgebra} if it is a (unital) subalgebra,
such that $\Delta(A) \subset A\tensor A$ and $S(A)
\subset A$, and which moreover contains the coassociator in the sense that $\coassQ \in
A^{\tensor 3}$.
Denote by
\begin{align}
	\restrictionFunctor = 
	\restrictionFunctorExplicit{H}{A} \colon \hmodM \to \hmodM[A]
	\notag
\end{align}
the canonical (forgetful) restriction functor, and note that it is strict monoidal.
We say that an $A$-module $M$ \emph{lifts} to the $H$-module $N$ if
$\restrictionFunctor N \cong M$ as $A$-modules.

Given a quasi-Hopf subalgebra $A \subset H$, we may now consider the category of
\emph{$A$-projective $H$-modules}.
This is the full subcategory $\relativeProjectives{A}{H}$ of $\hmodM$ consisting of
$H$-modules which are projective as $A$-modules, i.e.\ its objects are precisely the
lifts of projective $A$-modules.
In other words, it is a pullback in the sense of \Cref{prop:pullback_ideal},
\begin{align*}
	\relativeProjectives{A}{H}
	= 
	\restrictionFunctor^\ast \big( \ProjIdeal[{\hmodM[A]}] \big)
	\ ,
\end{align*}
and so in particular, $\relativeProjectives{A}{H}$ is a tensor ideal in $\hmodM$.

Since every ideal contains the projective ideal (cf.~\Cref{prop:proj_is_smallest}), we
have the following chain of ideals\footnotemark
\begin{align}
	\ProjIdeal[\hmodM]
	~ \subset ~
	\relativeProjectives{A}{H}
	~ \subset ~
	\hmodM
	\ .
	\label{eq:chain_of_ideals}
\end{align}
We also remark that the two extremes in \eqref{eq:chain_of_ideals} are precisely the
$H$-projective and the $\field$-projective $H$-modules.
\footnotetext{
	To see this without \Cref{prop:proj_is_smallest}, recall that $H$ is free as an
	$A$-module by the Nichols-Zoeller theorem for quasi-Hopf algebras
	\cite{Schauenburg_freeness}.
	In particular, every projective indecomposable $H$-module restricts to a direct
	summand of a free $A$-module;
	since $\ProjIdeal[\hmodM]$ is generated by indecomposable projectives, we get the
	inclusion of ideals as in \eqref{eq:chain_of_ideals}.
}

\smallskip

Let now $A \subset H$ be a unimodular quasi-Hopf subalgebra of the pivotal quasi-Hopf
algebra $(H,\pivotQ)$.
If $\pivotQ \in A$, then $(A,\pivotQ)$ is a unimodular pivotal quasi-Hopf algebra,
and thus admits a non-degenerate right modified trace $\modTr^A$ on
$\ProjIdeal[{\hmodM[A]}]$, constructed via its right symmetrised cointegral.
The same holds for the left version.
As a further application of \Cref{prop:pullback_ideal} we now get a modified trace on
$\relativeProjectives{A}{H}$ by pulling back $\modTr^A$ along the restriction functor.
Together with \eqref{eq:modified_trace_on_indecomp_proj}, we obtain
a quasi-Hopf generalisation of~\cite[Cor.\,2.8]{FOG}:

\begin{corollary}
	\label{prop:pullback_modTr_along_restriction_qHopf}
	Abbreviate $I(A) = \relativeProjectives{A}{H}$, and let 
	$\widehat\coint{}^A \in A^*$ 
	be a right
	(left) symmetrised cointegral for $A$, with corresponding right (left) modified trace
	$\modTr^A$ on $\ProjIdeal[{\hmodM[A]}]$.
 For $M\in I(A)$,
	the family of linear maps
	\begin{align*}
		(\pullbackTrace[{\modTr^A}])_M : \End_{I(A)}(M) \to \field
		,\qquad
		(\pullbackTrace[{\modTr^A}])_M (f) 
		=
		\modTr^A_{\restrictionFunctor M}(f)
		\ ,
	\end{align*}
	defines a right (left) modified trace on $I(A)$.
	Explicitly,
    with $\oneQ_A$ the unit of $A$,
	\begin{align}
		(\pullbackTrace[{\modTr^A}])_M (f) 
    = \sum_{i=1}^n \widehat\coint^A \big( (p_i \circ f \circ j_i)(\oneQ_A) 
    \big)
		\ ,
		\notag
	\end{align}
	where $n$ is a natural number and $p_i \colon \restrictionFunctor{M} \to A$, $j_i
	\colon A \to \restrictionFunctor{M}$ are $A$-module morphisms, for $i=1,...,n$, such
	that $\sum_{i=1}^n j_i \circ p_i = \id_M$ and $p_l \circ j_m = \delta_{l,m} p_l \circ
	j_l$.
\end{corollary}

\begin{remark}
  Suppose $H$ is ribbon.
  In light of \Cref{rem:modified_traces_in_ribbon_cats}, the pullback modified trace
  from \Cref{prop:pullback_modTr_along_restriction_qHopf} is then always two-sided,
  even if $\ProjIdeal[{\hmodM[A]}]$ has distinct right and left modified traces
  $\modTr^{A,r}$ and $\modTr^{A,l}$.
  In the example we consider below, $\modTr^{A,r}$ and $\modTr^{A,l}$ are not proportional, but the two pullback traces still turn out to agree up to a scalar (\Cref{prop:symCoint_for_subqHA} and \Cref{cor:pullback-tr-proportional}). 
\end{remark}

\begin{remark}\label{rem:tens-ideal-comodule-alg}
Every finite-dimensional right $H$-comodule algebra $A$ gives a right $\cat$-module category  $\hmodM[A]$ for $\cat=\hmodM[H]$ with the action functor $\rightact$ exact in both variables.
In particular, this holds for a right coideal subalgebra $I\subset H$.
By \Cref{rem:tens-ideal-module}, such an $A$ and a left $A$-module $M$ provide a $\cat$-linear functor $F_M\colon \hmodM \to {}_A\mathcal{M}$, $X \mapsto M \otimes X$.
The pullback $F_M^* \ProjIdeal[{\hmodM[A]}]$ therefore is a right tensor ideal in $\hmod$.
However, less is  known about existence of traces on $\ProjIdeal[{\hmodM[A]}]$, as compared to pivotal Hopf subalgebras treated in \Cref{prop:pullback_modTr_along_restriction_qHopf}. For example, we are not aware of analogues of symmetrised right cointegral expression, as in~\cite{BBGa,FOG,BGR1}.
\end{remark}

\section{The symplectic fermion quasi-Hopf algebras}
\label{sec:qHopf_applications}

After presenting some general aspects of the theory of quasi-Hopf algebras, we now turn to the family of examples which we will use to evaluate link and three-manifold invariants. These are the symplectic fermion quasi-Hopf algebras, which in fact are ribbon and factorisable. They were constructed in
\cite{Gainutdinov:2015lja,FGR2}
to reproduce the ribbon finite tensor categories obtained in \cite{Davydov:2012xg,Runkel:2012cf} from the study of the chiral two-dimensional logarithmic 
conformal field theory of symplectic fermions 
\cite{Kausch:1995py,Gaberdiel:1996np,Abe:2005}.

\subsection{Quasi-Hopf algebra and ribbon structure}
\label{sec:symplectic_fermions_qHA}

Here we define the family $\sympFerm(N, \beta)$ of \emph{symplectic fermion quasi-Hopf algebras},
give the ribbon structure on $\sympFerm(N, \beta)$ and
compute its (co)integrals.
The two parameters are a natural number $N \geq1$ and a complex number $\beta$  s.t.\ 
$\beta^4 = (-1)^N$. We fix such an $N$ and $\beta$ and abbreviate $\sympFerm = \sympFerm(N, \beta)$.

\subsubsection{Quasi-Hopf algebra structure}\label{sec:SF-qHopf-structure}
As an algebra, $\sympFerm$ is the (unital associative) $\field[C]$-algebra generated by
$\{\,\genK, \genF_i^\epsilon \ |\ 1\leq i \leq N,~ \epsilon= \pm \,\}$.
With the elements $\eQ_0  = \tfrac{1}{2}(\oneQ + \genK^2)$ and $\eQ_1  =
\tfrac{1}{2}(\oneQ - \genK^2)$, the relations of $\sympFerm$ are
\begin{align}
	\{\genF^{\pm}_i,\genK \} = 0\ ,
	\quad
	\{\genF^+_i, \genF^-_j \} = \delta_{i,j} \eQ_1\ ,
	\quad
	\{\genF^\pm_i, \genF^\pm_j \} = 0\ ,
	\quad
	\genK^4 = \oneQ\ ,
	\notag
\end{align}
for all $1\leq i,j\leq N$, and
where $\{-,-\}$ is the anticommutator. 
Then $\eQ_0, \eQ_1$ satisfy $\eQ_0 + \eQ_1 = \oneQ$, and are central orthogonal
idempotents.
A basis is given by
\begin{align}
	\{
	\mathfrak{b}(m, s, t) = 
	\genK^m \prod_{j=1}^N (\genF^+_j)^{s_j} (\genF^-_j)^{t_j}
	\mid m \in \field[Z]_4, s, t \in \field[Z]_2^N
	\}
	\ ,
	\label{eq:basis_sympFerm}
\end{align}
and the dimension of $\sympFerm$ is $2^{2N+2}$.
On generators, the coproduct and counit are
\begin{alignat}{2}
	\Delta(\genK) 
	&= \genK \otimes \genK - (1+(-1)^N)\ \eQ_1 \genK \otimes \eQ_1 \genK \ ,
	\qquad
	&& \counit(\genK) = 1 \ ,
	\notag\\ 
	\Delta(\genF^\pm_i)
	&= \genF^\pm_i \otimes \oneQ + \omega_\pm \otimes \genF^\pm_i \ ,
	&&
	\counit(\genF^\pm_i) = 0 \ ,
	\label{eq:sympFermCoprod}
\end{alignat}
where $\omega_\pm = (\eQ_0 \pm i \eQ_1) \genK$. 
With $\betaQ_\pm = \eQ_0 + \beta^2 (\pm i\genK)^N \eQ_1$, the coassociator and its
inverse are
\begin{align}
	\coassQ^{\pm 1} = 
	\oneQ \otimes \oneQ \otimes \oneQ
	+ \eQ_1 \otimes \eQ_1 \otimes 
	\left\{
	\eQ_0(\genK^N - \oneQ)
	+ \eQ_1(\betaQ_\pm  - \oneQ) 
	\right\}\ .
	\label{eq:coassociator_SF}
\end{align}
Finally, the antipode $S$ and the evaluation and coevaluation elements 
$\alphaQ$ and $\betaQ$ are given by
\begin{alignat}{2}
	S(\genK) &= \genK^{(-1)^N} = (\eQ_0 + (-1)^N \eQ_1)\genK \ ,
	\qquad\qquad && \alphaQ = \oneQ \ ,
	\notag\\ 
	S(\genF^\pm_i) &= \genF^\pm_i (\eQ_0 \pm (-1)^N i\eQ_1) \genK \ , 
	&& \betaQ = \betaQ_+\ .
	\notag
\end{alignat}
The inverse of the antipode is found to be
\begin{align}
	S\inv(\genK) = \genK^{(-1)^N} \ , \qquad
	S\inv(\genF^\pm_i) = \omega_{\pm} \genF^\pm_i\ .
	\notag
\end{align}
Note that $S(\betaQ_\pm)=S\inv(\betaQ_\pm)=\betaQ_\mp$, and $\betaQ_+\betaQ_-=\oneQ$.

The algebra $\sympFerm$ is the direct sum of two ideals $\sympFerm_i = \eQ_i
\sympFerm$, $i=0, 1$.
As algebras, $\sympFerm_0$ (resp.\ $\sympFerm_1$) is the semidirect product of a Grassmann algebra (resp.\ Clifford algebra) in $2N$ generators with the group algebra of $\field[Z]_2$. In particular, $\sympFerm_1$ is semisimple.
This decomposition imparts a $\field[Z]_2$-grading on the category of $\sympFerm$-modules:
$\hmodM[\sympFerm] = (\hmodM[\sympFerm])_0 \oplus (\hmodM[\sympFerm])_1$.
Since $\sympFerm_1$ is semisimple, $(\hmodM[\sympFerm])_1$ is semisimple as an abelian (sub)category.
The monoidal structure on $\hmodM[\sympFerm]$ respects the grading, see~\cite[Sec.\,3]{FGR2} for details.

\subsubsection{Further elements and structure}
\label{sec:qHopf:sympf_ferm_elements}
In \cite[Sec.\,5]{BGR1}, the canonical elements from \eqref{eq:qRpR} were computed to be
\begin{align}
	\qR &= \oneQ \tensor \oneQ + \eQ_1 \tensor \big( \eQ_1 (\betaQ - \oneQ) \big) \ ,
	\notag \\
	\pR &= \oneQ \tensor \oneQ + \eQ_0 \tensor \big( \eQ_1 (\betaQ - \oneQ) \big) \ ,
	\notag \\
	\qL &= \oneQ \tensor \oneQ + \eQ_1 \tensor 
	\big\{
	\eQ_0 (\genK^N - \oneQ) + \eQ_1 (\betaQ - \oneQ)
	\big\} \ ,
	\notag \\
	\pL &= \betaQ_- \tensor \oneQ + \eQ_1 \betaQ_- \tensor 
	\big\{
	\eQ_0 (\genK^N - \oneQ) + \eQ_1 (\betaQ_- - \oneQ)
	\big\}
	\ .
	\label{eq:qRpR_SF}
\end{align}
The Drinfeld twist and its inverse are given by
\begin{align}
	\Dt^{\pm 1} = \eQ_0 \tensor \oneQ + \eQ_1 \tensor \eQ_0 \genK^N 
	+ \eQ_1 \betaQ_{\mp} \tensor \eQ_1
	\ .
\end{align}
$\sympFerm$ is quasi-triangular.
The $R$-matrix and its inverse are
\begin{align*}
	\rMatrix = \big( \sum_{n,m = 0}^1 \beta^{nm} \rho_{n,m} \eQ_n \tensor \eQ_m \big)
	\cdot \prod_{j=1}^N (\oneQ \tensor \oneQ - 2 \genF^-_j \omega_- \tensor \genF^+_j)
\end{align*}
and
\begin{align*}
	\rMatrix\inv = 
	\prod_{j=1}^N (\oneQ \tensor \oneQ + 2 \genF^-_j \omega_- \tensor \genF^+_j)
	\cdot \big( \sum_{n,m = 0}^1 \beta^{-nm} \rho_{n,m} \eQ_n \tensor \eQ_m \big)
	\ ,
\end{align*}
where
\begin{align*}
	\rho_{n,m} 
	= \frac{1}{2} \sum_{k,l=0}^1 (-1)^{kl} i^{-kn + lm} \genK^k \tensor \genK^l
	\ .
\end{align*}
Its monodromy may be described as 
\begin{equation}\label{eq:M}
    \monodromy = \rMatrix_{21}
\rMatrix = \sum_{I} g_I \tensor f_I\ , \qquad
    \text{for}\; I = (a, b, c, d)
\end{equation}
with $a, b \in \field[Z]_2$ and $c, d \in \field[Z]_2^N$, and where 
\begin{align}
	f_I &= 
	\genK^a \eQ_b \prod_{k=1}^N (\genF^-_k)^{d_k} (\genF^+_k)^{c_k} \ .
	\notag \\
	g_I &=
	(- \beta^2)^{ab} 2^{|c| + |d|}
	(-1)^{|c| + b |d|}
	\genK^b \eQ_a \prod_{k=1}^N (\genF^+_k \omega_-)^{d_k} (\genF^-_k \omega_-)^{c_k} \ ,
	\label{eq:symp_ferm_monodromy_basis}
\end{align}
see \cite[Lem.~4.2]{FGR2}.
This is used there to further show that  
the Hopf pairing $\hopfPair$ in \eqref{eq:Hopf-pairing} is non-degenerate, 
i.e.\ that $\hmodM[\sympFerm]$ is factorisable.

Finally, $\sympFerm$ is ribbon with (inverse) ribbon element
\begin{align}
	\elRibbon^{\pm 1}
	&= (\eQ_0 - \beta^{\pm 1} i \genK \eQ_1) 
	\prod_{j=1}^N \big(\oneQ \mp 2 (\eQ_0 \pm \eQ_1) \genF^+_j \genF^-_j \big)
	\label{eq:symp_ferm_ribbon}
\end{align}
and so $\hmodM[\sympFerm]$ 
is a modular tensor category.
The pivot compatible with the ribbon structure is
\begin{align}
	\label{eq:SF_pivot}
	\pivotQ = (\eQ_0 + (-i)^{N+1} \eQ_1 \genK^N) \genK
	\ .
\end{align}
Note that $\pivotQ\inv = \pivotQ$.

\subsubsection{Integrals and cointegrals}\label{sec:SF-int-coint}

\newcommand{\normalizationModTrSF}{c_{\modTr}}
In the basis \eqref{eq:basis_sympFerm}, any left 
integral of $\sympFerm$ is a scalar multiple of $\sum_{m = 0}^3 \mathfrak{b}(m,
1, 1)$, where by $1 \in \field[Z]_2^N$ we mean the vector with $1$ everywhere.
    Since $\sympFerm$ is unimodular\footnote{
    This follows from the explicit form of $\projCover{\tensUnit}$ given in \Cref{sec:simple-projective} below. Or one can use that $\sympFerm$ is factorisable which implies unimodularity, see \cite[Sec.\,6]{BT} for the quasi-Hopf algebra result, and \cite[Lem.\,5.2.8]{KerlerLyubashenko} for the general categorical statement.},
    all left integrals of $\sympFerm$ are also right integrals. 
In \cite[Sec.\,5]{BGR1} it was shown that the left (and, since $H$ is ribbon, also
right) symmetrised cointegral yielding the non-degenerate modified trace on
$\ProjIdeal[\hmodM]$ satisfies
\begin{align}
	\symRightCoint(\mathfrak{b}(m, 1, 1))
	= \symLeftCoint(\mathfrak{b}(m, 1, 1))
	= \normalizationModTrSF \, \deltaOdd{m} (\beta^2 + i^m)\ ,
	\label{eq:symp_ferm_symmetrized_cointegrals}
\end{align}
and is zero on the other basis elements. The normalisation coefficient
where $\normalizationModTrSF \in \field[C]$ will
be determined shortly.
By $\deltaOdd{m}$ we mean the Kronecker delta modulo 2, i.e.\ $\deltaOdd{m} = 1$
if $m$ is odd and 0 otherwise.
The corresponding left (and, since $\pivotQ$ has order 2, also right) cointegral for
$H$ is given by
\begin{align}
	&\coint^r (\mathfrak{b}(m, 1, 1))
	= \coint^l (\mathfrak{b}(m, 1, 1))
	\notag\\
	&= 
	\normalizationModTrSF
	\left(
	\deltaEven{m} (\beta^2 + \deltaEven{N} i^m) + 
	\deltaOdd{m} 
	\deltaOdd{N} 
	i^m
	\right)
	\label{eq:symp_ferm_cointegrals}
\end{align}
and zero elsewhere.

We now determine the coefficient $\normalizationModTrSF$. The integral $\intLyu$ of $\coend$ in the normalisation $\hopfPair \circ (\intLyu \tensor \intLyu) = \id_{\tensUnit}$ was computed in \cite{FGR2} to be
\begin{align}
	\intLyu\big( \mathfrak{b}(m, 1, 1) \big)
	= \nu (-1)^N \beta^2 2^{-(N-1)}
	\delta_{m, 0} 
	\ ,
	\label{eq:intLyu_SF_recap}
\end{align}
where $\nu \in \{ \pm 1\}$ is the remaining sign freedom. We will for simplicity from now on choose $\nu = 1$.
As in \Cref{sec:Lyu-coend-int-for-qHopf} we normalise the cointegral $\coint^r$ of $\sympFerm$ by $\intLyu(h) = \coint^r(S(\betaQ) h )$.
One computes $\coint^r(S(\betaQ) \mathfrak{b}(m, 1, 1)) = \normalizationModTrSF 2 \beta^2 \delta_{m, 0}$, whence
\begin{align}
	\normalizationModTrSF 
	=
	(-1)^N 2^{-N}
	\ .
	\label{eq:normalization_constant_modTr}
\end{align}
Evaluating \eqref{eq:ribboncat-twistnondeg} as in \eqref{eq:qHopf-twistnondeg} gives
\begin{align}
	\stabCoeff
	= \counitL \circ \modT^{\pm 1} \circ \intLyu
	= \coint^r\big( S(\betaQ) \elRibbon^{\mp 1} \alphaQ \big)
	~ \oversetEq[\eqref{eq:symp_ferm_ribbon}] ~
	\beta^{\mp 2}
	\ ,
	\notag
\end{align}
which implies $\DD^2 = \stabCoeffM \stabCoeffP = 1$, and we choose $\DD = 1$.
Then we have $\anomaly = \beta^{-2}$ for the anomaly of $\hmodM[\sympFerm]$.

\subsection{The modular tensor category \texorpdfstring{$\hmodM[\sympFerm]$}{qM}}\label{sec:mod-cat-QM}

Here we briefly discuss the simple and projective modules of $\hmodM[\sympFerm]$ and give their internal characters.
Then we show that $\hmodM[\sympFerm]$ does not factorise into a product of
modular tensor categories. While we will not use this fact later on, we still think it is a noteworthy observation. 

\subsubsection{Simple and projective modules}\label{sec:simple-projective}
The simple and projective modules of $\sympFerm$ have been computed in \cite[Sec.\,3.7]{FGR2}.
One finds that $\sympFerm$ has four simple modules which we denote by $X_0^\pm \in (\hmodM[\sympFerm])_0$ and $X_1^\pm \in (\hmodM[\sympFerm])_1$.

The simple modules $X_0^\pm$ are one-dimensional.
They are distinguished by the action of $\genK$, which acts as $\pm1$ on $X_0^\pm$. The tensor unit of $\hmodM[\sympFerm]$ is given by $X_0^+$.
As submodules of $\sympFerm$, the projective covers $P_0^\pm$ of $X_0^\pm$ are
generated by the primitive (non-central) idempotents 
\begin{align}\label{eq:e0+-}
	\eQ_0^\pm = \tfrac{1}{2}(\oneQ \pm \genK)
	\eQ_0
	\ .
\end{align}
They are both $2^{2N}$-dimensional, and the spaces of intertwiners from
$P_0^\varepsilon$ to $P_0^\delta$ have dimension $2^{2N-1}$ for $\varepsilon, \delta
\in \{+, -\}$. 

The simple modules $X_1^\pm$ are $2^N$-dimensional and projective, and $\genK$ acts on their highest weight state as $\pm i$. As $X_1^\pm$ is projective, it can be realised as a direct summand of $\sympFerm$. The corresponding central (non-primitive) idempotents give $2^N$ copies of $X_1^\pm$. Namely, $2^N X_1^\pm \subset \sympFerm$ is the image of
\begin{align}\label{eq:e1+-}
	\eQ_1^\pm
	= \frac{1}{2} \eQ_1
	\big(
	\oneQ \mp i \genK \prod_{j=1}^N (\oneQ - 2 \genF^+_j \genF^-_j)
	\big)
	\ .
\end{align}

\begin{remark}\label{rem:X-transparent}
Since $\hmodM[\sympFerm]$ is factorisable, it has no transparent objects other than direct sums of the tensor unit (see \cite{Shimizu:2016} for equivalent characterisations of factorisability). The grade-0 component $(\hmodM[\sympFerm])_0$ is also a braided finite tensor category, but it is not factorisable. The transparent objects in $(\hmodM[\sympFerm])_0$ are precisely all finite direct sums of $X_0^+$ and $X_0^-$, see \cite[Prop.\,5.3]{Davydov:2012xg}. That $X_0^-$ is transparent in $(\hmodM[\sympFerm])_0$ can also be easily read off from the monodromy matrix $\monodromy$ in \eqref{eq:M}.
Namely, when acting on $V \otimes X_0^-$ for any 
$V \in (\hmodM[\sympFerm])_0$, in the sum over $I$ in $\monodromy$, only the summand with $I=0 \in \field[Z]_2 \times \field[Z]_2 \times \field[Z]_2^N \times \field[Z]_2^N$ contributes, and this summand acts as the identity on $V \otimes X_0^-$.
\end{remark}

The tensor product closes on the set of simples and projectives, explicitly we have (see \cite[Sec.\,2.2]{FGR2})
\begin{align*}
	X_1^\varepsilon \otimes X_1^\delta \cong P_0^{\varepsilon\delta}
	~~,\quad
	X_0^\varepsilon \otimes Y^\delta \cong Y^{\varepsilon \delta}
	~~\text{where}~~ 
	Y \in \{ X_0, X_1, P_0 \} \ ,
\end{align*}

In the Grothendieck ring $\grothring[{\hmodM[\sympFerm]}]$ of $\hmodM[\sympFerm]$, we have $\grothringclass{P_0^\varepsilon} = 2^{2N - 1} ( \grothringclass{X_0^+} + \grothringclass{X_0^-})$, so that the above tensor products result in the following product for the generators of $\grothring[{\hmodM[\sympFerm]}]$:
\begin{alignat}{2}
	\grothringclass{X_0^\delta} \cdot \grothringclass{X_0^\varepsilon}
	&= \grothringclass{X_0^{\delta \varepsilon}}~~,
	\quad \quad
	\grothringclass{X_0^\delta} \cdot \grothringclass{X_1^\varepsilon}
	&= \grothringclass{X_1^{\delta \varepsilon}}~,
	\notag
	\\
	\grothringclass{X_1^\delta} \cdot \grothringclass{X_1^\varepsilon}
	&= 2^{2N - 1} \big( \grothringclass{X_0^+} + \grothringclass{X_0^-} \big)
	\ .
	\label{eq:symp_ferm_structure_constants}
\end{alignat}

\subsubsection{Internal characters and their S-transformation}
\label{sec:SF_modular_action_elements}
Recall from \Cref{rem:modular_actions} that $\End(\id_{\cat})$ is equipped with an
action of $\slTwoZ$, and that there are special natural transformations $\phi_V$ and
$\modS_{\cat}(\phi_V)$ 
for any $V \in \cat$.
It is known from \cite{Fuchs:2010mw,Shimizu:2015} (see \cite[Sec.\,3]{GR-nonssi-Verlinde} for a summary) that the map
\begin{equation} \label{eq:sigma-ring-hom}
\grothring[\cat] \to \End(\id_{\cat})~~, \quad[V] \mapsto 
\modS_{\cat}(\phi_V)
\end{equation}
is an injective ring homomorphism.
This fact has been used in  \cite[Thm.\,3.9]{GR-nonssi-Verlinde} to give a non-semisimple variant of the Verlinde formula.

For any algebra $A$, $\End(\id_{\hmodM[A]})$ is canonically isomorphic to $Z(A)$, the
centre of~$A$. In the case of $\sympFerm$, the central elements $\elQbold{\chi}_V$ and $\elQbold{\phi}_V$
corresponding to $\modS_{\cat}(\phi_V)$ and $\phi_V$, respectively, have been worked
out in~\cite[Sec.\,4.4]{FGR2}.
Since $\elQbold{\chi}_V$, $\elQbold{\phi}_V$ depend on $V$ only through the class in $\grothring[{\hmodM[\sympFerm]}]$, knowing them on simple modules is
enough:
\begin{align}\label{eq:phi-SF-explicit_X}
	\elQbold{\phi}_{X_0^\pm} 
	=
	2^{N+1} \beta^2 \eQ_0^\pm \prod_{j=1}^N \genF^+_j \genF^-_j
	~~,\qquad
	\elQbold{\phi}_{X_1^\pm} = \pm 
	2^{N+1} \eQ_1^\pm
\end{align}
and
\begin{align}
	\elQbold{\chi}_{X_0^\pm} = \eQ_1 \pm \eQ_0
	~~,\qquad
	\elQbold{\chi}_{X_1^\pm} 
	= \pm \beta^2 4^N \eQ_0 \genK \prod_{j=1}^N \genF^+_j \genF^-_j 
	+ 2^N (\eQ_1^+ - \eQ_1^-)
	\ .
	\label{eq:internal_characters_central_SF}
\end{align}

As an immediate consequence of \eqref{eq:sigma-ring-hom}, the map $\grothring[{\hmodM[\sympFerm]}] \to Z(\sympFerm)$, $[V] \mapsto \elQbold{\chi}_V$ is an injective ring homomorphism, and thus has to agree with the product given in \eqref{eq:symp_ferm_structure_constants}. Indeed, we have, for example
\begin{align*}
\elQbold{\chi}_{X_0^-} \elQbold{\chi}_{X_1^+}=\elQbold{\chi}_{X_1^-}
~~,\quad
(\elQbold{\chi}_{X_1^+})^2 = 2^{2N} (\eQ_1^+ + \eQ_1^-) = 2^{2N-1}(\elQbold{\chi}_{X_0^+}+\elQbold{\chi}_{X_0^-})\ .
\end{align*}

\subsubsection{Primality of \texorpdfstring{$\hmodM[\sympFerm]$}{MQ}}
\updatelabelname{Primality of MQ}

Recall from e.g.~\cite{Laugwitz-Walton} that a \emph{topologising subcategory} of an
abelian category $\cat[A]$ is a full subcategory closed under finite direct sums and
subquotients in $\cat[A]$.
Recall further that the \emph{M\"uger centralizer} $M_{\cat[A]}(\cat[X])$ of a class of 
objects $\cat[X]$ in the braided monoidal category $\cat[A]$ is the full subcategory 
consisting of those objects in $\cat[A]$ which have trivial monodromy (i.e.\ double 
braiding) with every object in $\cat[X]$.
Note that for a braided tensor category $\cat$, we have $M_{\cat}(\tensUnit) \cong \cat$, and
if $\cat$ is factorisable also $M_{\cat}(\cat) \cong \Vect$ (see \cite{Shimizu:2016}).

In \cite{Laugwitz-Walton}, the following remarkable theorem was shown, extending
M\"uger's result to the \emph{non}-semisimple case.
Before restating it, let us remark that a \emph{tensor subcategory} 
of a tensor category $\cat$ is a
subcategory closed under tensor products and containing the
tensor unit.

\begin{theorem}[{\cite[Thm.~4.17]{Laugwitz-Walton}}]
	Let $\cat$ be a modular tensor category, and let $\cat[D]$ be a topologising tensor subcategory of $\cat$.
	If $\cat[D]$ is factorisable
with respect to the braiding inherited from $\cat$,
	then $\cat \cong \cat[D] \boxtimes M_{\cat}(\cat[D])$
	as ribbon categories.
\end{theorem}

This leads to the notion of \emph{prime modular tensor categories} as those modular
tensor categories $\cat$ which do not admit a proper non-trivial (i.e.\ not
equivalent to $\cat$ or $\Vect$) factorisable topologising
tensor subcategory.

We can now show that symplectic fermions provide a  prime example of such a category.

\begin{proposition}\label{prop:QM-prime}
	The modular tensor category $\hmodM[\sympFerm]$ is prime.
\end{proposition}

\begin{proof}
	Let $\cat[D]$ be a factorisable topologising tensor subcategory of $\hmodM[\sympFerm]$ which is not $\Vect$. We will show that then already $\cat[D]=\hmodM[\sympFerm]$.
	
	We start by showing that $\cat[D]$ contains the four simple objects of $\hmodM[\sympFerm]$. It contains $\tensUnit = X_0^+$ by definition. Since we assume $\cat[D] \neq \Vect$,
and there are no self-extensions of the tensor unit \cite[Sec.\,4.4]{EGNO},
 it must contain at least one more simple object. 
	
Since  $\cat[D]$ is factorisable, it cannot be wholly contained in $(\hmodM[\sympFerm])_0$.
Otherwise, it would necessarily contain $X_0^-$,
the only simple object of $(\hmodM[\sympFerm])_0$ that is not the tensor unit. But 
we have seen in \Cref{rem:X-transparent} that $X_0^-$ is transparent in $(\hmodM[\sympFerm])_0$.
 
Therefore, $\cat[D]$ must contain at least one of $X_1^\pm$. Suppose it contains $X_1^+$. Since $X_1^+ \otimes X_1^+ \cong P_0^+$ contains $X_0^+$ and $X_0^-$ as subquotients,
we have by topologising property that $X_0^-\in \cat[D]$.
Since $X_1^+ \otimes X_0^- = X_1^-$, $\cat[D]$ then also contains $X_1^-$. Finally, $X_1^+ \otimes X_1^- \cong P_0^-$, and so $\cat[D]$ contains all projectives and hence is equal to $\hmodM[\sympFerm]$. The argument starting from $X_1^-$ instead of $X_1^+$ is the same.
\end{proof}

One consequence of \Cref{prop:QM-prime} is that $\hmodM[\sympFerm]$ is not a product of more elementary modular tensor categories,
and in particular one cannot directly reduce the study of general $N$ to the $N=1$ case, see however an interesting observation in~\cite[Rem.\,5.6]{McRae}.

\subsection{Pullback ideals}
\label{sec:symp_ferm_pullback}

In this section we consider an example of a pull back ideal and compute the resulting
modified trace.
We start with some general comments on quasi-Hopf subalgebras of $\sympFerm(N,\beta)$ and
then specialise to a specific subalgebra $A$ in $H := \sympFerm(2,\beta)$.

\subsubsection{Quasi-Hopf subalgebras in $\sympFerm$}\label{sec:other-A}

Write $W^+ \subset \sympFerm = \sympFerm(N,\beta)$ for the $N$-dimen\-sional vector space linearly spanned by the $\genF^+_i$, and analogously for $W^-$. 
The total space $W^+ \oplus W^-$ carries a symplectic form given by $(\genF^+_i,\genF^-_j) = \delta_{i,j}$.

Before turning to the specific example we will study in more detail, let us look more generally at quasi-Hopf subalgebras $A$ generated by
	some subspace of
$W^+ \oplus W^-$ and by powers of $\genK$. 
First note that for $w = w^+ + w^- \in W^+ \oplus W^-$, the coproduct \eqref{eq:sympFermCoprod} gives 
\begin{equation}\label{eq:cop-w}
    \Delta(w) = w \otimes \oneQ + \eQ_0 \otimes w + i \eQ_1 \otimes (w^+-w^-)\ .
\end{equation}
By assumption, $\eQ_0, \eQ_1 \in A$, and so if $w^++w^- \in A$, closure under the coproduct requires that also $w^+-w^- \in A$.
We may therefore pick two subspaces $V^\pm \subset W^\pm$ and define $A = A(V^+,V^-) \subset \sympFerm$ to be the unital subalgebra generated by elements in $V^+ \oplus V^-$ and by $\genK$. A quick look at the explicit expressions in \Cref{sec:SF-qHopf-structure} confirms that the coproduct and antipode restrict to $A$, and that $\alphaQ,\betaQ \in A$ and
$\coassQ^{\pm 1} \in A^{\otimes 3}$. The pivot $\pivotQ$ from \eqref{eq:SF_pivot} is also contained in $A$.
However, the $R$-matrix does not restrict to $A$ (unless $A=\sympFerm$).

If the restricted symplectic form on $V^+ \oplus V^-$ is non-degenerate,
$A$ is isomorphic as an algebra to $\sympFerm(N',\beta')$ with $N' = \dim V^\pm$ and $\beta'$ arbitrary.
The isomorphism maps $\genK\mapsto\genK$ and takes a symplectic basis of $V^+ \oplus V^-$ to the generators $\genF^\pm_i$ of $\sympFerm(N',\beta')$. From the expressions in \Cref{sec:SF-qHopf-structure} one sees that this defines an isomorphism of quasi-Hopf algebras iff $N' \equiv N \, (\mathrm{mod} \, 2)$ and $\beta^2 = \beta'{}^2$.

The projective cover $P_0(A)$ of the trivial $A$-module $\field[C]$ (with $\genK$ acting as $1$) is given by $A \eQ_0^+$, cf.\ \eqref{eq:e0+-}. The socle of $P_0(A)$ is one-dimensional and $\genK$ acts as $(-1)^{\dim V^+ + \dim V^-}$. Thus $A$ is unimodular iff $\dim V^+ + \dim V^-$ is even.

\begin{remark}
We note that by~\eqref{eq:cop-w}, any element $w\in W_+ \oplus W_-$, together with $\genK$, 
generates a coideal subalgebra in $\sympFerm=\sympFerm(N,\beta)$. More generally, $\genK$ and any subspace $V\subset W_+ \oplus W_-$ 
generate a coideal subalgebra $I_V\subset \sympFerm$. 
Therefore, following \Cref{rem:tens-ideal-comodule-alg} we have a family of two-sided tensor ideals in $\hmodM[\sympFerm]$ parametrised by the subspaces~$V$ and a choice of module $M\in \hmodM[I_V]$, possibly with a modified trace on them.\footnote{
We thank the anonymous referee for pointing out this more general construction.} 
It is a very interesting  problem to classify such modified traces for a given subspace $V$ of the symplectic space $W_+ \oplus W_-$
but this is out of the scope of the current paper. 
\end{remark}

Altogether we see that each such choice of a unimodular quasi-Hopf subalgebra~$A$ gives a tensor ideal with modified trace via pull-back,
recall~\Cref{prop:pullback_modTr_along_restriction_qHopf}.
We have thus a continuum of intermediate tensor ideals 
in $\hmodM[\sympFerm]$,
but here we will study just one of them.

\medskip

Namely, the example we study in detail now is
\begin{align*}
	V^\pm = \field[C] \genF_1^\pm
	~~,\quad
	A = A(V^+,V^-) \, \subset \, \sympFerm(2,\beta) =: H
\end{align*}
where $\beta$ is any choice of 4th root of unity. 

\subsubsection{Lifts of projective modules of $A$}\label{sec:left-of-proj-A}

As an algebra (but not as a quasi-Hopf algebra), $A$ agrees with one pair of symplectic fermions $\sympFerm(1,\beta)$;
in particular, we know its projective modules.

Let $\projCoverTensorUnitSF[A]$ be the projective cover of the tensor unit of $\hmodM[A]$, and similarly let 
$P_0^-(A)$, $X_1^+(A)$, $X_1^-(A)$ 
the remaining indecomposable projectives.
If we refer to the projectives in $\hmodM[H]$, we will write $P_0^\pm(H)$ and $X_1^\pm(H)$ instead.
The $A$-module $\projCoverTensorUnitSF[A]$ is four-dimensional and 
is generated by a vector $v_0$ which has
$\genK$-eigen\-value~$1$.
We set $v_1^\pm = \genF_1^\pm . v_0$ and $v_2 = \genF_1^+ \genF_1^- . v_0$.
The $\genK$-eigenvalue of $v_2$ is $1$, while those of $v_1^\pm$ are $-1$.

Let $\parMat = \begin{psmallmatrix} a^- & a^+ \\ b^- & b^+ \end{psmallmatrix} \in
\operatorname{Mat}_2(\field[C])$.
On $\projCoverTensorUnitSF[A]$, we define the actions of $\genF_2^\pm \in H$ as
\begin{align}
	\label{eq:lifted_action}
	\genF_2^\varepsilon 
	= a^\varepsilon \, \genF_1^- + b^\varepsilon \, \genF_1^+
	\in \End_{\field}( \projCoverTensorUnitSF[A] )
	,\quad\quad
	\varepsilon = \pm
	\ .
\end{align}
It is not hard to verify that with this we indeed obtain an $H$-module structure on
$\projCoverTensorUnitSF[A]$, and we shall denote it by $P_\parMat \in \hmodM[H]$.
By construction, $P_\parMat$ is a lift of $\projCoverTensorUnitSF[A]$
and we will now show that these are in fact all lifts:

\begin{proposition}
	\label{prop:lift_of_proj_cover_tU_sympFerm}
	Every lift of $\projCoverTensorUnitSF[A]$ to $\hmodM[H]$ is of the form $P_\parMat$ for
	some $\parMat \in \operatorname{Mat}_2(\field[C])$.
	Moreover, $P_\parMat \cong P_{\parMat'}$ in $\hmodM[H]$ iff $\parMat = \parMat'$.
\end{proposition}

\begin{proof}
	To lift the $A$-module structure, we classify all possible $\genF_2^\pm$-actions
	on $\projCoverTensorUnitSF[A]$.
	This is a linear algebra problem: 
	the algebra relations (in particular the anticommutators) of $H$ have to
	be satisfied by the matrices representing the actions of generators.
	In the basis $\{v_0, v_1^-, v_1^+, v_2\}$ of $\projCoverTensorUnitSF[A]$, the
	action of the $A$-generators $\genK$, $\genF_1^-$, and $\genF_1^+$ is given by
	\begin{align*}
		\begingroup
		\genK =
		\begin{psmallmatrix}
			1 & 0 & 0 & 0 \\
			0 & -1 & 0 & 0 \\
			0 & 0 & -1 & 0 \\
			0 & 0 & 0 & 1
		\end{psmallmatrix}
		,
		\quad
		\genF_1^- =
		\begin{psmallmatrix}
			0 & 0 & 0 & 0 \\
			1 & 0 & 0 & 0 \\
			0 & 0 & 0 & 0 \\
			0 & 0 & 1 & 0
		\end{psmallmatrix},
		\quad \text{and }
		\genF_1^+ = 
		\begin{psmallmatrix}
			0 & 0 & 0 & 0 \\
			0 & 0 & 0 & 0 \\
			1 & 0 & 0 & 0 \\
			0 & -1 & 0 & 0
		\end{psmallmatrix}
		,
		\endgroup
	\end{align*}
	respectively.
	The only matrices that anticommute with these are of the form
	\begin{align*}
		E_{a,b} = 
		\begin{psmallmatrix}
			0 & 0 & 0 & 0 \\
			a & 0 & 0 & 0 \\
			b & 0 & 0 & 0 \\
			0 & -b & a & 0
		\end{psmallmatrix}
		= 
		a\, \genF_1^- + b\, \genF_1^+
		\ ,
		\qquad
		a,b \in \field[C]
		\ .
	\end{align*}
	The matrices
	$E_{a,b}$ anticommute with $E_{c,d}$ for any four
	complex numbers $a, b, c, d$.
	In particular, given $(a^-, a^+, b^-, b^+) \in \field[C]^4$, we get a
	$H$-module on which the action of $\genF_2^\pm$ is represented by
	$a^\pm$, $b^\pm$ 
	as in \eqref{eq:lifted_action}.
	By construction these are all possible lifts.
	
	Now we show that these are pairwise non-isomorphic. 
	Let $\parMat = \begin{psmallmatrix} a^- & a^+ \\ b^- & b^+ \end{psmallmatrix}$ and $\parMat' = \begin{psmallmatrix} c^- & c^+ \\ d^- & d^+ \end{psmallmatrix}$, and 
	suppose $P_{\parMat} \cong P_{\parMat'}$.
	This isomorphism must come from an invertible endomorphism of
	$\projCoverTensorUnitSF[A]$ intertwining the $A$-action, so is necessarily given
	by the action with $\xi_{x,y} = x \oneQ
	+ y \genF_1^- \genF_1^+$, where $x,y \in \field[C]$ with $x \neq 0$.
	Requiring $\xi_{x,y}$ to intertwine the $\genF_2^{\varepsilon}$-actions on
	$P_{\parMat}$ and $P_{\parMat'}$ 
	forces $a^\varepsilon = c^\varepsilon$ and
	$b^\varepsilon = d^\varepsilon$, for $\varepsilon = +, -$.
	Hence, $\parMat=\parMat'$.
\end{proof}

The following simple \namecref{prop:symp_ferm_f2_acting_on_PM} to the proof of the above proposition will be useful later.
\begin{corollary}
	\label{prop:symp_ferm_f2_acting_on_PM}
	The element $\genF^+_2 \genF^-_2 \in H$ acts on $P_\parMat$ as $\det\parMat \cdot
	\genF^+_1 \genF^-_1$, and consequently $\prod_{j=1}^2 (\oneQ - x \genF^+_j
	\genF^-_j)$ acts as $\oneQ - x (1 + \det\parMat) \genF^+_1 \genF^-_1$ for any $x \in
	\field[C]$.
\end{corollary}

\begin{remark}
By contrast, the simple projective modules of $A$ do not admit lifts, because their dimensions are too small:
the grade 1 part of $\relativeProjectives{A}{H}$ has to coincide with the grade 1 part of $\hmodM[H]$ by \Cref{rem:ideals_in_graded_categories}. 
Instead, a suitable direct sum of simple $A$-projectives lifts to each of the two simple projective of $H$, and so no continuum of parameters is involved.
\end{remark}

\subsubsection{Symmetrised cointegrals of $A$}

Next, we describe the left and the right symmetrised cointegral of $A$, which
by~\cite[Thm.~1.1]{BGR1} correspond precisely to the left and the right modified trace
on $\ProjIdeal[{\hmodM[A]}]$.

\begin{proposition}
	\label{prop:symCoint_for_subqHA}
	A symmetrised right resp.\ left cointegral for $A$ is given by
	\begin{align}
		\symRightCointA
		\big( \genF_1^+ \genF_1^- \genK^m \big)
		=
		\delta_{m,1} - i \beta^2 \delta_{m,3}
		\qtextq{resp.}
		\symLeftCointA
		\big( \genF_1^+ \genF_1^- \genK^m \big)
		=
		\delta_{m,1} + i \beta^2 \delta_{m,3},
		\ ,
		\notag
	\end{align}
	and by 0 on all other basis vectors.
	The left and the right modified trace $\modTr^{A,l}$ and $\modTr^{A,r}$ on
	$\ProjIdeal[{\hmodM[A]}]$ corresponding to the cointegrals
	\begin{itemize}
		\item are proportional on the grade 0 part and the grade 1 part of
		$\ProjIdeal[{\hmodM[A]}]$ separately, but
		\item are not proportional on all of $\ProjIdeal[{\hmodM[A]}]$.
	\end{itemize}
\end{proposition}

The final statement implies that the pivotal quasi-Hopf algebra $A$ cannot be made ribbon, or equivalently, 
that the pivotal structure on $\hmodM[A]$ cannot be extended to a ribbon structure.

\begin{proof}
	We start by computing the right symmetrised cointegral.
	Since $A$ is pivotal and unimodular, this can be done by solving
	\eqref{eq:symm_coint_equation}.
	Using the explicit formulas for $\qR, \pR$ from \eqref{eq:qRpR_SF},
	\eqref{eq:symm_coint_equation} becomes
	\begin{align}
		\symRightCointA(h) \oneQ 
		&= 
		\symRightCointA(h\sweedler{1}) 
		\pivotQ \eQ_0 h\sweedler{2}
		+
		\symRightCointA(\eQ_0 h\sweedler{1}) 
		\pivotQ \eQ_1 h\sweedler{2} \betaQ
		+
		\symRightCointA(\eQ_1 h\sweedler{1}) 
		\pivotQ \eQ_1 \betaQ h\sweedler{2}
		\ .
		\notag
	\end{align}
	Abbreviate $h_m = \genF^+_1 \genF^-_1 \genK^m$, and note that
	\begin{align}
		\Delta(h_m) 
		&\approx
		h_m \tensor \eQ_0 \genK^m 
		+ h_m \genK^{2 \deltaOdd{m}} \tensor \eQ_1 \genK^m 
		\ ,
		\notag
	\end{align}
	where $\approx$ means that we disregard terms that have fewer than all $\genF$s in
	the first tensor factor.
	We then make the ansatz $\symRightCointA(\genF^+_1 \genF^-_1 \genK^m) = c_m \in
	\field[C]$ and zero otherwise.
	Clearly this satisfies \eqref{eq:symm_coint_equation} for any $h$ not of the form
	$h_m$.
	Using the explicit expression \eqref{eq:SF_pivot} for $\pivotQ$ and the fact that
	$[\betaQ, \genK] = 0$, the cointegral equation \eqref{eq:symm_coint_equation} for $h
	= h_m$ becomes
	\begin{align}
		c_m \cdot \oneQ 
		&= 
		\symRightCointA((h_m)\sweedler{1}) \cdot
		\eQ_0 (h_m)\sweedler{2} \genK 
		- i
		\beta^2 
		\symRightCointA((h_m)\sweedler{1}) \cdot
		\eQ_1 (h_m)\sweedler{2} \genK
		\notag
		\\
		&= 
		c_m \eQ_0 \genK^{m+1}
		- i \beta^2 c_{m + 2 \deltaOdd{m}} \eQ_1 \genK^{m+1}
		\ .
		\notag
	\end{align}
	Immediately this implies $c_0 = c_2 = 0$, so we are only left with the two equations
	\begin{align}
		c_1 \cdot \oneQ 
		&= c_1 \eQ_0 + i \beta^2 c_{3} \eQ_1
		\qandq
		c_3 \cdot \oneQ 
		= 
		c_3 \eQ_0 - i \beta^2 c_{1} \eQ_1 
		\ ,
		\notag
	\end{align}
	which imply $c_3 = - i \beta^2 c_1$, so that $\symRightCointA$ can be chosen as in the
	statement of this \namecref{prop:symCoint_for_subqHA}.
	
	A left symmetrised cointegral is given by $\symRightCointA \circ S\inv$, see
	\cite[Sec.~3]{BGR1}.
	That is, up to a scalar $\symLeftCointA$ is completely determined as
	\begin{align}
		(\symRightCointA \circ S\inv) (\genF^+_1 \genF^-_1 \genK^m)
		= \delta_{m, 3} - i \beta^2 \delta_{m, 1}
		= - i \beta^2 (\delta_{m, 1} + i \beta^2 \delta_{m, 3})
		\ ,
		\notag
	\end{align}
	and we choose $- i \beta^2 \symLeftCointA = \symRightCointA \circ S\inv$, which yields
	$\symLeftCointA$ as in the statement.
	
	We will now show that the associated modified traces are proportional on the degree
	0 and the degree 1 component separately, but not simultaneously.
	Let us write $\widehat\coint{}^{A,-} := \symRightCointA$ and $\widehat\coint{}^{A,+} :=
	\symLeftCointA$, so that $\widehat\coint{}^{A,\pm}(\genK^m \genF^+_1 \genF^-_1) =
	\delta_{m, 1} \pm i \beta^2 \delta_{m, 3}$.
	
	The endomorphism space of $P_0^\pm(A)$ has as basis the identity and the map given
	by acting
	with $\genF^+_1 \genF^-_1$, 
    and hence we need to
	evaluate $\widehat\coint{}^{A,\pm}$ on $\eQ_0^\pm$ and $\eQ_0^\pm \genF^+_1 \genF^-_1$.
	Clearly, the value on $\eQ_0^\pm$ is zero.
	One easily computes
	\begin{align}\label{eq:sym-A-coint-computed}
		\widehat\coint{}^{A,\pm}(\eQ_0^\varepsilon \genF^+_1 \genF^-_1)
		= \tfrac{1}{4} \sum_{m=0}^3 \varepsilon^m \widehat\coint{}^{\pm,A}(\genK^m \genF^+_1 \genF^-_1)
		= \varepsilon \frac{1 \pm i \beta^2}{4}
		\ .
	\end{align}
	In the degree 1 component, indecomposable projectives only have the identity
	endomorphism.
	Using additivity of the modified trace, it is enough to compute its value on the
	identity endomorphism
$(-) \cdot \eQ_1^\pm$ of $2 X_1^\pm(A) \subset A$ (the image of $\eQ_1^\pm$ is given above \eqref{eq:e1+-}).
	We have
	\begin{align*}
		\widehat\coint{}^{\pm,A} (\eQ_1^\varepsilon)
		=  \varepsilon i \widehat\coint{}^{\pm,A}(\eQ_1 \genK \genF^+_1 \genF^-_1)
        =  \varepsilon i \frac{1 \mp i \beta^2}2
		\ .
	\end{align*}
	Since $\beta^4 = 1$, $(1 - i \beta^2) i \beta^2 = 1 + i \beta^2$.
	Thus for all $P \in \ProjIdeal[{\hmodM[A]}]$
	\begin{align}
		\modTr^{A,r}_{\eQ_0 P} = - i \beta^2 \modTr^{A,l}_{\eQ_0 P}
		\qandq
		\modTr^{A,r}_{\eQ_1 P} = i \beta^2 \modTr^{A,l}_{\eQ_1 P}
		\ ,
		\notag
	\end{align}
	and so the two traces are not proportional.
\end{proof}

\subsubsection{Pullback modified traces from $A$}\label{sec:pullback-trace-from-A}

The modified traces on $\relativeProjectives{A}{H}$ obtained via 
pullback (\Cref{prop:pullback_ideal} and \Cref{prop:pullback_modTr_along_restriction_qHopf}) 
from the
modified traces in \Cref{prop:symCoint_for_subqHA} are not non-degenerate.
Namely, we have:

\begin{proposition}
	\label{prop:pullback_trace_vanishes_on_proj_SF}
	The pullback modified traces vanish on projective $H$-modules:
	\begin{align*}
		\pullbackTrace[{\modTr^{A,r}}] \big|_{\ProjIdeal[{\hmodM[H]}]}
		\equiv 0
		\quad \text{and} \quad
		\pullbackTrace[{\modTr^{A,l}}] \big|_{\ProjIdeal[{\hmodM[H]}]}
		\equiv 0
		\ .
	\end{align*}
\end{proposition}

\begin{proof}
	We use the suggestive notation $\field[C]^{1 \mid 0}$ and $\field[C]^{0 \mid 1}$
	for the one-dimensional simple modules of both $A$ and $H$,
	on which $\genK$ acts as $1$ and $-1$, respectively, and on
	which the actions of the $\genF^\pm_j$ vanish.
	Similarly we write $\field[C]^{n \mid m} = n \field[C]^{1 \mid 0} \oplus m
	\field[C]^{0 \mid 1}$.
	We have $\trCat[A, r]_{V \tensor \field[C]^{n \mid n}}(f \tensor \id) = 0$ for any $V
	\in \hmodM[A]$ and $f \in \End_A(V)$, where $\trCat[A,r]$ is the right partial trace
	of $\hmodM[A]$.
	Indeed, this follows from 
	\begin{align}
		\trCat[A]_{\field[C]^{1 \mid 0}} (\id)
		= \trCat[{\Vect}]_{\field[C]} \left( \genK \right)
		= 1
		\qandq
		\trCat[A]_{\field[C]^{0 \mid 1}} (\id) = - 1
		\ ,
		\notag
	\end{align}
	see \eqref{eq:categorical_trace_qHopf}.
	It is not hard to verify that
	\begin{align}
		\restrictionFunctor (X_1^\pm(H)) \cong X_1^+(A) \tensor \field[C]^{1 \mid 1}
		\qandq
		\restrictionFunctor (P_0^\pm(H)) \cong P_0^+(A) \tensor \field[C]^{2 \mid 2}
		\label{eq:trace-vanish-aux1}
	\end{align}
	as $A$-modules.
	Thus for any projective indecomposable $P \in \ProjIdeal[{\hmodM[H]}]$ and any $f \in
	\End_H(P)$ we have e.g.\
	\begin{align}
		\pullbackTrace[{\modTr^{A,r}}]_P(f) 
		= \modTr^A_{P'}
		\big(
		\trCat[A,r]_{P' \tensor \field[C]^{n \mid n}}
		(f')
		\big)
		\ ,
		\notag
	\end{align}
	where $P'$ is $X_1^+(A)$ or $P_0^+(A)$, $n$ is $1$ or $2$, and
	$f'$ is $f$ conjugated by one of the isomorphisms in \eqref{eq:trace-vanish-aux1}.
	Our goal is therefore now to show that in all situations $f'$ will be a sum of terms of
	the form $g \tensor h$, with $h \colon \field[C]^{n \mid n} \to \field[C]^{n \mid n}$ of trace zero.
	
	For the odd sector, i.e.\ $P = X_1^\pm(H)$, this is obvious since
	$\End_H(X_1^\pm(H))$ is one-dimensional,
	and so 
    $h$ is a multiple of $\id$.
	
	For $P = P_0^\pm(H)$, we note first that an explicit basis of the $8$-dimensional
	space $\End_H(P_0^\pm(H))$ is given by acting with the elements
	\begin{align*}
		\oneQ
		, \quad \genF^+_1 \genF^-_1
		, \quad \genF^+_1 \genF^+_2
		, \quad \genF^+_1 \genF^-_2
		, \quad \genF^-_1 \genF^+_2
		, \quad \genF^-_1 \genF^-_2
		, \quad \genF^+_2 \genF^-_2
		, \quad \genF^+_1 \genF^-_1 \genF^+_2 \genF^-_2
		\ ,
	\end{align*}
	which are central on the even sector.
Let us denote the second isomorphism in \eqref{eq:trace-vanish-aux1} by $\phi$. 
Explicitly, a basis of $P_0^\pm(H)$ is given by 
\begin{align*}
    \big\{ \, (\genF_1^+)^a (\genF_1^-)^b (\genF_2^+)^c (\genF_2^-)^d \eQ_0^\pm \, \big|\, a,b,c,d \in \field[Z]_2 \big\} \ ,
\end{align*} 
see \cite[Sec.\,3.7]{FGR2}. A basis of $P_0^\pm(A)$ is obtained by just using $\genF_1^\pm$. If we denote a basis of $\field[C]^{2|2}$ by $\{ (\genF_2^+)^c (\genF_2^-)^d | c,d \in \field[Z]_2 \}$, then $\phi$ is simply given by 
\begin{align*}
(\genF_1^+)^a (\genF_1^-)^b (\genF_2^+)^c (\genF_2^-)^d \eQ_0^\pm ~\longmapsto~ (\genF_1^+)^a (\genF_1^-)^b \eQ_0^\pm \otimes (\genF_2^+)^c (\genF_2^-)^d \ ,
\end{align*}
and it commutes with the action of $\genF_1^\pm$.
One can now verify
that the transported actions $\phi \circ (\genF^\pm_2.(-)) \circ \phi^{-1}$ 
and $\phi \circ (\genF^+_2\genF^-_2.(-)) \circ \phi^{-1}$
are of the form $g \otimes h$ with nilpotent $h$ (which hence has trace zero).
\end{proof}

\begin{corollary}\label{cor:pullback-tr-proportional}
	The pullback traces $\pullbackTrace[{\modTr^{A,r}}]$ and
	$\pullbackTrace[{\modTr^{A,l}}]$ are proportional.
\end{corollary}

\begin{proof}
	By \Cref{rem:ideals_in_graded_categories}, the degree 1 component of
	$\relativeProjectives{A}{H}$ equals the degree 1 component of
	$\ProjIdeal[{\hmodM[H]}]$.
	Thus both pullback traces vanish on $( \relativeProjectives{A}{H} )_1$ by
	\Cref{prop:pullback_trace_vanishes_on_proj_SF}.
	By \Cref{prop:symCoint_for_subqHA},
	$\modTr^{A,r}$ and $\modTr^{A,l}$ are proportional on the degree~0 component of
	$\ProjIdeal[{\hmodM[A]}]$.  
\end{proof}

To make things more readable we choose a specific normalisation and set
\begin{align}
	\label{eq:pullback_trace_choice}
	\pullbackTrace[\modTr^A] 
	:= 
	\frac{4}{1 - i \beta^2} \pullbackTrace[{\modTr^{A,r}}]
	=
	\frac{4}{1 + i \beta^2} \pullbackTrace[{\modTr^{A,l}}]
	\ ,
\end{align}
which is well-defined since $\beta^2 = \pm 1$.

Since by \Cref{prop:renormalized_link_invariant} we are free to choose which $\tensIdeal$-labelled edge of a closed ribbon graph $T$ to cut when computing $\renRTinvariant(T)$, as another consequence from \Cref{prop:pullback_trace_vanishes_on_proj_SF} we get:

\begin{corollary}\label{cor:proj-inv-zero}
Let $T$ be an $\tensIdeal$-admissible
closed ribbon graph and $\modTr = \pullbackTrace[\modTr^A]$. If at least one edge of $T$ is coloured by a 
projective $H$-module, then $\renRTinvariant(T) = 0$.
\end{corollary}

\section{Explicit computations of link invariants}
\label{sec:explicit_computation_link_invariants}

In this section we provide the details on how to obtain the link invariants presented in Table~\ref{table:invariants}. We consider the  $n$-framed unknot, the $a,b$-framed Hopf link, and the $(2,m)$-torus knot, each for three choices of ideal in $\hmodM[\sympFerm]$.

Throughout this section we will often write $\varepsilon$, $\delta$, \emph{etc}.\ for multiple
disjoint occurrences of~$\pm$.

\subsection{Framed unknot}\label{sec:compute-unknot}

Recall the ribbon twist from \eqref{eq:symp_ferm_ribbon}.
One shows via induction that 
\begin{align*}
	\elRibbon^{\pm m}
	= (\eQ_0 + (- \beta^{\pm 1} i \genK)^m \eQ_1) 
	\prod_{j=1}^N 
	\big(
	\oneQ \mp 2(m \eQ_0 \pm \deltaOdd{m} \eQ_1) \genF^+_j \genF^-_j
	\big)
\end{align*}
holds for natural $m \geq 0$.
By convention, the ribbon twist is given by acting with $\elRibbon^{-1}$, and so the $n$-framed unknot corresponds to the trace of $\elRibbon^{-n}$, for $n  \in \field[Z]$.

\subsubsection*{Categorical trace}
The categorical trace of the projective module $X_1^\pm$ vanishes, and it is clear that
$\elRibbon$ 
acts as the identity on $X_0^\pm$ since the $\genF$'s act
trivially.
Thus, for $n \in \field[Z]$, 
\begin{align*}
	\trCat[{\hmodM[\sympFerm]}]_{X_0^\pm}(\elRibbon^{n})
	&= \tr_{X_0^\pm}(\pivotQ \betaQ S(\alphaQ))
	= \tr_{X_0^\pm}(\genK)
	= \pm 1
	\ ,
\end{align*}
as claimed in Table~\ref{table:invariants}.

\subsubsection*{The modified trace on the projective ideal}

For the modified trace on the projective ideal, we find on the even sector that the invariant of the $n$-framed unknot ($n \in \field[Z]$) coloured with $P_0^\varepsilon$ is given by
\begin{align*}
	\modTr_{P_0^\varepsilon} (\elRibbon^{-n})
	&= \symRightCoint(\elRibbon^{-n} \eQ_0^\varepsilon)
	= 
	(2 n)^N
	\symRightCoint
	\big(
	\eQ_0^\varepsilon
	\prod_{j=1}^N \genF^+_j \genF^-_j
	\big)
	\overset{(*)}=
	n^N \tfrac{1}{2} \varepsilon \beta^{-2} 
	\ .
\end{align*}
In step $(*)$ we used
\begin{align}
	\symRightCoint
	\big(
	\eQ_0^\varepsilon
	\prod_{j=1}^N \genF^+_j \genF^-_j
	\big)
	&\overset{\eqref{eq:e0+-}}= 
	\frac{1}{4}
	\sum_{m = 0}^3 \varepsilon^{m} \symRightCoint(\genK^m
	\prod_{j=1}^N \genF^+_j \genF^-_j
	)
\underset{\eqref{eq:symp_ferm_symmetrized_cointegrals}}{\overset{\eqref{eq:normalization_constant_modTr}}=}
\frac{1}{2} \varepsilon \beta^2 (-1)^N 2^{-N}
	\nonumber\\
	&    \overset{\phantom{\eqref{eq:e0+-}}}= 2^{-N-1} \varepsilon \beta^{-2} 
	\ .
	\label{eq:hat-lam-e0+-top}
\end{align}
On the odd sector, indecomposable projectives are simple and therefore have
one-dimensional endomorphism spaces.
It was shown in \cite[Lem.~3.14]{FGR2} that $\elRibbon^{\pm 1} \eQ_1^\varepsilon = \varepsilon \beta^{\pm 1}
\eQ_1^\varepsilon$ (where $\eQ_1^\varepsilon$ was given in \eqref{eq:e1+-}),
so that for integer powers we have $\elRibbon^{-n} \eQ_1^\varepsilon =
(\varepsilon \beta)^{-n} \eQ_1^\varepsilon$.
The invariant of the $n$-framed unknot coloured by $X_1^\pm$ is then given by
\begin{align*}
	\modTr_{X_1^\varepsilon}(\elRibbon^{-n}) 
	= 2^{-N} \modTr_{2^N X_1^\varepsilon}(\elRibbon^{-n}) 
	= 2^{-N} (\varepsilon \beta)^{-n} \modTr_{2^N X_1^\varepsilon}
	(\id) = 2^{-N} (\varepsilon \beta)^{-n} 
	\symRightCoint(\eQ_1^\varepsilon)
	\ .
\end{align*}
We find
\begin{align}
	\symRightCoint(\eQ_1^\varepsilon)
	&=
	- \varepsilon
	\frac{1}{2} i
	(-2)^N
	\symRightCoint
	\big(
	\eQ_1 \genK \prod_{j=1}^N \genF^+_j \genF^-_j
	\big)
	\nonumber\\
	&=
	- \varepsilon
	\frac{1}{2} i
	(-2)^N
	\frac{1}{2}
	c_\modTr (\beta^2 + i - \beta^2 - i^3)
	=
	\varepsilon \frac{1}{2}
	\ ,
	\label{eq:sym-coint-on-e1+-}
\end{align}
where $c_{\modTr}$ is the constant from \eqref{eq:normalization_constant_modTr}.

\subsubsection*{Pullback trace}

For the pullback trace, we restrict ourselves to the situation worked out in \Cref{sec:pullback-trace-from-A}, i.e.\ to the pullback trace obtained from $A \subset H = \sympFerm(2,\beta)$.
We first use \Cref{prop:symp_ferm_f2_acting_on_PM} to
rewrite the action of $\elRibbon^{-n}$ on $P_\parMat$ as, for $u\in P_\parMat$,
\begin{align*}
	\elRibbon^{-n}.u  =
	\prod_{j=1}^2
	\big(
	\oneQ +2n \genF^+_j \genF^-_j
	\big).u
	= \bigl(\oneQ +2n (1 + \det\parMat) \genF^+_1 \genF^-_1\bigr).u
	\ .
\end{align*}
Thus for $n \in \field[Z]$ the invariant of the $n$-framed unknot coloured by $P_\parMat$ is given by
\begin{align*}
	\pullbackTrace[{\modTr^A}]_{P_\parMat} (\elRibbon^{-n})
	= 2n (1 + \det\parMat) 
	\frac{4}{1 - i \beta^2}
	\modTr^{A,r}_{P_0^+(A)}(\genF^+_1 \genF^-_1)
	= 2 n (1 + \det\parMat) ~,
\end{align*}
where we used \eqref{eq:modified_trace_on_indecomp_proj} and \eqref{eq:sym-A-coint-computed} to see that in terms of the symmetrised  cointegral for $A$,
\begin{align}\label{eq:P0A+-mod-trace}
	\modTr^{A,r}_{P_0^+(A)}(\genF^+_1 \genF^-_1) = 
	\symRightCointA(\eQ_0^+\genF^+_1 \genF^-_1) = (1-i \beta^2)/4
\end{align} 

\subsection{Framed Hopf link}

We compute the  Hopf link invariants 
$\hopflink{X,n}{U,m}$ coloured with $X \in  \tensIdeal$ and $U \in \Irr$, and with $n$-fold twist on $X$ and $m$-fold twist on $U$. 
Since we can take $U$ to be irreducible, the twist amounts to an overall factor times the invariant of $\hopflink{X,n}{U,0}$. For the latter we need to compute the trace $\modTr_X(\elRibbon^{-n} \modS_{\cat}(\phi_U)_X)$ of the $(-n)$'th power of the ribbon element \eqref{eq:symp_ferm_ribbon} composed  with the open Hopf link operator given in  \eqref{eq:openHopfLink_definition}.
We saw in Section~\ref{sec:SF_modular_action_elements} that $\modS_{\cat}(\phi_U)_X$ is given by multiplication with the central element $\elQbold{\chi}_U$ from \eqref{eq:internal_characters_central_SF}.

\subsubsection*{Categorical trace}
For $\tensIdeal = \cat$,
the only thing to compute is the invariant of $\hopflink{X_0^\varepsilon,n}{X_0^\delta,m}$.
But with respect to the grade 0 part of $\hmodM[\sympFerm]$, $X_0^\pm$ is transparent (see \Cref{rem:X-transparent}) 
and has trivial twist, therefore we just obtain the product of the categorical dimensions of the simple objects.

\subsubsection*{Modified trace on projective ideal}

We first note that since $X_1^\delta$ is simple, the $m$-fold twist is given by multiplying with an overall factor, which in
    this case is given by $(\delta \beta^{-1})^m$ as explained below \eqref{eq:hat-lam-e0+-top}.
Alternatively, one can directly use the twist eigenvalue given in  \cite[Sec.\,2.6]{FGR2}.

We start by computing $\hopflink{P,n}{U,m}$ for $U = X_0^\pm$.
If $P$ is in the grade 0 part of $\hmodM[\sympFerm]$, we use that $X_0^\pm$ is transparent relative to $P$ and has trivial twist, resulting in the quantum dimension $\pm1$ of $X_0^\pm$ times the invariant of the $n$-framed $P$-coloured unknot already computed in \Cref{sec:compute-unknot}:
\begin{align*}
\hopflink{P_0^\varepsilon,n}{X_0^\delta,m}: ~~  \tfrac{\varepsilon \delta}{2} \, n^N  \beta^{-2} \ .
\end{align*}
On the other hand, for $P = 2^N X_1^\varepsilon$ we first compute the zero-framed invariant,
\begin{align*}
	\hopflink{2^N X_1^\varepsilon,0}{X_0^\delta,0}: \quad
	\modTr_{2^N X_1^\varepsilon} ( \modS_{\cat}(\phi_{X_0^\delta})_{2^N X_1^\varepsilon} )
	= \symRightCoint(\eQ_1^\varepsilon)
	\overset{\eqref{eq:sym-coint-on-e1+-}}= \tfrac{1}{2} \varepsilon
	\ ,
\end{align*}
and by linearity the invariant of $\hopflink{X_1^\varepsilon, 0}{X_0^\delta, 0}$ is then equal to $\varepsilon 2^{-N-1}$. Using the twist eigenvalue for $X_1^\varepsilon$ and the fact that $X_0^\pm$ has trivial twist, we get
\begin{align*}
\hopflink{X_1^\varepsilon,n}{X_0^\delta,m}: ~~\varepsilon 2^{-N-1} (\varepsilon \beta^{-1})^n \ .
\end{align*}

Next we compute $\hopflink{P,n}{U,m}$ for $U = X_1^\delta$. 
For the zero-framed Hopf link we have
\begin{align*}
	\hopflink{2^N X_1^\varepsilon,0}{X_1^\delta,0}: \quad
	\modTr_{2^N X_1^\varepsilon} ( \modS_{\cat}(\phi_{X_1^\delta})_{2^N X_1^\varepsilon} )
	&= 2^N \symRightCoint( (\eQ_1^+ - \eQ_1^-) \eQ_1^\varepsilon )
 \\
	&= 2^N \varepsilon \symRightCoint( \eQ_1^\varepsilon ) = 2^{N-1}
	\ ,
\end{align*}
whence the invariant of $\hopflink{X_1^\varepsilon,n}{X_1^\delta,m}$ is $2^{-1} \, (\varepsilon \beta^{-1})^n \, (\delta \beta^{-1})^m $.

Finally,
\begin{align*}
	\hopflink{P_0^\varepsilon,n}{X_1^\delta,0}: \quad
	\modTr_{P_0^\varepsilon} (\elRibbon^{-n} \modS_{\cat}(\phi_{X_1^\delta})_{P_0^\varepsilon} )
	&\overset{\eqref{eq:internal_characters_central_SF}}= 
	\delta \beta^2 4^N
	\symRightCoint( \elRibbon^{-n} \eQ_0^\varepsilon \genK \prod_{j=1}^N \genF^+_j \genF^-_j )
	\overset{(*)}= \delta 2^{N - 1}
	\ ,
\end{align*}
where the calculation for $(*)$ is as in \eqref{eq:hat-lam-e0+-top}, up to an overall factor of $\varepsilon$, together with the observation that from the ribbon element \eqref{eq:symp_ferm_ribbon} only $\eQ_0$ contributes, which acts as one on $\eQ_0^\varepsilon$. Thus we arrive at
\begin{align*}
\hopflink{P_0^\varepsilon,n}{X_1^\delta,m}: ~~ \delta 2^{N - 1} (\delta \beta^{-1})^m \ .
\end{align*}

\subsubsection*{Pullback trace}

The next remark shows that the invariant of $\hopflink{P_\parMat,n}{X_0^\pm,m}$ is $\pm 1$ times that of the framed unknot  $\framedUnknot{n}{P_\parMat}$, while the invariant of $\hopflink{P_\parMat,n}{X_1^\pm,m}$ is zero.

\begin{remark}\label{rem:link-only-knot-matters}
Consider any link $L$ with (at least) one of its components coloured by $P_\parMat$. In this situation the symplectic fermion example is somewhat special in the following sense.
By \Cref{cor:manifold_invariants_factor_through_Gr} it is enough to consider the case where the other components are labelled by simple objects. If any component is labelled by $X_1^\pm$, we can use that
$X_1^\pm$ is in the pullback ideal. So we may as well compute the invariant via the
pullback trace on the appropriate endomorphism of $X_1^\pm$.
But by \Cref{prop:pullback_trace_vanishes_on_proj_SF}, the pullback trace vanishes on
projective modules of the larger algebra. Hence to get a non-zero invariant, all other link components have to be decorated by $X_0^\pm$. But these objects are transparent in the degree 0 (see \Cref{rem:X-transparent}),
and so just result in overall factors of $\pm1$. Altogether, the invariant of $L$ is given by the invariant of the $P_\parMat$-coloured component times the product of the quantum dimensions of the colours of the other components.
\end{remark}

\subsection{Torus knot}
The torus knot $\torusknot{m}{X}$ for $m \in 2 \field[Z]+1$ is the braid closure of $m$
braidings.
Cutting open one strand results in the endomorphism $(\xi_m)_X$ of $X$, defined as (omitting all tensor products):
\begin{align}
	(\xi_m)_X := \big[ & \underbrace{X 
		\xrightarrow{\sim} \tensUnit X
		\xrightarrow{\coevR_X \,1_X} (\dualL{X} X) X
		\xrightarrow{\alpha^{-1}} \dualL{X} (X X)}_{=:F}
	\notag\\
	&\xrightarrow{1_{\dualL{X}} \,c_{X,X}^m} 
	\underbrace{\dualL{X} (X X)
		\xrightarrow{\alpha} (\dualL{X} X) X
		\xrightarrow{\evL_X \,1_X} \tensUnit X
		\xrightarrow{\sim} X}_{=:G} \big]~.
	\label{eq:xi-F-G-def}
\end{align}
Note that $\xi_m$ defines a natural isomorphism of the identity functor and is therefore given by the action of a central element of the quasi-Hopf algebra, which by abuse of notation we also call $\xi_m$, i.e.\ $(\xi_m)_X = \xi_m.(-)$. To compute $\xi_m$ in terms of the quasi-Hopf algebra data in \eqref{eq:qHopf-ev},
\eqref{eq:qRpR} and
\eqref{eq:pivot-qHopf-coev}, first note that the maps $F$ and $G$ in \eqref{eq:xi-F-G-def} read
\begin{align*}
	F : u \mapsto \sum_{i=1}^{\dim X} v^i \tensor \pL_1 \pivotQ\inv . v_i \otimes \pL_2.u 
	\quad,\quad
	G : w^* \otimes u \otimes v \mapsto 
	\langle w^* \mid \qL_1 . u \rangle \, \qL_2.v \ ,
\end{align*}
where $\{ v_i \}$ is a basis of $X$ and $\{ v^i \}$ the dual basis of $\dualL{X}$.
For the explicit computation it will be convenient to distinguish $m>0$ and $m<0$. Namely,
write $|m| = 2n+1$ and 
for $m>0$ set $\xi^+_n := \xi_{2n+1}$ and for $m<0$ set $\xi^-_n := \xi_{-(2n+1)}$. 
Then, explicitly,
\begin{align}
	\xi_n^+
	&= \qL_2 \rMatrix_1 (\monodromy^n)_1 \pL_1 
	\pivotQ\inv \qL_1 \rMatrix_2 (\monodromy^n)_2 \pL_2
	\ ,\label{eq:xi+-_n_def}
	\\
	\xi_n^-
	&= \qL_2 (\overline{\monodromy}^n)_2 \overline{\rMatrix}_2 \pL_1 
	\pivotQ\inv \qL_1 (\overline{\monodromy}^n)_1 \overline{\rMatrix}_1 \pL_2
	\ ,\nonumber
\end{align}
where $\overline{\rMatrix}$ and $\overline{\monodromy}$ are the multiplicative inverses
of $\rMatrix$ and $\monodromy$.

\subsubsection*{Categorical trace}
Let us start with $\tensIdeal = \cat$.
Using the explicit formulas for the braiding and its inverse, one immediately sees that
$\rMatrix$ and $\overline{\rMatrix}$ act as $\pm \oneQ \tensor \oneQ$ on $X_0^\pm
\tensor X_0^\pm$.
Since on $X_0^\pm$ also $\pivotQ$ acts as $\pm \oneQ$, we find that $\xi_n^\pm$ acts as
$\oneQ$, and so $\trCat[{\hmodM[\sympFerm]}]_{X_0^\pm}(\xi_n^\pm) = \trCat[{\hmodM[\sympFerm]}]_{X_0^\pm}(\id) = \pm 1$ as claimed in Table~\ref{table:invariants}.

\subsubsection*{Modified trace on projective ideal}

On the projective ideal, we used Mathematica to compute $\xi^\pm_n$ for small values of
$n$ and $N$, and then interpolated a (conjectural)
general formula from that.
Explicitly, we did the following, splitting the problem into odd and even sector.

On the even sector, the endomorphism spaces of $P_0^\pm$ have basis $\prod_{j=1}^N
(\genF^+_j)^{s_j} (\genF^-_j)^{t_j}$ with $s, t \in \field[Z]_2^N$ such that $|s| +
|t|$ is even. 
After multiplying with $\eQ_0$ (which acts as $\id$ on $P_0^\pm$), these elements are central \cite[Prop.\,3.15]{FGR2} and the endomorphism is given by acting with them.
The modified trace only sees the top component, i.e.\ the coefficient of the basis
vector  
with all $s_j$ and $t_j$ equal to $1$.
Thus, to compute the link invariant of $\torusknot{m}{P_0^\varepsilon}$, it is
enough to extract the top component of the element $\xi_m \eQ_0^\varepsilon$.
Small values -- more precisely, $N = 1$ for $0 \leq n \leq 10$, $N=2$ for $0 \leq n
\leq 5$, and $N = 3$ for $n = 0, 1$ -- suggest the following general formula
\begin{align*}
	\xi_n^\pm\big|_{\substack{\text{Top component,} \\ \text{acting on }
			P_0^\varepsilon}}
	= (\pm 2 (2n+1))^N \eQ_0 \prod_{j=1}^N \genF^+_j \genF^-_j
	\ ,
\end{align*}
and from this one computes
\begin{align*}
	\modTr_{P_0^\varepsilon} ( \xi_m )
	= (2m)^N 
	\symRightCoint(\eQ_0^\varepsilon \prod_{j=1}^N \genF^+_j \genF^-_j)
	\overset{\eqref{eq:hat-lam-e0+-top}}=\varepsilon m^N \tfrac{1}{2} \beta^{-2}
	\ .
\end{align*}
On the odd sector, projective modules are simple and thus $\xi_n^\pm$ act as scalars.
Small values ($N = 1$ for $0 \leq n \leq 10$, and $N=2$ for $0 \leq n
\leq 5$) suggest the following general formula
\begin{align*}
	\xi^\pm_n \big|_{\text{acting on } X_1^\varepsilon}
	= \varepsilon (2n + 1)^N \beta^{\pm (2 n - 1)}
	= \varepsilon m^N \beta^{m-2}
	\ ,
\end{align*}
where we used $m = \pm(2n+1)$ and $(\pm1)^N = \beta^{\pm2 - 2}$.
From this one computes
\begin{align*}
	\modTr_{X_1^\varepsilon} ( \xi_m )        
	= 2^{-N} \varepsilon m^N  \beta^{m-2}
	\symRightCoint(\eQ_1^\varepsilon)
	\overset{\eqref{eq:sym-coint-on-e1+-}}=
	2^{-N-1} m^N \beta^{m-2}
	\ .
\end{align*}

\subsubsection*{Pullback trace}

Let $P_\parMat$ be the lift of $P_0^+(A)$ defined in Section~\ref{sec:left-of-proj-A}.
We introduce a shorthand notation, which will allow for more compact statements and computations. For any $X \in \sympFerm$, $u \in P_\parMat$, consider the maps $\rho^{1,2}_{X,u} : \sympFerm \otimes \sympFerm \to P_\parMat$,
$$
    \rho^1_{X,u}(E \otimes F) = EXF.u \quad , \quad
    \rho^2_{X,u}(E \otimes F) = FXE.u \ .
$$
For two elements $A,B \in \sympFerm \otimes \sympFerm$ we will write
$$
A \approx B
~~\overset{\text{def.}}{\Longleftrightarrow}~~
\rho^{a}_{X,u}(A) = \rho^{a}_{X,u}(B) ~
\text{for all $X,\parMat,u$ and $a = 1,2$} \ .
$$

A helpful lemma, which is not hard to verify, is the following.

\begin{lemma}\label{lem:R-simple-PMPM}
We have
	\begin{align*}
		\prod_{j=1}^2 (x \oneQ \tensor \oneQ + y \genF^-_j \omega_\pm \tensor \genF^+_j)
		~\approx~& 
		x^2 \oneQ \tensor \oneQ 
		+ xy (1 + a^- b^+) \genF^-_1 \genK \tensor \genF^+_1
		\\ \quad
		&  + xy a^+ b^- \genF^+_1 \genK \tensor \genF^-_1
	\end{align*}
\end{lemma}


Let us first compute $\xi_n^+$ from~\eqref{eq:xi+-_n_def}.
Using the formula for the $R$-matrix from \Cref{sec:qHopf:sympf_ferm_elements}, one
finds:
\begin{lemma}\label{lem:R-calc}
We have
	\begin{align*}
		\rMatrix
		\approx
		\sum_{r, s \in \field[Z]_2}
		(-1)^{rs} 
		\big(
		\tfrac{1}{2} 
		\genK^r \tensor \genK^s 
		+ (1 + a^- b^+) 
		\genK^{r+1} \genF^-_1 \tensor \genK^s \genF^+_1
		+ a^+ b^- 
		\genK^{r+1} \genF^+_1 \tensor \genK^s \genF^-_1
		\big)
	\end{align*}
	and for $n \in \field[N]$
	\begin{align*}
		\monodromy^n
		\approx
		\oneQ \tensor \oneQ
		+ 2 n (1 + \det\parMat) 
		\genK \genF^-_1 \tensor \genF^+_1
		- 2 n (1 + \det\parMat)
		\genK \genF^+_1 \tensor \genF^-_1
		\ .
	\end{align*}
\end{lemma}

\begin{proof}
	We have
	\begin{align*}
		\rMatrix 
		&= \tfrac{1}{2} 
		\big(
		\sum_{n, m, r, s \in \field[Z]_2}
		\beta^{2nm} (-1)^{rs} i^{-rn+sm} \genK^r \eQ_n \tensor \genK^s \eQ_m
		\big)
		\prod_{j=1}^2 (\oneQ \tensor \oneQ - 2 \genF^-_j \omega_- \tensor \genF^+_j)
		\\
    & \hspace*{-8mm} \overset{\text{Lem.\,\ref{lem:R-simple-PMPM}}}\approx
		\tfrac{1}{2} 
		\big(
		\sum_{r, s \in \field[Z]_2}
		(-1)^{rs} \genK^r \tensor \genK^s 
		\big)
		\big(
		\oneQ \tensor \oneQ
		- 2 (1 + a^- b^+) \genF^-_1 \genK \tensor \genF^+_1
		- 2 a^+ b^- \genF^+_1 \genK \tensor \genF^-_1
		\big)
		\\
		=& 
		\sum_{r, s \in \field[Z]_2}
		(-1)^{rs} 
		\big(
		\tfrac{1}{2} 
		\genK^r \tensor \genK^s 
		+ (1 + a^- b^+) 
		\genK^{r+1} \genF^-_1 \tensor \genK^s \genF^+_1
		+ a^+ b^- 
		\genK^{r+1} \genF^+_1 \tensor \genK^s \genF^-_1
		\big)
	\end{align*}
	Thus
	\begin{align*}
		\monodromy
		&= \rMatrix_{21} \rMatrix
		\\
		&\approx
		\sum_{r, s \in \field[Z]_2}
		(-1)^{rs} 
		\big(
		\tfrac{1}{2} 
		\genK^s 
		\tensor 
		\genK^r 
		+ (1 + a^- b^+) 
		\genK^s \genF^+_1
		\tensor 
		\genK^{r+1} \genF^-_1 
		+ a^+ b^- 
		\genK^s \genF^-_1
		\tensor 
		\genK^{r+1} \genF^+_1 
		\big)
		\\ \times &
		\sum_{m, n \in \field[Z]_2}
		(-1)^{mn} 
		\big(
		\tfrac{1}{2} 
		\genK^m \tensor \genK^n 
		+ (1 + a^- b^+) 
		\genK^{m+1} \genF^-_1 \tensor \genK^n \genF^+_1
		+ a^+ b^- 
		\genK^{m+1} \genF^+_1 \tensor \genK^n \genF^-_1
		\big)
		\\
		&=
		\tfrac{1}{2} 
		\sum_{r, s, m, n\in \field[Z]_2}
		(-1)^{rs + mn} 
		\bigg[
		\genK^{m+s} \tensor \genK^{n+r}
		\big(
		\tfrac{1}{2} 
		\oneQ \tensor \oneQ
		+ (1 + a^- b^+) 
		\genK \genF^-_1 \tensor \genF^+_1
		\\ & \quad
		+ a^+ b^- 
		\genK \genF^+_1 \tensor \genF^-_1
		\big)
		+
		(-1)^{m+n}
		\genK^{s+m} \tensor \genK^{n+r}
		\big(
		(1 + a^- b^+) 
		\genF^+_1 \tensor \genK \genF^-_1
		\\ & \quad
		+ a^+ b^- 
		\genF^-_1 \tensor \genK \genF^+_1
		\big)
		\bigg]
	\end{align*}
	One can show that
	\begin{align*}
		\eQ_0 \tensor \eQ_0 \cdot
		\sum_{r, s, m, n\in \field[Z]_2}
		(-1)^{rs + mn} \genK^{m+s} \tensor \genK^{n+r}
		= 4 \eQ_0 \tensor \eQ_0
	\end{align*}
	and 
	\begin{align*}
		\eQ_0 \tensor \eQ_0 \cdot
		\sum_{r, s, m, n\in \field[Z]_2}
		(-1)^{rs + mn + m + n} \genK^{m+s} \tensor \genK^{n+r}
		= -4 \eQ_0 \genK \tensor \eQ_0 \genK
		\ ,
	\end{align*}
	whence
	\begin{align*}
		\monodromy
		&\approx
		2
		\bigg[
		\big(
		\tfrac{1}{2} 
		\oneQ \tensor \oneQ
		+ (1 + a^- b^+) 
		\genK \genF^-_1 \tensor \genF^+_1
		+ a^+ b^- 
		\genK \genF^+_1 \tensor \genF^-_1
		\big)
		\\ & \quad
		-
		\genK \tensor \genK
		\big(
		(1 + a^- b^+) 
		\genF^+_1 \tensor \genK \genF^-_1
		+ a^+ b^- 
		\genF^-_1 \tensor \genK \genF^+_1
		\big)
		\bigg]
		\\
		&=
		\oneQ \tensor \oneQ
		+ 2 (1 + \det\parMat) 
		\genK \genF^-_1 \tensor \genF^+_1
		- 2 (1 + \det\parMat)
		\genK \genF^+_1 \tensor \genF^-_1
	\end{align*}
The expression for $\monodromy^n$ from the statement now follows by induction.
\end{proof}

\begin{lemma}
	For $n \in \field[N]$ we have
    \begin{align*}
		\pullbackTrace[{\modTr^A}]_{P_\parMat} \big(\xi^+_n\big) 
		~=~
		2 (2n + 1) (1 + \det\parMat) \pullbackTrace[{\modTr^A}]_{P_\parMat} \big(\genF^+_1 \genF^-_1\big)
	\end{align*}
\end{lemma}

\begin{proof}
	We  first compute $\rMatrix \monodromy^n$ using \Cref{lem:R-calc}:
	\begin{align*}
		\rMatrix
		\monodromy^n
		& \approx
		\sum_{r, s \in \field[Z]_2}
		(-1)^{rs} 
		\big(
		\tfrac{1}{2} 
		\genK^r \tensor \genK^s 
		+ (1 + a^- b^+) 
		\genK^{r+1} \genF^-_1 \tensor \genK^s \genF^+_1
		\\ &
		+ a^+ b^- 
		\genK^{r+1} \genF^+_1 \tensor \genK^s \genF^-_1
		\big)
		\times
		\big(
		\oneQ \tensor \oneQ
		+ 2 n (1 + \det\parMat) 
		\genK \genF^-_1 \tensor \genF^+_1
		\\ &
		- 2 n (1 + \det\parMat)
		\genK \genF^+_1 \tensor \genF^-_1
		\big)
		\\ 
		& \approx
		\sum_{r, s \in \field[Z]_2}
		(-1)^{rs} 
		\times
		\big(
		\tfrac{1}{2} 
		\genK^r \tensor \genK^s 
		+ (n (1 + \det\parMat) + 1 + a^- b^+) 
		\genK^{r+1} \genF^-_1 \tensor \genK^s \genF^+_1
		\\ & \quad
		- (n (1 + \det\parMat) - a^+ b^-)
		\genK^{r+1} \genF^+_1 \tensor \genK^s \genF^-_1
		\big)
		\ .
	\end{align*}
	By definition in \eqref{eq:xi+-_n_def}, $\xi_n^+
		= \qL_2 \rMatrix_1 (\monodromy^n)_1 \pL_1 
		\pivotQ\inv \qL_1 \rMatrix_2 (\monodromy^n)_2 \pL_2$,
	and since we will evaluate the right symmetrised cointegral of the smaller algebra
	$A$ on $\eQ_0^+ \xi_n^+$, we see that we can also disregard the $\genK^r \tensor
	\genK^s$ terms in the above expression for $\rMatrix \monodromy^n$.
	Indeed, for these terms, no $\genF$'s will appear in $\xi^+_n \eQ_0^+$, and so these
	terms don't contribute to the final result.
	
	Note now that on the tensor product of even sectors, $\pL$ and $\qL$ from
	\eqref{eq:qRpR_SF} act as the identity, and also $\eQ_0 \pivotQ = \eQ_0 \genK$.
	Thus, abbreviating $\mathrm{T}(X) = \pullbackTrace[{\modTr^A}]_{P_\parMat} (X)$, we get 
	\begin{align*}
		\mathrm{T}(\xi_n^+)
		&=
        \mathrm{T}\bigl(
		\rMatrix_1 (\monodromy^n)_1 \genK \rMatrix_2 (\monodromy^n)_2\bigr) 
		\\
		&=
        \sum_{r, s \in \field[Z]_2}
        (-1)^{rs}\mathrm{T} 
		\big(
		(n (1 + \det\parMat) + 1 + a^- b^+) 
		\genK^{r+1} \genF^-_1 \genK \genK^s \genF^+_1
		\\ & \hspace{10em}
		- (n (1 + \det\parMat) - a^+ b^-)
		\genK^{r+1} \genF^+_1 \genK \genK^s \genF^-_1
		\big)
		\\
		&=
		\sum_{r, s \in \field[Z]_2}
        (-1)^{rs + s}
		(2n + 1) (1 + \det\parMat) \mathrm{T}\big(
		\genK^{r+s} \genF^+_1 \genF^-_1 \big)
		\ .
	\end{align*}
Noticing
		$\eQ_0 \sum_{r, s \in \field[Z]_2}
		(-1)^{rs + s}
		\genK^{r+s} 
		= 2 \eQ_0$,
	 the claim follows.
\end{proof}

As a result of all of the computations, we finally find
that the $\relativeProjectives{A}{\sympFerm}$-based invariant of the $P_\parMat$-coloured
torus knot $\torusknot{2n+1}{P_\parMat}$ is
\begin{align*}  
    \pullbackTrace[{\modTr^A}]_{P_\parMat} (\xi^+_n)
	&\overset{\eqref{eq:pullback_trace_choice}}= 
	2 (2n + 1) (1 + \det\parMat) 
	\,\frac{4}{1 - i \beta^2}\,
	\modTr^{A,r}_{P_0^+(A)}(\genF^+_1 \genF^-_1)
	\\
	&\overset{\eqref{eq:P0A+-mod-trace}}= 2 (2n + 1) (1 + \det\parMat) \ .
\end{align*}

Completely analogous computations for the inverse braiding and the associated monodromy
show that $\pullbackTrace[{\modTr^A}]_{P_\parMat} (\xi_n^-) = - \pullbackTrace[{\modTr^A}]_{P_\parMat} (\xi_n^+)$.

\section{Lens space invariants for symplectic fermions}
\label{sec:lens_space_symp_ferm}
In order to do the computation, we will first make some observations
which simplify the problem and do not depend on the specific case of symplectic
fermions, and then we proceed with the actual computation.

\subsection{Lens space invariants for quasi-Hopf algebras}

Here we will rewrite the categorical expression in \Cref{prop:invariants_lens_spaces_general} in the case $\cat = \hmodM[H]$ for a unimodular and twist-nondegenerate (recall \eqref{eq:qHopf-twistnondeg}) ribbon quasi-Hopf algebra $H$.

We start with the morphism $f(a) : \tensUnit \to \coend$ from \eqref{eq:definition_morphism_representing_Lpq}. Recall from 
\eqref{eq:isomorphism_EndidC_Hom1L}
the isomorphism $\rho \circ
\psi \colon \End(\id_{\cat}) \to \cat(\tensUnit, \coend)$. The natural endomorphisms of the identity functor are in bijection (as an algebra) with the centre of $H$.
If we call this last bijection $\xi$, then altogether
\begin{align}\label{eq:ZH-EndId-C1L}
	Z(H) \xrightarrow{~\xi~} \End(\id_{\cat}) \xrightarrow{~\rho \circ
		\psi~} \cat(\tensUnit, \coend) \ .
\end{align}
In \eqref{eq:SC-TC-explict} we transported the endomorphisms $\modS_\coend \circ (-)$, $\modT_\coend \circ (-)$ of $\cat(\tensUnit, \coend)$ to endomorphisms $\modS_{\cat}$, $\modT_{\cat}$ of  $\End(\id_{\cat})$. 
We denote their further transport to $Z(H)$ by $\modS_{Z}$ and $\modT_{Z}$. Explicit expressions for $\modS_{Z},\modT_{Z} :Z(H) \to Z(H)$ in terms of quasi-Hopf data can be found in \cite[Thm.\,8.1]{FGR1}.\footnote{
    The formulas in \cite{FGR1} were computed in the factorisable case.
    However, the same computation also works without assuming factorisability, resulting in the same expressions for the endomorphisms $\modS_Z$ and $\modT_Z$ of $Z(H)$.
    But they do not give a projective $\slTwoZ$-action unless $H$ is factorisable.
}
For example, $\modT_{Z}(z) = \ribbon\inv.z$, while the expression for $\modS_{Z}$ is more involved. As noted in \Cref{rem:modular_actions}\,(3), for an $H$-module $V$ we set 
$\phi_V = (\rho \circ \psi)\inv(\chi_V)$, and we have $\chi_\tensUnit = \unitL$.
We denote $\phiQ{V} := \xi^{-1}(\phi_V) \in Z(H)$, so that \eqref{eq:ZH-EndId-C1L} maps
\begin{align}\label{eq:phi-phi-chi-qHopf}
	\phiQ{V} \longmapsto \phi_V \longmapsto \chi_V~.
\end{align}
See \cite[Sec.\,7.6]{FGR1} for an explicit quasi-Hopf algebra formula for $\phiQ{V}$.

Let us write 
$\widehat{f(a)} := (\rho \circ \psi\circ \xi)^{-1}(f(a)) \in Z(H)$. 
Using the above notation, we obtain
\begin{align}\label{eq:hat-fa}
	\widehat{f(a)} = 
	\big(
	\prod_{i=n}^1 
	\modT_{Z}^{a_i} \circ \modS_{Z} 
	\big)
    (\phiQ{\tensUnit})
 \ .
\end{align}

Given a tensor ideal $\tensIdeal \subset \cat$, a modified trace $\modTr$ on $\tensIdeal$ and an object $P \in \tensIdeal$, we obtain a (possibly degenerate) symmetric and invariant bilinear form on $Z(H)$ via
\begin{align}
	(z, z')_P
	= \modTr_{P}(z z')  \ .
	\label{eq:Z(H)-pairing-via-trace_general}
\end{align}
We can now reformulate the categorical expressions from \Cref{prop:invariants_lens_spaces_general} as follows:

\begin{proposition}\label{prop:lens-quasi-Hopf-general}
	In terms of quasi-Hopf data, the invariants of $\mathfrak{L}_{x}(\alpha, P)$, $x \in \{*,\circ\}$ from \Cref{prop:invariants_lens_spaces_general} read
	\begin{align}
		\invDGGPR(\mathfrak{L}_{x}(\alpha, P))
		= \anomaly^{-\sigma} \DD^{-1 -n} 
		(\widehat\alpha,\zeta_x)_P
		\notag
	\end{align}
	where $\widehat\alpha \in Z(H)$ represents the natural transformation $\alpha$ of the identity functor and 
	\begin{align}
		\zeta_* &= \langle \intLyu | \alphaQ \widehat{f(a)} \rangle \cdot \oneQ \ ,
		\notag
		\\
		\zeta_\circ &= 
		\langle
		\intLyu |
		S(\coassQ_1 \invCoassQ_1) 
		\alphaQ \widehat{f(a)}
		\coassQ_2 \monodromy_1 \invCoassQ_2
		\rangle
		\cdot \coassQ_3 \monodromy_2 \invCoassQ_3
		\ .
		\notag
	\end{align}
	Here, $\intLyu\in H^*$ is the integral of $\coend$, see \Cref{sec:Lyu-coend-int-for-qHopf}.
\end{proposition}

\begin{proof}
	Let $z \in Z(H)$ and define $E(z) \colon \coend \to \coend$ by $E(z) \circ \dinatLyu_X = \dinatLyu_X \circ (\id \tensor \xi(z)_X)$.
 The dinatural transformation $\dinatLyu_X \colon X^* \otimes X \to \coend$ of $\coend$ is explicitly given in \eqref{eq:dinatCoend_qHopf}.
 From this it is easy to check that for $f \in \coend = H^*$ we have
	$E(z)(f) = f((-) \cdot z)$.
	
	Using the definition of the coproduct $\Delta_\coend$ (see \cite[Sec.\,3.3]{FGR1}), one can verify the identity $(E(z) \otimes \id) \circ \Delta_\coend = (\id \otimes E(z)) \circ \Delta_\coend$.
	
	The morphism $g := \psi(\xi(z)) \in \cat(\coend,\tensUnit) = H^{**}$ is defined by the dinatural transformation
	$g \circ \dinatLyu_X = \ev_X \circ (\id \otimes \xi(z)_X)$. Comparing to the definition of the counit $\counitL$ and to that of $E(z)$ above, we get $\psi(\xi(z)) = \counitL \circ E(z)$.
	The isomorphism $\rho : \cat(\coend,\tensUnit) \to \cat(\tensUnit,\coend)$ acts on $g$ as \cite[Sec.~2]{GR-proj}
	\begin{align*}
		\rho(g) = \big[ \tensUnit \xrightarrow{~\Lambda~} \coend \xrightarrow{~\Delta_\coend~} \coend \coend \xrightarrow{~g \otimes \id~} \tensUnit \coend \xrightarrow{~\sim~} \coend \big] \ .
	\end{align*}
	But $(g \otimes \id) \circ \Delta_\coend = ( \counitL \otimes \id) \circ (E(z) \otimes \id) \circ \Delta_\coend
	= ( \counitL \otimes E(z)) \circ \Delta_\coend$. Hence altogether
	\begin{align}
		\rho(\psi(\xi(z))) = E(z) \circ \Lambda = \langle \intLyu | (-) \cdot z \rangle ~.
		\label{eq:Ez-Lam}
	\end{align}
	After these preparations, we can compute $\zeta_*$ and $\zeta_\circ$:
	\begin{itemize}
		\item[$\zeta_*$:] We have
		\begin{align}
			\counitL \circ f(a)
			\overset{\eqref{eq:Ez-Lam}}=
			\counitL \circ E(\widehat{f(a)}) \circ \Lambda
			=
			\langle \intLyu | \alphaQ \widehat{f(a)} \rangle ~,
			\label{eq:eps-f(a)=Lam-alpha-hat-f(a)}
		\end{align}
		where the last equality follows from the explicit form of 
		$\counitL \in \cat(\coend,\tensUnit)=H^{**}$,
		which acts by evaluating an element of $H^*$ on $\alphaQ$  (see \cite[Sec.\,3.3]{FGR1}).
		
		\item[$\zeta_\circ$:]
		We need to compute $\modTr_P(\alpha_P \circ Q_P \circ (f(a) \tensor \id_P))$, where $Q \colon \coend \tensor
		\id_{\cat} \To \id_{\cat}$ is the natural transformation defined in
		\eqref{eq:definition_natural_transformation_Lid_id}. 
		It is not hard to see that the action of $Q_V \colon \coend \tensor V \to V$ on $f \tensor v \in \coend \tensor V$ is given by
		\begin{align}
			Q_V (f \tensor v) 
			= 
			\langle
			f | S(\coassQ_1 \invCoassQ_1) \alphaQ \coassQ_2 \monodromy_1 \invCoassQ_2
			\rangle
			\cdot
			\coassQ_3 \monodromy_2 \invCoassQ_3 . v
			\ .
		\end{align}
		Substituting $f = f(a) = E(\widehat{f(a)}) \circ \Lambda = \langle \intLyu | (-) \cdot \widehat{f(a)} \rangle$ and using that $\widehat{f(a)}$ is central gives the expressions as in the statement of the proposition. 
	\end{itemize}
\end{proof}

\subsection{The central element \texorpdfstring{$\widehat{f(a)}$}{f(a)} for symplectic fermions}

After rewriting the lens space invariants  for a general 
    unimodular and twist-nondegenerate
ribbon quasi-Hopf algebra in \Cref{prop:lens-quasi-Hopf-general}, in this section we specialise to the symplectic fermion quasi-Hopf algebra $H = \sympFerm(N, \beta)$, and we will work out the elements $\widehat{f(a)}$
introduced in~\eqref{eq:hat-fa}.

Recall from \eqref{eq:intLyu_SF_recap} our normalisation of the integral $\intLyu :
\tensUnit \to \coend$ in $\cat = \hmodM$,
\begin{align}
	\intLyu
	= \beta^6 2^{-(N-1)} \ \TOP^*
	\ ,
	\label{eq:Lyu-int-SF-with-nu=1}
\end{align}
where we abbreviated 
$$
\TOP = \prod_{j=1}^N \genF^+_j \genF^-_j \ .
$$
(When projected to $\sympFerm_0$ via $\eQ_0$, this is the top element of the Grassmann algebra, hence the name.)

We will need the internal characters $\phiQ{V}$ from \eqref{eq:phi-phi-chi-qHopf} for $\sympFerm(N, \beta)$. 
For $V = X_0^\pm, X_1^\pm$, they were given in \eqref{eq:phi-SF-explicit_X}. And since $\phiQ{V}$ only depends on the class of $V$ in the Grothendieck ring, for $V = P_0^\pm$ we have
\begin{align}\label{eq:phi-SF-explicit_P0}
	\phiQ{P_0^\pm} &= 2^{2N-1} \phiQ{X_0^+} + 2^{2N-1} \phiQ{X_0^-} \ .
\end{align}
Recall from \cite[Sec.~5]{FGR2} that the centre $Z(H)$ of $H$ can be decomposed as
\begin{align}\label{eq:SF-centre-sum}
	Z(H) = Z_P \oplus Z_{\mathbf{\Lambda}} 
\end{align}  
where the two summands have bases
\begin{align}
	Z_P ~&:~\left\{ \phiQ{P_0^+},\ \phiQ{X_1^+},\ \phiQ{X_1^-} \right\} ~,
	\nonumber\\
	Z_{\mathbf{\Lambda}} ~&:~\left\{
	\eQ_0 \prod_{j=1}^N (\genF^+_j)^{s_j} (\genF^-_j)^{t_j}
	~|~ s, t \in \field[Z]_2^N, ~ |s| + |t| \in 2\field[Z]
	\right\}
	\ .
	\label{eq:basis_center_SF}
\end{align}
The dimension of the centre is $\dim Z(H) = \dim Z_P + \dim Z_{\mathbf{\Lambda}} = 3 + 2^{2N-1}$.

We fix
\begin{align*}
	\tfrac{p}{q} = [a_n; a_{n-1}, \ldots, a_1]
\end{align*}
subject to the positivity condition \eqref{eq:cont-frac-pos-assum}.

\begin{proposition}
	\label{prop:central_element_SF_computed}
	The central element $\widehat{f(a)} = \big( \prod_{i=n}^1 \modT_{Z}^{a_i} \circ
	\modS_{Z} \big) (\phiQ{\tensUnit})$ is given by
	\begin{align*}
		\widehat{f(a)} = 
		\widehat{f(a)}\big|_{Z_{P}} + 
		\widehat{f(a)}\big|_{Z_{\mathbf{\Lambda}}} ~,
	\end{align*}    
	where
	\begin{align*}
		\widehat{f(a)}\big|_{Z_{P}}
		&=
		2^{-2N}
		\big(
		c_n^0 \phiQ{P_0^+} + c_n^+ \phiQ{X_1^+} + c_n^- \phiQ{X_1^-}
		\big) \ ,
		\\
		\widehat{f(a)}\big|_{Z_{\mathbf{\Lambda}}}
		&=
		2^N \beta^{6n + 2}
		\sum_{s \in \field[Z]_2^N}
		p^{|s|}
		\left( \frac{q}{2} \right)^{N - |s|}
		\eQ_0 \prod_{j=1}^N \big( \genF^+_j \genF^-_j \big)^{s_j} 
		\ .
	\end{align*}
 The coefficients $c_i^{0,\pm}$ are determined recursively by
	\begin{align}
		c_i^0 &= 2^{-N} ( c_{i-1}^+ - c_{i-1}^- ) \ ,
		\nonumber \\
		c_i^\pm &= \frac{1}{2} (\pm 1)^{a_i} \beta^{-a_i} 
		(c_{i-1}^+ + c_{i-1}^- \pm 2^N c_{i-1}^0) \ ,
		\label{eq:coefficients_ST_projective center}
	\end{align}
	for $1 \leq i \leq n$, with initial values $c_0^0 = 1$ and $c_0^\pm = 0$.
\end{proposition}

The proof of this proposition will be given in the remainder of this subsection.
The decomposition \eqref{eq:SF-centre-sum} respects the 
$\slTwoZ$-action on $Z(H)$, i.e.\ $Z_P$ and $Z_{\mathbf{\Lambda}}$ are subrepresentations.
The central element
$\phiQ{\tensUnit}=\phiQ{X_0^+}$
from \eqref{eq:phi-SF-explicit_X} splits according to the direct sum decomposition \eqref{eq:SF-centre-sum} as
\begin{align}
	\phiQ{\tensUnit} = 2^{-2N} \phiQ{P_0^+} + 2^N \beta^2 \eQ_0 \TOP
	\ .
	\label{eq:phi1-in-basis}
\end{align}

\begin{proof}[Proof of \Cref{prop:central_element_SF_computed} Part 1]  
	In the ordered basis of $Z_P$ given above, the $S$- and $T$-generators act via the
	matrices (see \cite[Lem.\,5.1]{FGR2})
	\begin{align}
		\modS_{Z_P} = 
		\begin{pmatrix}
			0 & 2^{-N} & -2^{-N} \\
			2^{N-1} & \frac{1}{2} & \frac{1}{2} \\
			-2^{N-1} & \frac{1}{2} & \frac{1}{2}
		\end{pmatrix}
		\qandq
		\modT_{Z_P} =
		\begin{pmatrix}
			1 & 0 & 0 \\
			0 & \beta\inv & 0 \\
			0 & 0 & -\beta\inv
		\end{pmatrix}
		\ ,
		\notag
	\end{align}
	and it is not too hard to see that the coefficients in
	\begin{align}
		\big(
		\prod_{i=n}^1 \modT_{Z_P}^{a_i} \circ \modS_{Z_P}
		\big)
		(\phiQ{P_0^+})
		= c_n^0 \phiQ{P_0^+} + c_n^+ \phiQ{X_1^+} + c_n^- \phiQ{X_1^-}
		\ ,
		\notag
	\end{align}
	satisfy the recursive relation
	given in \eqref{eq:coefficients_ST_projective center}.
\end{proof}

\begin{remark}\label{rem:form-of-c0}
	For later use we note that the recursion relation \eqref{eq:coefficients_ST_projective center} implies that $c_i^0 \in \{0\} \cup \{ \beta^j \,|\, j \in \field[Z] \}$ for all $i$. To see this set $g_i = c_i^0$ and $h_i = 2^{-N}(c_i^+ + c_i^-)$. Then $g_0 = 1$, $g_1=0$, $h_0 = 0$, and the recursion now reads:
	\begin{align*}
		g_i = \frac{c^+_{i-1} - c^-_{i-1}}{2^{N}} = \beta^{-a_{i-1}}  \begin{cases}
			g_{i-2} &; a_{i-1} \text{ even}
			\\
			h_{i-2} &; a_{i-1} \text{ odd}
		\end{cases}
		~,~~
		h_i = \beta^{-a_{i}}  \begin{cases}
			h_{i-1} &; a_{i} \text{ even}
			\\
			g_{i-1} &; a_{i} \text{ odd}
		\end{cases}
	\end{align*}
	Thus the condition $g_i, h_i \in  \{0\} \cup \{ \beta^j \,|\, j \in \field[Z] \}$ is satisfied by the initial conditions and preserved by the recursion.
\end{remark}

To compute the action of $\prod_{i=n}^1 \modT_{Z}^{a_i} \circ \modS_{Z}$ 
on $\eQ_0 \TOP$, let $1 \leq j \leq N$ and let $Z_{\textup{ev},
	j}$ be the vector space with basis $\eQ_0 \genF^+_j \genF^-_j$ and $\eQ_0$.
We consider it as an algebra in the obvious way.
Denote by $Z_{\textup{ev}}$ the subspace of $Z_{\mathbf{\Lambda}}$ generated (as an
algebra with unit $\eQ_0$) by $\eQ_0 \genF^+_j \genF^-_j$, $1 \leq j \leq N$.
There is an algebra isomorphism
\begin{align}
	\kappa \colon \bigotimes_{j=1}^N Z_{\textup{ev}, j} \to Z_{\textup{ev}}
	, \quad x_1 \tensor \ldots \tensor x_N \mapsto x_1 \cdot \ldots \cdot x_N
	\ .
	\notag
\end{align}
Then by \cite[Lem.~5.2]{FGR2}, the $\slTwoZ$-action on $Z_{\mathbf{\Lambda}}$ restricts
to an action on $Z_{\textup{ev}}$, which can be described as
\begin{align}
	\modS_{Z_{\mathsf{\Lambda}}} \big|_{Z_{\textup{ev}}} 
	= \kappa \circ \bigg( \beta^2 \cdot \bigotimes_{j=1}^N S \bigg) \circ \kappa\inv
	\qandq
	\modT_{Z_{\mathsf{\Lambda}}} \big|_{Z_{\textup{ev}}} 
	= \kappa \circ \bigg( \bigotimes_{j=1}^N T \bigg) \circ \kappa\inv
	\ ,
	\notag
\end{align}
where $S$ and $T$ are the matrices
\begin{align}
	S = \begin{pmatrix} 0 & 2 \\ - \frac{1}{2} & 0 \end{pmatrix}
	\qandq
	T = \begin{pmatrix} 1 & 2 \\ 0 & 1 \end{pmatrix}
	\notag
\end{align}
in our chosen basis $\{\eQ_0 \genF^+_j \genF^-_j, \eQ_0\}$ of $Z_{\textup{ev}, j}$ for
each $j$.

\begin{lemma}
	\label{prop:fa_on_non-proj_center}
	Let $b_1, \ldots, b_n$ be non-zero integers, $n \geq 1$.
	Then
	\begin{align}
		&
		\big(
		\prod_{i=n}^1
		\modT_{Z_{\mathbf{\Lambda}}}^{b_i} \circ
		\modS_{Z_{\mathbf{\Lambda}}}
		\big)
		(\eQ_0 \TOP)
		\notag \\ 
		&=
		\beta^{6n} 
		\bigg(
		\frac{1}{2}
		\prod_{j=1}^{n-1} 
		[b_j; b_{j-1}, \ldots, b_1]
		\bigg)^N
		\sum_{s \in \field[Z]_2^N}
		\big(
		2 [b_n; b_{n-1}, \ldots, b_1] 
		\big)^{|s|}
		\eQ_0
		\prod_{j=1}^N
		\big( \genF^+_j \genF^-_j \big)^{s_j}
		\ .
		\label{eq:cont_fract_rewritten}
	\end{align}
\end{lemma}

\begin{proof}
We claim first that the left hand side of \eqref{eq:cont_fract_rewritten} equals
	\begin{align}
		(-1)^{Nn} \beta^{2n} 
		2^{-N}
		\bigg(
		\prod_{j=1}^{n-1} 
		[b_j; b_{j-1}, \ldots, b_1]
		\bigg)^N
		\prod_{j=1}^N
		\big(
		\eQ_0 + 2 [b_n; b_{n-1}, \ldots, b_1] \eQ_0 \genF^+_j \genF^-_j
		\big)
		\ ,
		\notag
	\end{align}
	and we will proceed now to prove this via induction.
	Clearly $\eQ_0 \TOP \in Z_{\textup{ev}}$.
	For any $m \in \field[N]$, we have 
	\begin{align*}
		T^m S = 
		\begin{pmatrix} -m & 2 \\ -\tfrac12 & 0 \end{pmatrix} \ .
	\end{align*}    
	This implies the base case for induction on $n$, i.e. we have
	\begin{align}
		\big( \modT^{b_1}_{Z_{\mathbf{\Lambda}}} \circ \modS_{Z_{\mathbf{\Lambda}}} \big)
		(\eQ_0 \TOP)
		=
		(-1)^N
		\beta^2 
		2^{-N}
		\prod_{j=1}^N \big( \eQ_0 + 2 b_1 \eQ_0 \genF^+_j \genF^-_j \big)
		\ .
		\notag
	\end{align}
	Assume now that it is true up to $n$.
	Then
	\begin{align}
		& 
		\left(
		(-1)^{Nn} \beta^{2n} 
		2^{-N}
		\bigg(
		\prod_{j=1}^{n-1} 
		[b_j; b_{j-1}, \ldots, b_1]
		\bigg)^N
		\right)\inv
		\big(
		\prod_{i=n+1}^1
		\modT_{Z_{\mathbf{\Lambda}}}^{b_i} \circ
		\modS_{Z_{\mathbf{\Lambda}}}
		\big)
		(\eQ_0 \TOP)
		\notag \\ \quad
		&=
		\big( 
		\modT_{Z_{\mathbf{\Lambda}}}^{b_{n+1}} \circ \modS_{Z_{\mathbf{\Lambda}}}
		\big)
		\prod_{j=1}^N
		\big(
		\eQ_0 + 2 [b_n; b_{n-1}, \ldots, b_1] \eQ_0 \genF^+_j \genF^-_j
		\big)
		\notag \\ \quad
		&=
		\beta^2
		\prod_{j=1}^N
		\big(
		2 \eQ_0 \genF^+_j \genF^-_j
		+ 2 [b_n; b_{n-1}, \ldots, b_1] 
		\big(
		- b_{n+1} \eQ_0 \genF^+_j \genF^-_j
		- \frac{1}{2} \eQ_0 
		\big)
		\big)
		\notag \\ \quad
		&\overset{(*)}=
		(-1)^N \beta^2
		\big( [b_n; b_{n-1}, \ldots, b_1] \big)^N
		\prod_{j=1}^N
		\big(
		\eQ_0 + 2 [b_{n+1}; b_{n}, \ldots, b_1] \eQ_0 \genF^+_j \genF^-_j
		\big)
		\ ,
		\notag
	\end{align}
	as claimed.
	In step $(*)$ we used that by definition of the continued fraction in \eqref{eq:continued-fraction-notation} we have
	\begin{align*}
		[b_{n+1}; b_{n}, \ldots, b_1]
		= b_{n+1} - \frac{1}{[b_{n}; b_{i-1}, \ldots, b_1]} \ .
	\end{align*}
	Multiplying out the product
	and using $\beta^4 = (-1)^N$
	yields \eqref{eq:cont_fract_rewritten}.
\end{proof}

\begin{proof}[Proof of \Cref{prop:central_element_SF_computed} Part 2]
Because of the assumptions on the continued fraction,
\Cref{prop:continued_fractions_main_statement} gives 
$$
 \prod_{j=1}^{n-1} [a_j; a_{j-1}, \ldots, a_1]
	= \frac{q}{p} \prod_{j=1}^{n} [a_j; a_{j-1}, \ldots, a_1] = q\ .
 $$
	The formula then follows from \Cref{prop:fa_on_non-proj_center}.
\end{proof}

\subsection{Categorical trace: Lyubashenko invariant}\label{sec:cat-tr-Lyu-Lens}

By \eqref{eq:lens-Lyu=epsof(a)}, it remains to compute $\counitL \circ f(a)$ which we know from \eqref{eq:eps-f(a)=Lam-alpha-hat-f(a)} to be equal to $\langle \intLyu | \widehat{f(a)} \rangle$ (recall that $\alphaQ=\oneQ$ for $\sympFerm(N, \beta)$).

Since by \eqref{eq:Lyu-int-SF-with-nu=1}, $\intLyu$ is proportional to $\TOP^*$, it can only see part of
$\widehat{f(a)}$ which lives in $Z_{\mathbf{\Lambda}}$.
Using \Cref{prop:central_element_SF_computed} we can simply read off
\begin{align}
	\langle \intLyu | \widehat{f(a)} \rangle
	& =
	2 
	\beta^{6n}
	\sum_{s \in \field[Z]_2^N}
	p^{|s|}
	\left( \frac{q}{2} \right)^{N - |s|}
	\langle
	\TOP^* |
	\eQ_0 \prod_{j=1}^N \big( \genF^+_j \genF^-_j \big)^{s_j}
	\rangle
	= \beta^{6n} p^N
	\notag
\end{align}
Combining this with $\DD = 1$ and $\anomaly = \beta^{-2}$ from \Cref{sec:SF-int-coint}, altogether \eqref{eq:lens-Lyu=epsof(a)} gives
\begin{align}
	\invLyu(\mathfrak{L}(p, q))
	= p^N
	\ .
	\label{eq:Lens-Lyu-SF}
\end{align}

\subsection{Modified trace: Renormalised Lyubashenko invariant}

Here we compute the renormalised Lyubashenko invariant $\invDGGPR(\mathfrak{L}_{x}(\alpha, P))$, $x \in \{ *,\circ \}$ with $\modTr$ the modified trace on the projective ideal. 
To do so, according to \Cref{prop:lens-quasi-Hopf-general} we need to compute the central elements $\zeta_x$ and the pairings $(-,-)_U$. 

We start with the explicit expressions for $\zeta_x$ for which we will need two preliminary lemmas. For $z \in Z(\sympFerm)$ write
\begin{align*}
	F(z) =       \langle
	\intLyu |
	S\inv(\coassQ_1 \invCoassQ_1) 
	\alphaQ \,z\,
	\coassQ_2 \monodromy_1 \invCoassQ_2
	\rangle
	\cdot \coassQ_3 \monodromy_2 \invCoassQ_3 \ .
\end{align*}
Note that $F(\widehat{f(a)}) = \zeta_\circ$.

\begin{lemma}
	$F(z) = \langle
	\intLyu | z \monodromy_1 \rangle
	\cdot \monodromy_2$
\end{lemma}

\begin{proof}
	From the expression \eqref{eq:coassociator_SF} for the coassociator of $\sympFerm$, we
	immediately see that the claim is true if $z$ is in the even sector, and
	therefore we must only show it for $z = \eQ_1 z$. 
	The first and second tensor factor of $\coassQ^{\pm1}$ in $F(z)$ are then multiplied by $\eQ_1$, and we have
	\begin{align*}
		(\eQ_1 \otimes \eQ_1 \otimes \oneQ) \cdot \coassQ^{\pm1}
		= \eQ_1 \otimes \eQ_1 \otimes 
		\big(
		\eQ_0 \genK^N 
		+ \eQ_1 \betaQ_\pm
		\big)\ .
	\end{align*}    
	Inserting this simplified expression in $F(z)$ gives
	\begin{align}
		F(z) &=
		\langle \intLyu | z \monodromy_1 \rangle 
		\big(
		\eQ_0 \genK^N 
		+ \eQ_1 \betaQ_+
		\big)
		\monodromy_2
		\big(
		\eQ_0 \genK^N 
		+ \eQ_1 \betaQ_-
		\big)
		\nonumber\\
		&=
		\langle \intLyu | z \monodromy_1 \rangle 
		\big(
		\eQ_0 \genK^N \monodromy_2 \genK^N
		+ \eQ_1 \betaQ_+ \monodromy_2 \betaQ_-
		\big)
		\ ,
	\end{align}
	where in the second step we used that $\eQ_0, \eQ_1$ are central. 
	Since $z$ is central, it in particular commutes with $\genK$ and hence contains an even number of $\genF$s. 
By~\eqref{eq:Lyu-int-SF-with-nu=1}, $\intLyu$ is a multiple of $\TOP^*$
    and so the amount of $\genF$s in the monodromy factors that contribute to the above expression is even as well.
	It follows that the non-zero summands $\monodromy_2$ commute with $\genK$ and therefore also with $\genK^N$ and $\betaQ_\pm$. 
	Since $\genK^{2N} \eQ_0 = \eQ_0$ and $\betaQ_+\betaQ_- = \oneQ$, the claim follows.
\end{proof}

Next we compute $F(z)$ for some special values of $z$:

\begin{lemma}
	\label{prop:central_element_of_natural_transformation_SF}
	We have
	\begin{align}
		F\big(\phiQ{X_1^\pm}\big) 
		&=
		\pm 2^{-N} \phiQ{P_0^+}
		+ \tfrac{1}{2} (\phiQ{X_1^+} + \phiQ{X_1^-}) \ ,
		\notag \\
		F\big(\phiQ{P_0^+}\big) 
		&= 2^{N - 1} (\phiQ{X_1^+} - \phiQ{X_1^-}) \ ,
		\notag
		\\
		F\big(\eQ_0 \prod_{j=1}^N (\genF^+_j \genF^-_j)^{s_j}\big)
		&=      \beta^2 2^{N - 2 |s|} (-1)^{|s|}
		\eQ_0 \prod_{k=1}^N (\genF^+_k \genF^-_k)^{1 - s_k}
		\qquad \text{for} ~~ s \in \field[Z]_2^N \ .
		\notag
	\end{align}
\end{lemma}

\begin{proof}
	Recall from
	\Cref{sec:qHopf:sympf_ferm_elements} that the monodromy matrix of $\sympFerm$ can be written
	as $\monodromy = \sum_{I \in X} g_I \tensor f_I$, where the index set $X$ is
	$\field[Z]_2 \times \field[Z]_2 \times \field[Z]_2^N \times \field[Z]_2^N$, and for
	$I = (a,b,c,d) \in X$, $f_I$ and $g_I$ are as in
	\eqref{eq:symp_ferm_monodromy_basis}.
	Substituting \eqref{eq:Lyu-int-SF-with-nu=1} and \eqref{eq:phi-SF-explicit_X}, one computes
	\begin{align}
		&F(\phiQ{X_1^\pm}) 
		= 
		\pm \beta^6 \, 2^{2} \sum_{I \in X} f_I \times
		\langle \TOP^* | \eQ_1^\pm g_I \rangle
		\notag \\ 
		&\overset{(a)}=
		\pm \beta^6 2^{2}
		\sum_{b,c}
		(- \beta^2)^{b} 2^{2 |c|} (-1)^{b |c|}
		\cdot 
		\genK \eQ_b \prod_{k=1}^N (\genF^-_k \genF^+_k)^{c_k}
		\times
		\langle \TOP^* | 
		\eQ_1^\pm
		\genK^b \prod_{k=1}^N (\genF^+_k \genF^-_k)^{c_k}
		\rangle 
		\notag \\ 
		&\overset{(b)}=
		\pm \beta^6 2
		\sum_{b,c}
		(- \beta^2)^{b} 2^{2 |c|} (-1)^{b |c|}
		\cdot 
		\genK \eQ_b \prod_{k=1}^N (\genF^-_k \genF^+_k)^{c_k}
		\notag \\ & \quad \times
		\left(
		\delta_{b, 0} \delta_{c, 1}
		\langle \TOP^* | \eQ_1 \TOP \rangle 
		\pm i
		\delta_{b, 1}
		\langle \TOP^* | 
		\eQ_1 
		\prod_{k=1}^N 
		\big(
		\delta_{c_k, 0} \oneQ
		- (2 \delta_{c_k, 0} + \delta_{c_k, 1}) \genF^+_k \genF^-_k
		\big)
		\rangle 
		\right)
		\notag \\ 
		&\overset{(c)}=
		\pm \beta^6
		\bigg(
		2^{2 N} \genK \eQ_0 \prod_{k=1}^N \genF^-_k \genF^+_k
		\mp i (-1)^N \beta^2 
		\cdot 
		\genK \eQ_1 
		\prod_{k=1}^N 
		\sum_{c_k = 0}^1
		(-4)^{c_k}
		\big( 2 \delta_{c_k, 0} + \delta_{c_k, 1} \big)
		(\genF^-_k \genF^+_k)^{c_k}
		\bigg)
		\notag \\ 
		&\overset{(d)}=
		\pm 
		\beta^6 (-1)^N 
		2^{2 N} 
		\genK \eQ_0 \TOP
		- i 2^N
		\genK \eQ_1 \prod_{k=1}^N \big( \oneQ - 2 \genF^+_k \genF^-_k \big) \ .
		\notag
	\end{align}
	In step $(a)$ we used that $\eQ_1^\pm \eQ_a = \delta_{a,1} \eQ_1^\pm$, that
	the evaluation against $\TOP^*$ can be non-zero only if each $\genF^+_k$ is paired with $\genF^-_k$, resulting in $d=c$, 
	and that $\eQ_1 \genF^+_k \omega_- \genF^-_k \omega_- = -\eQ_1 \genF^+_k\genF^-_k$, giving an additional factor $(-1)^{|c|}$.
	In step $(b)$ we substituted the explicit form of $\eQ_1^\pm$ in \eqref{eq:e1+-} and used
	$\eQ_1(\oneQ - 2 \genF^+_k \genF^-_k) (\genF^+_k \genF^-_k)^{c_k} = 
	\eQ_1 \big(\delta_{c_k, 0} \oneQ
	- (2 \delta_{c_k, 0} + \delta_{c_k, 1}) \genF^+_k \genF^-_k\big)$.
	For step (c) note that only the product with all $\genF^+_k \genF^-_k$ can be non-zero against $\TOP^*$, resulting in a factor 
	$\prod_{k=1}^N \big( 2 \delta_{c_k, 0} + \delta_{c_k, 1} \big)$.
	In addition, we replaced $\sum_c \prod_{k=1}^N \leadsto \prod_{k=1}^N \sum_{c_k=0}^1$. The extra factor of $(-1)^N$ in both summands in step $(d)$ arises from reordering the $\genF$s.
	
	To write this in our preferred basis of the centre, note from \eqref{eq:e1+-} and \eqref{eq:phi-SF-explicit_P0} that
	\begin{align}
		\phiQ{P_0^+} &= 2^{3N}
		\beta^2 \genK \eQ_0 \TOP \ ,
		\nonumber\\
		\phiQ{X_1^+} + \phiQ{X_1^-}
		&= 2^{N+1} \big( \eQ_1^+ - \eQ_1^- \big)
		= -2^{N+1} i \eQ_1 \genK \prod_{j=1}^N (\oneQ - 2 \genF^+_j \genF^-_j)
		\ ,
		\nonumber\\
		\phiQ{X_1^+} - \phiQ{X_1^-}
		&= 2^{N+1} \big( \eQ_1^+ + \eQ_1^- \big)= 2^{N+1} \eQ_1
		\notag
		\ .
	\end{align}
	From this we read off
	\begin{align}
		F( \phiQ{X_1^\pm} )
		=
		\pm 2^{-N} \phiQ{P_0^+}
		+ \tfrac{1}{2} (\phiQ{X_1^+} + \phiQ{X_1^-})
		\ ,
		\notag
	\end{align}
	as claimed.
	For the remaining two equalities, first note that for $z \in Z(\sympFerm)$ in the even sector, i.e.\ $z = \eQ_0 z$, we have
	\begin{align}
		F(z)
		&= 
		\beta^6 \, 2^{-N+1} \sum_{I \in X} f_I \,
		\langle \TOP^* | z\, g_I \rangle
		\notag \\ 
		&=
		\beta^6 \, 2^{-N+1} \sum_{b, c}
		2^{2 |c|}
		(-1)^{b |c|}
		\langle 
		\TOP^*
		| z 
		\genK^b \prod_{k=1}^N (\genF^+_k \genF^-_k)^{c_k}
		\rangle
		\eQ_b \prod_{k=1}^N (\genF^-_k \genF^+_k)^{c_k}
		\notag
		\ .
	\end{align}
	This immediately yields
	\begin{align}
		F( \phiQ{P_0^+} )
		= 2^{2N} \eQ_1
		= 2^{N - 1} (\phiQ{X_1^+} - \phiQ{X_1^-})
		\notag
	\end{align}
	and
	\begin{align}
		F\big(\eQ_0 \prod_{j=1}^N (\genF^+_j \genF^-_j)^{s_j} \big)
		&= \beta^6 \, 2^{-N} 
		\sum_{b, c}
		2^{2 |c|}
		\langle 
		\TOP^* 
		| 
		\prod_{j=1}^N (\genF^+_j \genF^-_j)^{s_j + c_j}
		\rangle
		\eQ_0 \prod_{k=1}^N (\genF^-_k \genF^+_k)^{c_k}
		\notag
		\notag \\ 
		&=
		\beta^6 
		2^{N - 2 |s|} 
		(-1)^{N - |s|}
		\eQ_0 
		\prod_{k=1}^N (\genF^+_k \genF^-_k)^{1 - s_k} ~.
		\notag
	\end{align}
\end{proof}

With these preparations, we can compute $\zeta_*$, $\zeta_\circ$ from \Cref{prop:lens-quasi-Hopf-general}:

\begin{lemma}
	\label{prop:central_element_of_LPQ}
	We have
	\begin{align*}
		\zeta_* &= \langle \intLyu | \widehat{f(a)} \rangle \cdot \oneQ
		= \beta^{6n} p^N \cdot \oneQ
		\\
		\zeta_\circ &= F(\widehat{f(a)})
		= 
		2^{-2N}
		\big(
		c_{n+1}^0 \phiQ{P_0^+}
		+ c_{n+1}^+ \phiQ{X_1^+}
		+ c_{n+1}^- \phiQ{X_1^-}
		\big)
		\\ 
		& 
		\hspace{6em} +
		\beta^{6n}
		p^N
		\sum_{s \in \field[Z]_2^N}
		\left(- \tfrac{2 q}{p} \right)^{|s|}
		\eQ_0 
		\prod_{j=1}^N 
		(\genF^+_j \genF^-_j)^{s_j}
		\ .
	\end{align*}
	In the expression for $\zeta_\circ$ we have extended the recursive definition of $c_i^{0,\pm}$ in \eqref{eq:coefficients_ST_projective center} to $n+1$ by setting $a_{n+1} = 0$.
\end{lemma}

\begin{proof}
	The expression for $\zeta_*$ was already computed in \Cref{sec:cat-tr-Lyu-Lens}. To get the result for $\zeta_\circ$, recall from \Cref{prop:central_element_SF_computed} the decomposition of $\widehat{f(a)}$ into summands in $Z(\sympFerm) = Z_P \oplus Z_{\mathbf{\Lambda}}$.
	Using the formulas from the preceding
	\namecref{prop:central_element_of_natural_transformation_SF}, we get
	\begin{align}
		&F\big( \widehat{f(a)}\big|_{Z_{\mathbf{\Lambda}}} \big) 
		\overset{\text{Prop.\,\ref{prop:central_element_SF_computed}}}=
		2^N \beta^{6n + 2}
		\sum_{s \in \field[Z]_2^N}
		p^{|s|}
		\left( \frac{q}{2} \right)^{N - |s|}
		F\big(\eQ_0 \prod_{j=1}^N \big( \genF^+_j \genF^-_j \big)^{s_j} \big)
		\notag \\ & \quad
		\overset{\text{Lem.\,\ref{prop:central_element_of_natural_transformation_SF}}}=
		2^N \beta^{6n + 4}
		\eQ_0 
		\prod_{j=1}^N 
		\sum_{s_j = 0}^1
		p^{s_j}
		\left( \frac{q}{2} \right)^{1 - s_j}
		2^{1 - 2 s_j} (-1)^{s_j}
		\big( \genF^+_j \genF^-_j \big)^{1 - s_j}
		\notag \\ & \quad
		\overset{\phantom{\text{Lem.\,\ref{prop:central_element_of_natural_transformation_SF}}}}=
		\hspace{-.5em}
		2^N \beta^{6n + 4}
		\eQ_0 
		\prod_{j=1}^N 
		\big(
		q \genF^+_j \genF^-_j
		- p 2^{-1} 
		\oneQ
		\big)
		\notag \\ & \quad
		\overset{\phantom{\text{Lem.\,\ref{prop:central_element_of_natural_transformation_SF}}}}=
		\hspace{-.5em}
		\beta^{6n}
		p^N
		\sum_{s \in \field[Z]_2^N}
		\left(- \frac{2 q}{p} \right)^{|s|}
		\eQ_0 
		\prod_{j=1}^N 
		(\genF^+_j \genF^-_j)^{s_j}
		\ .
		\notag
	\end{align}
	
	On the projective centre $Z_P$, we find
	\begin{align}
		F\big( \widehat{f(a)} \big|_{Z_{P}} \big)
		&\overset{\text{Prop.\,\ref{prop:central_element_SF_computed}}}=
		2^{-2N}
		\big(
		c_n^0 F(\phiQ{P_0^+}) + c_n^+ F(\phiQ{X_1^+}) + c_n^- F(\phiQ{X_1^-})
		\big)
		\nonumber\\      
		&\overset{\text{Lem.\,\ref{prop:central_element_of_natural_transformation_SF}}}=
		2^{-2N}
		\bigg(
		2^{-N} \big( c_n^+ - c_n^- \big)
		\phiQ{P_0^+}
		+
		\big(
		c_n^0 2^{N-1}
		+ \tfrac{1}{2} (c_n^+ + c_n^-)
		\big)
		\phiQ{X_1^+}
		\notag \\ & \hspace{7em} 
		+
		\big(
		- c_n^0 2^{N-1}
		+ \tfrac{1}{2} (c_n^+ + c_n^-)
		\big)
		\phiQ{X_1^-}
		\bigg)
		\ .
		\notag
	\end{align}
	The coefficients match the recursive definition of $c^0_{n+1}$, $c_{n+1}^\pm$ from
	\eqref{eq:coefficients_ST_projective center} under the convention that $a_{n+1}=0$. Adding the two expressions gives the claimed result for $\zeta_\circ$.
\end{proof}

This completes the computation of $\zeta_*$ and $\zeta_\circ$ and we now turn to the pairings $(-,-)_P$ from \eqref{eq:Z(H)-pairing-via-trace_general}.
Let $\modTr$ be the modified trace associated with the symmetrised cointegral
in~\eqref{eq:symp_ferm_symmetrized_cointegrals}, which was explicitly given by
\begin{align*}
	\symRightCoint(\genK^m \TOP) = \normalizationModTrSF \deltaOdd{m} (\beta^2 +
	i^m) \quad \text{where} ~~ \normalizationModTrSF = (-1)^N 2^{-N} \ .
\end{align*}
Fix an isotypic decomposition $\sympFerm = \bigoplus_U n_U \projCover{U}$ with
idempotents $e_U$ and where the sum runs over distinct irreducibles $U$.
We abbreviate 
    $\{z, z'\}_U:=(z, z')_{n_U P_U}$.
In terms of the symmetrised cointegral, the pairing is given by
\begin{align}
    \{z, z'\}_U
	= \langle \symRightCoint | e_U z z' \rangle \ .
	\label{eq:Z(H)-pairing-via-trace_proj}
\end{align}
Explicitly, $n_{X_0^\pm}=1$ and $n_{X_1^\pm}=2^N$, so that
    \begin{align}\label{eq:curly-noncurly}
    \{z, z'\}_{X_0^\pm} = (z, z')_{P_0^\pm}
    ~~ , \qquad
    \{z, z'\}_{X_1^\pm} = 2^N (z, z')_{X_1^\pm} \ .
    \end{align}
The next lemma gives the value of the pairing on the basis \eqref{eq:basis_center_SF} of $Z(\sympFerm)$.

\begin{lemma}
	\label{prop:bilinear_symmetric_forms_on_center}
	\begin{enumerate}
		\item
		For $U = X_0^\pm$ we have 
		\begin{align}
			\left\{
			\eQ_0, \phiQ{P_0^+}
			\right\}_{X_0^\pm}
			= 2^{2N - 1}
			\notag
		\end{align}
		and
		\begin{align}
			\left\{
			\eQ_0 \prod_{j=1}^N (\genF^+_j)^{r_j} (\genF^-_j)^{s_j}
			, \eQ_0 \prod_{j=1}^N (\genF^+_j)^{t_j} (\genF^-_j)^{u_j}
			\right\}_{X_0^\pm}
			&=
			\pm
			\varepsilon \delta_{r + t, 1} \delta_{s + u, 1}
			2^{-N - 1} \beta^{-2}
			\ ,
			\notag
		\end{align}
		where $\varepsilon = \varepsilon(r,s,t,u)$ is a sign, which
		is $+1$ if $r = s$ or $t = u$. 
		For all other combinations of basis elements in \eqref{eq:basis_center_SF}, the pairing is zero.

		\item
		For $U = X_1^\pm$, the only non-zero values of the pairing are
		\begin{align}
			\left\{
			\phiQ{X_1^\pm}, \phiQ{X_1^\pm}
			\right\}_{X_1^\pm}
			& = \pm 2^{2N+1}
			\notag
		\end{align}
	\end{enumerate}
\end{lemma}

\begin{proof}
	Note first of all that the map 
 $z \mapsto \{z, -\}_U$ 
 is zero on the odd part of $z$ if
	$\projCover{U}$ is from the even sector.
  The same is true when swapping even and odd.
	
	\smallskip
	
	\noindent
	(1)
	For $U = X_0^\pm$ we have $e_U = \eQ_0^\pm$ as in \eqref{eq:e0+-}.
	From \eqref{eq:phi-SF-explicit_P0} we see that $\phiQ{P_0^+} = 2^{3N} \beta^2 \genK \eQ_0 \TOP$ already contains $\TOP$ as a factor and so the only non-zero pairing can be with the basis element $\eQ_0$. We compute
	\begin{align}
		\left\{
		\eQ_0, \phiQ{P_0^+}
		\right\}_{X_0^\pm}
		&=
		2^{3N} \beta^2 
		\langle
		\symRightCoint | \eQ_0^\pm \genK \cdot \TOP
		\rangle
		= 2^{2N - 1}
		\ .
		\notag
	\end{align}
	Let now $r, s, t, u \in \field[Z]_2^N$ with $|r| + |s|$ and $|t| + |u|$ even.
	We have
	\begin{align}
		\bigg(
		\eQ_0 \prod_{j=1}^N (\genF^+_j)^{r_j} (\genF^-_j)^{s_j}
		\bigg)
		\cdot
		\bigg(
		\eQ_0 \prod_{j=1}^N (\genF^+_j)^{t_j} (\genF^-_j)^{u_j}
		\bigg)
		=
		\varepsilon \eQ_0 \prod_{j=1}^N (\genF^+_j)^{r_j + t_j} (\genF^-_j)^{s_j + u_j}
		\ ,
		\notag
	\end{align}
	where $\varepsilon$ is a sign depending on $r, s, t, u$, which we do not compute
	explicitly.
	If $r = s$ or $t = u$, this sign is $+1$ as a pair $\genF^+_j \genF^-_j$ commutes with everything in the even sector.
	This shows that
	\begin{align}
		\left\{
		\eQ_0 \prod_{j=1}^N (\genF^+_j)^{r_j} (\genF^-_j)^{s_j}
		, \eQ_0 \prod_{j=1}^N (\genF^+_j)^{t_j} (\genF^-_j)^{u_j}
		\right\}_{X_0^\pm}
		&=
		\varepsilon \delta_{r + t, 1} \delta_{s + u, 1}
		\langle
		\symRightCoint | \eQ_0^\pm \TOP
		\rangle
		\notag \\
		&\overset{\eqref{eq:hat-lam-e0+-top}}=
		\pm
		\varepsilon \delta_{r + t, 1} \delta_{s + u, 1}
		2^{-N - 1} \beta^{-2}
		\ .
		\notag
	\end{align}
	
	\noindent
	(2)
	For the odd sector, recall from \eqref{eq:phi-SF-explicit_X} that $\phiQ{X_1^\pm} = \pm 2^{N + 1} \eQ_1^\pm$, and since
	$\eQ_1^\pm$ are orthogonal, we only need to compute:
	\begin{align}
		\left\{
		\phiQ{X_1^\pm}, \phiQ{X_1^\pm}
		\right\}_{X_1^\pm}
		&=
		2^{2N+2}
		\langle
		\symRightCoint | \eQ_1^\pm
		\rangle
		\overset{\eqref{eq:sym-coint-on-e1+-}}=
		\pm
		2^{2N+1}
		\
	\end{align}
	as claimed.
\end{proof}

Now we are ready to provide the invariants for all  natural endomorphisms of the identity functor,
acting on the non-contractible loop in $L(p, q)$ as given in \eqref{eq:Lens-circ-def}. 

\begin{theorem}
	\label{prop:invariants_lens_spaces_symp_ferm_extended_version}
	We have, for the basis elements in \eqref{eq:basis_center_SF} and $t \in \field[Z]_2^N$,
	\begin{align*}
		\invDGGPR(\mathfrak{L}_\circ(\eQ_0 \prod_{j=1}^N (\genF^+_j \genF^-_j)^{t_j},
		P_0^\pm))
		&=
		\begin{cases}
			\frac{1}{2} c_{n+1}^0 \beta^{2n} \pm
			\frac{1}{2} \beta^{2} q^N
			, & t = 0,
			\\
			\pm (-1)^{|t|} \beta^{2} 
			2^{- |t| - 1} q^{N - |t|} p^{|t|}
			, & |t| > 0
		\end{cases}
	\\
		\invDGGPR(\mathfrak{L}_\circ(\phiQ{P_0^+}, P_0^\pm))
		&= 2^{2N-1} p^N
\\
		\invDGGPR(\mathfrak{L}_\circ(\phiQ{X_1^\pm}, X_1^\pm))
  &= \pm 2^{1-N} \beta^{2 n} c^\pm_{n+1}
		\ .
	\end{align*}
\end{theorem}

Note that combining the first formula in the theorem above for $t=0$ with \Cref{rem:form-of-c0} results in the invariant $\invDGGPR(\mathfrak{L}_\circ(\id, P_0^\pm))$ as stated in \Cref{prop:invariants_lens_spaces_symp_ferm}.

\begin{proof}
Recall from \Cref{prop:lens-quasi-Hopf-general} that
$\invDGGPR(\mathfrak{L}_{\circ}(\alpha, P)) = \anomaly^{-\sigma} \DD^{-1 -n}
(\widehat\alpha,\zeta_\circ)_P$, 
with
	$\sigma=n$ by \eqref{eq:Lpq-surgery-link-signature} and
$\zeta_\circ$ is given in \Cref{prop:central_element_of_LPQ}.
	From \Cref{sec:SF-int-coint} we have $\DD = 1$ and $\anomaly = \beta^{-2}$ so the surgery coefficient of $\mathfrak{L}_\circ$ is $\beta^{2\sigma}$.
	Let $t \in \field[Z]_2^N$ with $|t| \geq 1$.
	Using \Cref{prop:bilinear_symmetric_forms_on_center}, we compute
	\begin{align}
		&
		\big\{
		\eQ_0 \prod_{j=1}^{N} (\genF^+_j \genF^-_j)^{t_j}, \zeta_\circ \big\}_{X_0^\pm}
		\notag \\ & \quad
		=
		\beta^{6n} p^N
		\sum_{s \in \field[Z]_2^N}
		\left(- \tfrac{2 q}{p} \right)^{|s|}
		\big\{
		\eQ_0 \prod_{j=1}^N (\genF^+_j \genF^-_j)^{t_j}
		\,,\,
		\eQ_0 \prod_{j=1}^N (\genF^+_j \genF^-_j)^{s_j}
		\big\}_{X_0^\pm}
		\notag \\ & \quad
		=
		\pm 
		\beta^{6n + 6} p^N
		2^{-N-1} 
		\sum_{s \in \field[Z]_2^N}
		\left(- \tfrac{2 q}{p} \right)^{|s|}
		\delta_{s,1 - t} 
		\notag \\ & \quad
		=
		\pm 
		\beta^{6n + 6} p^N
		2^{-N-1} 
		\left(- \tfrac{2 q}{p} \right)^{N - |t|}
		\,=\,
		\pm 
		(-1)^{|t|}
		\beta^{6n + 2} 
		2^{- |t| - 1}
		q^{N - |t|}
		p^{|t|}
		\ .
		\notag
	\end{align}
	If $|t| = 0$, then 
	\begin{align}
		\big\{
		\eQ_0,\zeta_\circ
		\big\}_{X_0^\pm}
		&=
		2^{-2N} c_{n+1}^0 
		\left\{
		\eQ_0,
		\phiQ{P_0^+}
		\right\}_{X_0^\pm}
		+
		\beta^{6n}
		(-2q)^N
		\left\{
		\eQ_0,
		\eQ_0 \TOP
		\right\}_{X_0^\pm}
		\notag \\ 
		&=
		2^{-2N} c_{n+1}^0 
		2^{2N - 1}
		\pm
		\beta^{6n}
		(-2q)^N
		2^{-N - 1} 
		\beta^2 (-1)^N
		\notag \\ 
		&=
		\tfrac{1}{2}
		c_{n+1}^0 
		\pm
		\tfrac{1}{2}
		\beta^{6n+2}
		q^N
		\ .
		\notag
	\end{align}
Finally, we have
	\begin{align*}
		\big\{
		\phiQ{P_0^\pm},\zeta_\circ
		\big\}_{X_0^\pm}
		&=
		\beta^{6n}
		p^N
		\sum_{s \in \field[Z]_2^N}
		\left(- \tfrac{2 q}{p} \right)^{|s|}
		\big(
		\phiQ{P_0^\pm},
		\eQ_0 
		\prod_{j=1}^N 
		(\genF^+_j \genF^-_j)^{s_j}
		\big\}_{X_0^\pm}
		\\
		&=
		\beta^{6n}
		p^N
		\big\{
		\phiQ{P_0^\pm},
		\eQ_0 
		\big\}_{X_0^\pm}
		\,=\, 2^{2N-1} \beta^{6n} p^N
		\ , \text{ and }
		\\
		\big\{
		\phiQ{X_1^\pm},\zeta_\circ
		\big\}_{X_1^\pm}
		&=
		\pm 2 c_{n+1}^\pm \ .
	\end{align*}
Combining with \eqref{eq:curly-noncurly} to translate $\{-,-\}_U$ into $(-,-)_{P_U}$ results in the statement of the theorem.
\end{proof}

\begin{remark}
	Let us discuss a quick sanity check of the formulas above.
	In the category of modules over a factorisable (or, more generally, unimodular)
	ribbon quasi-Hopf algebra $H$, there is a natural transformation $\alpha \colon
	\id_{\cat} \To \id_{\cat}$ 
 such that $\alpha_{P_\tensUnit}$ is non-zero and factors through $\tensUnit$, cf.\ \Cref{rem:L-vs-Lyu}.
 Indeed, if $c$ is an integral in $H$, then it is central by unimodularity, and thus
	acting by it yields a natural transformation.
	In the notation of \Cref{sec:Examples:lens_spaces}, we must then necessarily have
	$ \invDGGPR (\mathfrak{L}_*(c, \projCover{\tensUnit})) = \invDGGPR
	(\mathfrak{L}_{\circ}(c, \projCover{\tensUnit}))$.
Indeed, this can be checked directly.
	Let $c$ be the integral of $\sympFerm$ corresponding to $\cointLyu$ with
	normalisation such that $\cointLyu \circ \intLyu = \id_{\tensUnit}$.
	By \cite[Sec.~3.5]{FGR2}, it is explicitly given by 
	\begin{align*}
		c = 2^{N+1} \beta^2 \eQ_0^+ \TOP 
		\overset{\eqref{eq:phi-SF-explicit_X}}= \phiQ{\tensUnit} 
		\overset{\eqref{eq:phi1-in-basis}}= 2^{-2N} \phiQ{P_0^+} + 2^N \beta^2 \eQ_0 \TOP
		\ .
	\end{align*}
	We saw in the proof of \Cref{lem:tHopf-normal} that $\modTr_{P_0^+}(c) = 1$.
	On the one hand, combining \eqref{eq:lens-L-vs-Lyu} and \eqref{eq:Lens-Lyu-SF} gives
	\begin{align*}
		\invDGGPR (\mathfrak{L}_*(c, P_0^+))
		= p^N
		\ .
	\end{align*}
	On the other hand, using \Cref{prop:invariants_lens_spaces_symp_ferm_extended_version} and the expansion of $c$ in the basis of the centre we get
	\begin{align*}
		\invDGGPR (\mathfrak{L}_{\circ}(c, P_0^+))
		&=
		2^{-2N} 
		\invDGGPR (\mathfrak{L}_{\circ}(\phiQ{P_0^+}, P_0^+))
		+ 2^N \beta^2 
		\invDGGPR (\mathfrak{L}_{\circ}(\eQ_0 \TOP, P_0^+))
		\\
		&=
		\tfrac12 p^N
+ \tfrac12 \beta^2 (-1)^N \beta^2 p^{N}
= p^N
		\ ,
	\end{align*}
	as expected.
\end{remark}

\subsection{Pullback trace}

\label{sec:proof_pullback_lens_spaces}
Finally, we consider $H = \sympFerm(2,\beta)$ and $A$ as in \Cref{sec:symp_ferm_pullback}, and treat the
invariants for lens spaces with one loop coloured by the object $P_\parMat$ from the pullback
ideal (cf.~\Cref{prop:lift_of_proj_cover_tU_sympFerm}).

We start by computing the values of 
$\pullbackTrace[{\modTr^A}]_{P_\parMat}(z)$ for $z$ from the
11-dimensional centre of $H$
    (recall \eqref{eq:SF-centre-sum}),
and where $\parMat = \begin{psmallmatrix} a^- & a^+ \\ b^- & b^+ \end{psmallmatrix}$.

\begin{lemma}
	\label{prop:pullback_trace_on_central_elements}
	The only non-zero values of $\pullbackTrace[{\modTr^A}]_{P_\parMat}(z)$ for $z$ an element from the
	canonical basis \eqref{eq:basis_center_SF} of $Z(H)$ are
	\begin{align*}
		\begin{tabular}{c||c|c|c}
			$z \in Z(H)$
			&
			$\eQ_0 \genF_1^+ \genF_1^-$
			&
			$\eQ_0 \genF_2^+ \genF_2^-$
			&
			$\eQ_0 \genF_1^\gamma \genF_2^\varepsilon$
			\\
			\hline
			$\pullbackTrace[{\modTr^A}]_{P_\parMat}(z)$
			& $1$ & $\det\parMat$ & $\delta_{\gamma, +} a^\varepsilon -
			\delta_{\gamma, -} b^\varepsilon$
		\end{tabular}
		\ ,
	\end{align*}
	where $\gamma, \varepsilon \in \{+, -\}$.
\end{lemma}
\begin{proof}
	Using \Cref{prop:symCoint_for_subqHA}, we compute
	\begin{align*}
		\widehat\coint{}^{A,r}(\eQ_0^+ \genF_1^- \genF_1^+)
		= \frac{1}{4} \sum_{m=0}^3 \widehat\coint{}^{A,r}(\genK^m \genF_1^- \genF_1^+)
		= \frac{1}{4} (1 - i \beta^2)
		\ ,
	\end{align*}
	so that in our chosen normalisation \eqref{eq:pullback_trace_choice}, the pullback
	trace at $P_\parMat$ equals 1 on $\eQ_0^+ \genF^+_1 \genF^-_1$.
	From \Cref{prop:symp_ferm_f2_acting_on_PM}, we know that $\genF_2^+ \genF_2^-$ acts
	on $P_\parMat$ as $\det\parMat \cdot \genF_1^+ \genF_1^-$.
	Similarly, from \eqref{eq:lifted_action} one finds for $\gamma, \varepsilon \in \{+, -\}$ and $p \in P_\parMat$ that
	$
	\genF_1^\gamma \genF_2^\varepsilon\,p
	=
	(\delta_{\gamma, +} a^\varepsilon - \delta_{\gamma, -} b^\varepsilon) \cdot
	\genF_1^+ \genF_1^- \, p
	$.

 Likewise, $\genF_1^+ \genF_1^- \genF_2^+ \genF_2^-$ is seen to act as $0$, whence the 
	trace vanishes for $\phiQ{P_0^+}$.
	Finally, since $P_\parMat$ is purely even, $\phiQ{X_1^\pm}$ act as zero as well.
\end{proof}

As above, we have a symmetric bilinear pairing $(z, z')_{P_\parMat} := \pullbackTrace[{\modTr^A}]_{P_\parMat}
(zz')$ for $z, z' \in Z(H)$.
Then \Cref{prop:pullback_trace_on_central_elements} shows that the only non-zero values
are
\begin{align*}
	(\eQ_0, \eQ_0 \genF_1^+ \genF_1^-)_{P_\parMat}
	= 1,
	\quad
	(\eQ_0, \eQ_0 \genF_2^+ \genF_2^-)_{P_\parMat}
	= \det\parMat,
\end{align*}
and
\begin{align*}
	(\eQ_0, \eQ_0 \genF_1^\gamma \genF_2^\varepsilon)_{P_\parMat}
	= \delta_{\gamma, +} a^\varepsilon - \delta_{\gamma, -} b^\varepsilon
	\ .
\end{align*}

\begin{theorem}\label{thm:lens-pullback-trace-SF}
	The invariants
 $\invDGGPRtrace{\pullbackTrace[{\modTr^A}]} (\mathfrak{L}_{\circ}(z, P_\parMat))$
 for basis elements 
        $z$
 of $Z(H)$ as in \eqref{eq:basis_center_SF} are:
	\begin{align*}
		\begin{tabular}{c|c}
			$z$ & $\invDGGPRtrace{\pullbackTrace[{\modTr^A}]} (\mathfrak{L}_{\circ}(z, P_\parMat))$ 
			\\
			\hline
			$\phiQ{P_0^+}$, $\phiQ{X_1^\pm}$ & $0$ \\
			$\eQ_0$ & $    -
			2
			p q
			(1 + \det\parMat)
			$\\
			$\eQ_0 \genF_j^+ \genF_j^-$ & $  
			p^2
			(\delta_{j, 1} + \delta_{j,2} \det\parMat)
			$\\
			$\eQ_0 \genF_1^\gamma \genF_2^\varepsilon $ & $  
			p^2
			(\delta_{\gamma, +} a^\varepsilon - \delta_{\gamma, -} 
			b^\varepsilon)
			$\\
			$\eQ_0 \TOP$ & $0$
		\end{tabular}
	\end{align*}
\end{theorem}

\begin{proof}
	The surgery coefficient is $\beta^{2\sigma}$ as in the proof of \Cref{prop:invariants_lens_spaces_symp_ferm_extended_version}.
	By \Cref{prop:lens-quasi-Hopf-general} we need to compute $(z,\zeta_\circ)_{P_\parMat}$. The zero entries in the table are immediate from \Cref{prop:pullback_trace_on_central_elements}.
	Using \Cref{prop:central_element_of_LPQ}, we first compute
	\begin{align}
		\big(
		\eQ_0, \zeta_\circ
		\big)_{P_\parMat}
		&=
		\beta^{6n}
		p^2
		\sum_{s \in \field[Z]_2^2}
		\left(- \frac{2 q}{p} \right)^{|s|}
		\big(
		\eQ_0,
		\eQ_0 \prod_{j=1}^2 (\genF^+_j \genF^-_j)^{s_j}
		\big)_{P_\parMat}
		\notag
		\\
		& =
		-
		\beta^{6n}
		2 p q
		\big(
		\eQ_0,
		\eQ_0 
		\genF^+_1 \genF^-_1
		+
		\eQ_0 
		\genF^+_2 \genF^-_2
		\big)_{P_\parMat}
		\notag
		\\
		& =
		-
		2
		\beta^{6n}
		p q
		(1 + \det\parMat)
		\notag
		\ .
	\end{align}
	Since $\sigma = n$, $\beta^{2 \sigma + 6n} = 1.$
	Similarly, we have
	\begin{align}
		\big(
		\eQ_0 \genF_j^+ \genF_j^-, 
		\zeta_\circ
		\big)_{P_\parMat}
		&=
		\beta^{6n}
		p^2
		\sum_{s \in \field[Z]_2^2}
		\left(- \frac{2 q}{p} \right)^{|s|}
		\big(
		\eQ_0 \genF_j^+ \genF_j^-, 
		\eQ_0 \prod_{j=1}^2 (\genF^+_j \genF^-_j)^{s_j}
		\big)_{P_\parMat}
		\notag
		\\
		& =
		\beta^{6n}
		p^2
		\big(
		\eQ_0 \genF_j^+ \genF_j^-,
		\eQ_0 
		\big)_{P_\parMat}
		\notag
		\\
		& =
		\beta^{6n}
		p^2
		(\delta_{j, 1} + \delta_{j,2} \det\parMat)
		\notag
		\ .
	\end{align}
	A completely analogous computation shows
	\begin{align}
		\big(
		\eQ_0 \genF_1^\gamma \genF_2^\varepsilon,
		\zeta_\circ
		\big)_{P_\parMat}
		& =
		\beta^{6n}
		p^2
		\big(
		\eQ_0 \genF_1^\gamma \genF_2^\varepsilon,
		\eQ_0 
		\big)_{P_\parMat}
		=
		\beta^{6n}
		p^2
		(\delta_{\gamma, +} a^\varepsilon - \delta_{\gamma, -} 
		b^\varepsilon)
		\notag
		\ .
	\end{align}
\end{proof}

\Cref{thm:lens-pullback-trace-SF} shows the result stated in \Cref{thm:lens-for-intermediate}.
In particular, 
 the invariants in the above theorem manage to distinguish
all $P_\parMat$, by evaluating on all $\eQ_0 \genF_1^\gamma \genF_2^\varepsilon$. 
Explicitly, the choices for $\alpha_{j,l}$ in \Cref{thm:lens-for-intermediate} are
\begin{align*}
\alpha_{1,1} = \eQ_0 \genF_1^+ \genF_2^- ~,~~
\alpha_{1,2} = \eQ_0 \genF_1^+ \genF_2^+ ~,~~
\alpha_{2,1} = -\eQ_0 \genF_1^- \genF_2^- ~,~~
\alpha_{2,2} = -\eQ_0 \genF_1^- \genF_2^+ \ .
\end{align*}

\appendix

\section{More on continued fractions}
\label{sec:continued_fractions}
Consider the $(n \times n)$-matrix $M_a$ with entries
\begin{align}
	(M_a)_{ij} =
	\begin{cases}
		a_i & i = j \ , \\
		1 & (i < n ~\land~ j = i + 1) ~\lor~ (i > 1 ~\land~ j = i - 1) \ .
	\end{cases}
	\notag
\end{align}
This is precisely the linking matrix of the surgery link of $\mathfrak{L}(p, q)$ as given in \eqref{eq:Lpq-surgery-link}, see e.g.~\cite[Sec.\,5.3]{Lyu-inv-MCG}.
By definition, the signature $\sigma(\mathfrak{L}(p, q))$ of the surgery link is the signature of the linking matrix, i.e.\ the number of positive minus negative eigenvalues.

\begin{proposition}
	\label{prop:continued_fractions_main_statement}
	Let $a_1, \ldots, a_n \in \field[Z]$ be non-zero such that there are coprime positive
	integers $p_i, q_i$ with $\tfrac{p_i}{q_i} = [a_i; a_{i+1}, \ldots, a_n]$
    for $i = 1,\dots,n$.	
	Then we have:
	\begin{enumerate}
		\item $\det M_a = \prod_{j=1}^n [a_j; a_{j+1}, \ldots, a_n]$ .
		\item $M_a$ is positive definite, and so its signature is $n$.
		\item $p_1 = \prod_{j=1}^n [a_j; a_{j+1}, \ldots, a_n]
		= \prod_{j=1}^n [a_{n-j+1}; a_{n-j}, \ldots, a_1]$ .
	\end{enumerate}    
\end{proposition}

\begin{proof}
	\textit{Part 1:}
	We proceed by induction on $n$.
	For $n = 1$ this is obvious, and to see what is happening note that for $n = 2$ we
	have
	\begin{align*}
		[a_1; a_2] \cdot a_2 
		= (a_1 - \tfrac{1}{a_2}) a_2
		= a_1 a_2 - 1
		= \det \begin{psmallmatrix} a_1 & 1 \\ 1 & a_2 \end{psmallmatrix}
		\ .
	\end{align*}
Assume now that the statement holds up to $n-1$ for a given $n \ge 2$.
	Setting $d(a_1, \ldots, a_n)$ $=$ $\prod_{j=1}^n [a_j; a_{j+1}, \ldots, a_n]$, we note that
	$d(a_i, \ldots, a_n)$ $=$ $[a_i; a_{i+1}, \ldots, a_n]$ $d(a_{i+1}, \ldots, a_n)$.
    Using the Laplace expansion for the first row, we obtain
	\begin{align*}
		\det M_a
		&= a_1 \cdot \det M_{a \setminus (a_1)} - 1 \cdot 
		\det
		\begin{psmallmatrix}
			1 & 1 & 0 & \cdots 0 \\
			0 &  &  & \\
			\vdots & & M_{a \setminus (a_1, a_2)} & \\
			0 &  &  & 
		\end{psmallmatrix}
		\\
		&= a_1 \cdot \det M_{a \setminus (a_1)} - \det M_{a \setminus (a_1, a_2)} 
		\ ,
	\end{align*}
	where by $M_{a \setminus b}$ we mean the matrix associated to the list obtained by
	removing entries of $b$ from $a$.
	By hypothesis, we know the values of the two determinants on the right hand side, and so we obtain
	\begin{align*}
		\det M_a
		&= a_1 \cdot d(a_2, \ldots, a_n) - d(a_3, \ldots, a_n)
		\\
		&= d(a_2, \ldots, a_n) \cdot 
		\left( a_1 - \frac{1}{[a_2; \ldots, a_n]} \right)
		= d(a_2, \ldots, a_n) \cdot [a_1; a_2, \ldots, a_n]
		\ ,
	\end{align*}
	which equals $d(a_1, \ldots, a_n)$, as claimed.
	
	\medskip
	
	\noindent\textit{Part 2:} Positive definiteness now follows from the Sylvester criterion. 
	Namely, consider the lower right submatrices $N_d$, $d=1,\dots,n$ of $M$.
	That is, $(N_d)_{st} = (M_a)_{s+n-d,t+n-d}$ and 
	$s,t = 1,\dots,d$. 
	
	Using the above notation, we see that $N_d = M_{(a_{n-d+1},a_{n-d+2},\dots,a_n)}$.
	Thus by part 1 we get 
	$\det N_d = \prod_{j=n-d+1}^n  [a_j; a_{j+1}, \ldots, a_n] 
	= \prod_{j=n-d+1}^n \tfrac{p_j}{q_j}$,
	and all factors are positive by assumption.
	
	\medskip
	
	\noindent\textit{Part 3:}
	By part 1, the second 
	equality is simply the statement that the similar matrices 
	$M_{(a_1, \ldots, a_n)}$ and $M_{(a_n, \ldots, a_1)}$ have the same determinant.
	
	To see the first equality, note first that
	\begin{align}
		\prod_{j=1}^n [a_{j}; a_{j+1}, \ldots, a_n]
		= \prod_{j=1}^n \frac{p_{j}}{q_{j}}
= \frac{p_1}{q_n} \prod_{j=1}^{n-1} \frac{p_{j+1}}{q_{j}}
        \ .
		\notag
	\end{align}
	Since clearly $q_n = 1$, this equals $p_1$ if $q_j = p_{j+1}$ for all $1 \leq j < n$.
	Let us show that this holds.
	For $j < n$, we have
	\begin{align*}
		\frac{p_j}{q_j}
		= a_j - \frac{q_{j+1}}{p_{j+1}}
		= \frac{a_j p_{j+1} - q_{j+1}}{p_{j+1}}
		\ .
	\end{align*}
	Now, for any $c, d, n \in \field[Z]$ with $d | n$, we have $\operatorname{gcd}(n \pm
	c, d) \leq \operatorname{gcd}(c, d)$.\footnote{%
		If $x$ divides $\operatorname{gcd}(d, n \pm c)$, then since $d$ divides $n$ also 
		$x$ divides $n$, and hence $x$ divides $c$.
	}
	Applying this to our situation, we find $\operatorname{gcd}(a_j p_{j+1} - q_{j+1},
	p_{j+1}) \leq \operatorname{gcd}(q_{j+1}, p_{j+1}) = 1$, and so the above is in fact an
	equality of reduced fractions.
	Thus we even obtain equality of  numerators and denominators up to sign.
	But the sign vanishes in light of our positivity assumption; in particular, we obtain 
	$q_j = p_{j+1}$.	
\end{proof}

The proposition above in particular implies \eqref{eq:Lpq-surgery-link-signature} for the signature of the surgery link. The same expression (with the stronger assumption that all $a_i$ are at least 2) was obtained e.g.\ in \cite[Prop.\,6.24]{Blanchet:2014ova}.


\newcommand{\arxiv}[2]
    {[\href{http://arXiv.org/abs/#1}{ \color{orange}#1 [#2]}]}
\newcommand{\doi}[2]{\href{http://dx.doi.org/#1}{#2}}
\newcommand{\sameauthors}[1][2]{\underline{\hspace*{#1cm}},~}

\end{document}